\documentclass[a4paper]{amsart}

\usepackage{amsmath,amsfonts,amssymb,latexsym,epic,eepic}
\usepackage[overload]{empheq}
\usepackage[dvips]{graphicx,graphics,epsfig,color}
\usepackage{ifthen}
\usepackage[a4paper, margin=3cm]{geometry}
\usepackage{xcolor}
\usepackage{amsthm}
\usepackage{mathrsfs}
\usepackage{tikz,tkz-tab,pstricks,fp}
\usepackage{multirow}
\usepackage{array}
\usepackage{enumerate}
\usepackage{setspace}
\usepackage{enumitem}
\usepackage{pdfpages}
\usepackage{tcolorbox}
\usepackage[utf8]{inputenc}
\usepackage[T1]{fontenc}
\usepackage[english]{babel}
\usepackage{url}
\usepackage[colorlinks=true, pdfstartview=FitV, linkcolor=blue, citecolor=red, urlcolor=blue,unicode,psdextra]{hyperref}
\makeatletter
\def\Hy@Warning#1{}
\makeatother
\usepackage{todonotes}

\let\oldsection\section            % sauvegarde de la commande \section originale
\usepackage[explicit]{titlesec}    % chargement de titlesec
\let\section\oldsection            % restauration de \section d'amsart
\numberwithin{equation}{section}
\newtheorem{thrm}{Theorem}[section]
\newtheorem{lem}[thrm]{Lemma}

\newtheorem{prop}[thrm]{Proposition}
\newtheorem{defi}{Definition}[section]
\newtheorem{rmrk}{Remark}[section]
\newtheorem{corol}{Corollary}[section]

\def\xR{{\mathbb R}}

\def\xN{{\mathbb N}}
\def\xd{{\mathrm d}}

\newcommand{\R}{\mathbb{R}}

\begin{document}

\title[]{A quasi-linear Schrödinger equation for nucleons : existence, uniqueness and asymptotic behavior}

\author{Meriem Bahhi}
\address{Université Côte d'Azur, CNRS, Laboratoire J.A.Dieudonné, UMR 7351\\Parc Valrose, 28, av. Valrose, 06108 Nice, France\\E-mail: Meriem.BAHHI@univ-cotedazur.fr}

\author{Jonas Lampart}
\address{Université Bourgogne Europe, CNRS, Laboratoire Interdisciplinaire Carnot de Bourgogne, UMR 6303, 9 Av. Alain Savary, 21000 Dijon, France\\E-mail: jonas.lampart@u-bourgogne.fr}

\author{Simona Rota Nodari}
\address{Université Côte d'Azur, CNRS, Laboratoire J.A.Dieudonné, UMR 7351\\Parc Valrose, 28, av. Valrose, 06108 Nice, France\\Institut Universitaire de France\\
E-mail: simona.rotanodari@univ-cotedazur.fr}

\date{\today}

\thanks{This work was funded by the EIPHI Graduate School (ANR-17-EURE-0002) and the Bourgogne-Franche-Comté Region through the project EEQuaR}

\begin{abstract}
We study a quasi-linear Schrödinger-type equation arising in the description of particle dynamics within atomic nuclei.
Under suitable assumptions on the model parameters, we establish the existence, uniqueness, and non-degeneracy of positive radial solutions.
Moreover, we analyze how these solutions depend on the parameters.
\end{abstract}

\maketitle
\tableofcontents

\section{Introduction}

This work is concerned with the study of the following quasi-linear Schrödinger equation:
\begin{equation}
\label{SNL_Rescaled}
i\partial_t \Phi
=
-\nabla \cdot \left( \frac{\nabla \Phi}{1-|\Phi|^{2\alpha}} \right)
+ \alpha |\Phi|^{2\alpha-2}
\frac{|\nabla \Phi|^2}{(1-|\Phi|^{2\alpha})^2}\Phi
- |\Phi|^{2\alpha}\Phi.
\end{equation}
Here, $\Phi \in L^2(\mathbb{R}\times\mathbb{R}^d,\mathbb{C})$, with $d\geq 1$, denotes the quantum state of a nucleon, and $\alpha>0$ is a parameter characterizing the non-linearity.

Equation \eqref{SNL_Rescaled} arises in a particular non-relativistic regime of nuclear physics. It describes the effective dynamics of a particle confined within an atomic nucleus and can be formally derived from a relativistic model involving a Dirac operator; see Appendix A of \cite{klein2021nonlinear}. More precisely, it appears as an effective equation obtained from relativistic mean-field theory in a suitable asymptotic regime.

The theoretical description of atomic nuclei has been a central topic in nuclear physics for several decades, leading to the development of a variety of relativistic and non-relativistic models. Among the relativistic approaches, the relativistic mean-field theory (RMFT) has emerged as one of the most successful frameworks (\cite{ring1996relativistic},\cite{meng2006relativistic},\cite{reinhard1989relativistic},\cite{greiner1996nuclear}). In this theory, nucleons are described as Dirac particles interacting through meson fields. Since its introduction, the RMFT approach has provided an accurate description of numerous nuclear phenomena. Despite its remarkable success in physical applications, the mathematical analysis of relativistic nuclear models remains relatively limited (\cite{esteban2012symmetric},\cite{EstRot-13},\cite{lampart2021dirac}). This motivates the study of effective models such as \eqref{SNL_Rescaled}, which retain important features of the underlying relativistic dynamics while being more amenable to mathematical analysis.

In particular, equation \eqref{SNL_Rescaled} exhibits a quasi-linear structure, due to the singular coefficient $(1-|\Phi|^{2\alpha})^{-1}$ appearing in the kinetic term. This feature gives rise to significant analytical difficulties and distinguishes the model from the classical nonlinear Schrödinger equation. 

A central object of study in this context is the class of solitary wave solutions, which are sought in the form
\begin{equation}
\Phi(x,t) = e^{i \omega t} \varphi(x),
\end{equation}
where $\varphi$ is a real-valued, square-integrable function solving the associated stationary problem.

\begin{equation}
\label{SSNL}
-\omega \varphi
=
-\nabla \cdot \left(\frac{\nabla \varphi}{1-|\varphi|^{2\alpha}} \right)
+ \alpha |\varphi|^{2\alpha-2}
\frac{|\nabla \varphi|^2}{(1-|\varphi|^{2\alpha})^2}\varphi
- |\varphi|^{2\alpha}\varphi.
\tag{QLS}
\end{equation}

Solutions of \eqref{SSNL} arise as critical points of the energy functional
\begin{equation}
\label{energy}
\mathcal{E}(\varphi)
=
\frac{1}{2}\int_{\mathbb{R}^d}
\frac{|\nabla \varphi|^2}{1-|\varphi|^{2\alpha}}\,\mathrm{d}x
-
\frac{1}{2(\alpha+1)}\int_{\mathbb{R}^d} |\varphi|^{2\alpha+2}\,\mathrm{d}x.
\tag{$\mathcal{E}$}
\end{equation}
under the mass constraint
\[
\int_{\mathbb{R}^d} |\varphi|^2\,\mathrm{d}x = \mu > 0.
\]

In particular, a (normalized) ground state can be defined as a solution of the minimization problem
\begin{equation}\label{minimization_pb}\tag{$I$}
I(\mu)
=
\inf_{\substack{\varphi \in X \\ \int_{\mathbb{R}^d} |\varphi|^2 = \mu}}
\mathcal{E}(\varphi),
\end{equation}
where the functional space $X$ is given by
\begin{equation*}
X
=
\left\{
\varphi \in L^2(\mathbb{R}^d)
:\;
\int_{\mathbb{R}^d}
\frac{|\nabla \varphi|^2}{(1-|\varphi|^{2\alpha})_+}\,\mathrm{d}x
< +\infty
\right\}
\subset H^1(\mathbb{R}^d),
\end{equation*}
and $f_+$ denotes the positive part of a function $f$. This choice of functional setting is natural for the definition of ground states associated with this energy (see~\cite{EstRot-13}). Indeed, the energy $\mathcal E$ is not bounded from below on the constraint set
\[
\left\{
\varphi \in H^1(\mathbb{R}^d)
:\;
\int_{\mathbb{R}^d} |\varphi|^2\,\mathrm{d}x = \mu
\right\}.
\]
It has been shown in~\cite{EstRot-13} that if $\varphi \in X$, then $|\varphi|^2 \le 1$ almost everywhere in $\mathbb{R}^d$. Moreover, 
any solution to \eqref{minimization_pb} can be replaced by its absolute value without increasing the energy, since $|\nabla |\varphi||\le |\nabla \varphi|$ a.e. in $\R^d$. Therefore, we restrict our attention to solitary waves satisfying $0 \le \varphi < 1$.
The term $\int_{\mathbb{R}^d} |\varphi|^{2\alpha+2}\,\mathrm{d}x$ in \eqref{energy} corresponds to the standard focusing nonlinear Schrödinger interaction, which models the effective attraction responsible for nucleon confinement. The factor $(1 - |\varphi|^{2\alpha})^{-1}$ in the kinetic energy introduces a nonlinear saturation effect, enforcing the physical constraint $|\varphi| < 1$.
\medskip

 The special case $\alpha = 1$ and $d=3$ was first investigated by Esteban and Rota Nodari~\cite{esteban2012symmetric}, who proved the existence of positive spherically symmetric solutions to the associated stationary equation for frequencies $0<\omega<\frac{1}{\alpha+1}$. Subsequently, in~\cite{EstRot-13}, the same authors established the existence of normalized ground states for this equation. Later, Lewin and Rota Nodari~\cite{lewin2015uniqueness} proved the uniqueness and non-degeneracy of these solutions and extended the existence result of normalized ground states to dimension $d\geq 2$.

More recently, Klein and Rota Nodari~\cite{klein2021nonlinear} obtained existence and uniqueness results for $\alpha > 1$ in one spatial dimension ($d=1$), while Le Treust and Rota Nodari ~\cite{le2013symmetric} provided the first insights into the existence of excited states in the case $\alpha=1$ and $d=3$.  

From a mathematical perspective, the main goal of this paper is to establish existence, uniqueness, and non-degeneracy of positive solution for all $d \ge 2$ and $\alpha > 0$, as stated in the following theorem.

\begin{thrm}[Existence and Uniqueness of a Positive Solution]
\label{Thrm_existence_uniqueness_quasi0}
Let $d \ge 2$, $\alpha > 0$, and $\omega > 0$. 
\begin{itemize}
\item For $\omega \ge \frac{1}{\alpha+1}$, the nonlinear equation \eqref{SSNL} has no non-trivial solution $\varphi_\omega \in H^1(\xR^d)$ satisfying $0 \le \varphi_\omega < 1$.
\item For $0 < \omega < \frac{1}{\alpha+1}$, the nonlinear equation \eqref{SSNL} has a unique solution (up to translations) $\varphi_\omega \in H^1(\xR^d)$ such that $0 < \varphi_\omega < 1$ and $\lim_{|x| \to +\infty} \varphi_\omega(x) = 0$. Moreover, this solution satisfies:  
\begin{itemize}
    \item There exists a decreasing function $\psi: [0,+\infty) \to \xR$ and $T \in \xR^d$ such that $\varphi_\omega(x) = \psi(|x+T|)$ for all $x \in \xR^d$.
    \item $\varphi_\omega$ is non-degenerate in the sense that $\ker \mathcal L = \mathrm{span}\{\partial_{x_1}\varphi_\omega, \ldots, \partial_{x_d}\varphi_\omega, i \varphi_\omega\}$, where $\mathcal L$ is the linearized operator associated with \eqref{SSNL}.
    \item The mapping $\omega \mapsto \varphi_\omega$ is in $C^1((0, \frac{1}{\alpha+1}), H^2_\mathrm{rad}(\xR^d))$.
\end{itemize}
\end{itemize}
\end{thrm}

The existence result holds only for a restricted range of \(\omega\). Consequently, the asymptotic behavior of solutions near the endpoints of the admissible interval is of particular interest, as it reveals important features of the solution set and of the corresponding standing waves. Denoting $\omega^*:=\frac{1}{1+\alpha},$
we establish the following result concerning the critical regime \(\omega\to\omega^*\).

\begin{thrm}
\label{Asymptotic_omegastar}
Let $\varphi_\omega$ be the positive solution to \eqref{SSNL}. Then
\begin{equation*}
\lim_{\omega \to \omega^*} \|\varphi_\omega - 1\|_{L^\infty_{\mathrm{loc}}} = 0.
\end{equation*}
For the mass, we have
\begin{equation}
M(\omega) = \|\varphi_\omega\|^2_{L^2(\mathbb{R}^d)} =  \frac{K}{(\omega^* - \omega)^d} + o\Big( \frac{1}{(\omega^* - \omega)^d} \Big),
\end{equation}
where $K$ is a constant depending on $\alpha$ and $d$.
\end{thrm}

Thus, $\varphi_\omega$ converges to 1 as $\omega \to \omega^*$, leading to a singularity since the denominator in \eqref{SSNL} vanishes. This aligns with the non-existence of solutions beyond this critical value. 

When $\omega$ approaches the critical value $0$, three regimes depending on $\alpha$ emerge: super-critical, critical, and sub-critical. The following theorem summarizes the behavior of the solution $\varphi_\omega$ for $\omega \to 0$.

\begin{thrm}
\label{THRM_omega_0}
Let $\varphi_\omega$ be the unique positive solution to \eqref{SSNL}. We have 
\begin{itemize}

 \item \textbf{Super-critical case:} Suppose $d \ge 3$ and $\alpha > \frac{2}{d-2}$. Then, as $\omega \to 0^+$, $\varphi_\omega \to \varphi_0$ in $\dot{H}^1(\mathbb{R}^d)\cap L^q(\xR^d)$ for all $q\ge\frac{2d}{d-2}$, where $\varphi_0$ is the unique positive radial solution to
    \begin{equation*}
            -\nabla \cdot \left( \frac{\nabla \varphi}{1 - |\varphi|^{2\alpha}} \right) + \alpha |\varphi|^{2\alpha - 2} \frac{|\nabla \varphi|^2}{(1 - |\varphi|^{2\alpha})^2} \varphi - |\varphi|^{2\alpha} \varphi = 0.
    \end{equation*}
    The $L^2$-norm satisfies
    \[
    \lim_{\omega \to 0} \|\varphi_\omega\|_{L^2(\xR^d)} = 
    \begin{cases} 
    +\infty, & d \in \{3,4\},\\
    \|\varphi_0\|_{L^2(\xR^d)}, & d \ge 5.
    \end{cases}
    \]
    
    \item \textbf{Critical case:} Suppose $d \ge 3$ and $\alpha = \frac{2}{d-2}$. Then there exists a function $\omega\mapsto \lambda_\omega\in (0,+\infty)$ such that, as $\omega \to 0^+$, $\lambda_\omega\to +\infty$ and the rescaled function $\tilde \varphi_\omega(x) = \lambda_\omega^{\frac{d-2}{2}} \varphi_\omega(\lambda_\omega x)$ converges to $S$, the positive solution to
    \[
    -\Delta S = S^{\frac{d+2}{d-2}}
    \]
    in $\dot{H}^1(\xR^d) \cap L^{\frac{2d}{d-2}}(\xR^d)$. Moreover,
    \[
    \lim_{\omega \to 0} \|\varphi_\omega\|_{L^2(\xR^d)} = +\infty.
    \]
    \item \textbf{Sub-critical case:} Suppose $d=2$, or $d \ge 3$ and $\alpha < \frac{2}{d-2}$. As $\omega \to 0$, the rescaled function $\tilde{\varphi}_\omega(x) = \omega^{-\frac{1}{2\alpha}} \varphi_\omega(\omega^{-\frac12} x)$ converges in $H^1(\xR^d) \cap L^{\infty}(\xR^d)$ to $\psi_0$, the unique positive radial solution of
    \begin{equation}
    \label{limit_RSNL01}
    -\Delta \psi + \psi - |\psi|^{2\alpha} \psi = 0.
    \end{equation}
    Moreover, the $L^2$-norm of $\varphi_\omega$ has the following asymptotics:
    \begin{itemize}
        \item \textbf{$L^2$-critical case ($\alpha = \frac{2}{d}$):} 
        \[
        \|\varphi_\omega\|_{L^2(\mathbb{R}^d)}^2 \to \|\psi_0\|_{L^2(\xR^d)}^2 \quad \text{as } \omega \to 0.
        \]
        \item \textbf{$L^2$-super-critical case ($\alpha > \frac{2}{d}$):} 
        \[
        \|\varphi_\omega\|_{L^2(\mathbb{R}^d)}^2 \to +\infty.
        \]
        \item \textbf{$L^2$-sub-critical case ($\alpha < \frac{2}{d}$):} 
        \[
        \|\varphi_\omega\|_{L^2(\mathbb{R}^d)}^2 \to 0.
        \]
    \end{itemize}
    
\end{itemize}
\end{thrm}

We study the asymptotic behavior of the solution and its mass near the critical values, as this provides insight into whether the solution is a ground state. This question is investigated in detail in \cite{Bahhi-25}.

This paper is organized as follows. In Section \ref{sec:wellposdness}, we establish the existence and uniqueness of solutions. Section \ref{sec:limit_omega*} is devoted to the study of the behavior of $\varphi_\omega$ near the critical value $\omega^*$, where we prove Theorem \ref{Asymptotic_omegastar}. Finally, in Section \ref{sec:limit_0}, we investigate the asymptotic behavior of the solution as $\omega \to 0$ and provide the proof of Theorem \ref{THRM_omega_0}.

\section{Existence, uniqueness and non-degeneracy of the positive solution to the quasi-linear Schrödinger equation }

\label{sec:wellposdness}

\subsection{Change of variable.}

The solutions of the equation \eqref{SSNL}  are critical points of the energy \eqref{energy} 
under the $L^2$ normalisation constraint. The first term on the right hand side of \eqref{energy}, that can be interpreted by analogy as the kinetic energy, exhibits a singularity. Therefore we want to eliminate this singularity by employing a change of variables. Formally, we look for a function $\Lambda: \xR \to \xR $ and a function $u:\xR^d \to \xR$ such that
 \begin{equation*}
\label{change}
\left\{
\begin{aligned}
\varphi(x) &=\Lambda(u(x)),\\ 
  |\nabla u(x)|^2&=\frac{|\nabla \varphi(x)|^2}{1-|\varphi(x)|^{2\alpha}}.
\end{aligned}
\right.
\end{equation*}
In other words, the function $\Lambda$ has to be a solution of the following Cauchy problem
\begin{equation}
 \label{ODE1}
 \left \{
 \begin{array}{rcl}
 v'(t)&=&\sqrt{1-v(t)^{2\alpha}}  \\
 v(0)&=&0.
\end{array}
 \right.
\end{equation}

The solution to \eqref{ODE1} for positive $t$ can be constructed explicitly. For any $x\in [0,1)$, we define 
\begin{equation*}
    \Phi(x)=\int_0^x \frac{1}{\sqrt{1-w^{2\alpha}}}\xd w.
\end{equation*}
The function $\Phi$ is in $C^1((0,1))$, with 
\begin{equation*}
    \Phi'(x)=\frac{1}{\sqrt{1-x^{2\alpha}}}>0 \text{ for all } x\in (0,1), \text{ and } \lim_{x\to 0^+} \Phi'(x)=1.
\end{equation*}
Hence, $\Phi$ is increasing and therefore admits an inverse on its range. Since  $\int_0^{1} \frac{1}{\sqrt{1-w^{2\alpha}}}dw<\infty$ for any $\alpha>0$, we can define 
%Let us note that $1\le t^*<\infty$. In fact we know that 
\begin{equation}
\label{t*}
t^*
%=\int_0^{t^*} \frac{v'(t)}{\sqrt{1-v^{2\alpha}(t)}}\xd t
=\int_0^{1} \frac{1}{\sqrt{1-w^{2\alpha}}}\xd w.
\end{equation}
As a consequence $\Phi : [0,1)\to [0,t^*)$ is a bijection, and we can define 
\begin{equation*}
    \tilde \Lambda(t)=\Phi^{-1}(t) \text{ for any } t\in [0,t^*).
\end{equation*}
Since $\Phi(0)=0$, we deduce $\tilde \Lambda(0)=0$ and a direct computation gives
\begin{equation*}
    \tilde \Lambda'(t)= \sqrt{1-\tilde \Lambda^{2\alpha}(t)}
\end{equation*}
for any $t\in (0,t^*)$. Hence, $\tilde \Lambda(t)$ is a solution to the Cauchy problem \eqref{ODE1}. Moreover, from \eqref{t*}, we deduce that
\begin{equation*}
    1\le t^*<+\infty \text{ and } \lim_{t\to t^*}\tilde{\Lambda}(t)=1.
\end{equation*}
Furthermore, the solution of \eqref{ODE1} is uniquely defined on $[0,t^*)$. Indeed, all the solutions $v$ to \eqref{ODE1} are such that $v'(t)>0$ for $t$ close to $0$ so that $v$ is positive and increasing close to $0$. More precisely, $v'(t)>0$ whenever $v(t)<1$ and we can write $\Phi(v(t))=t$ with $\Phi$ defined as above. Since $\Phi$ is injective, this implies $v(t)=\tilde{\Lambda}(t)$ for all $t\in [0,t^*)$, and proves the uniqueness.

We can then define our change of variables as follows.

\begin{defi}
\label{defi_change_var}
  Let $\Lambda: \xR \to \xR$ be defined as
  \begin{equation*}
      \Lambda(t)=\left\{
        \begin{aligned}
            &\Tilde{\Lambda}(t) & 0\le t<t^*\\
            &1 & t\ge t^*
        \end{aligned}
      \right.
  \end{equation*}
  and $\Lambda(t)=\Lambda(-t)$ for all $t<0$.
\end{defi}

\begin{prop}
\label{prop_change_var}
The function $\Lambda:\xR\to\xR$ given by Definition \ref{defi_change_var} satisfies:
\begin{enumerate}[label=\alph*)]
    \item for all $t\in (0,t^*)$, $\Lambda'(t)= \sqrt{1-\Lambda^{2\alpha}(t)}$ and $\Lambda''(t)= -\alpha \Lambda^{2\alpha-1}(t)$; \label{prop_change_var7}
    \item for all $t\in \xR$, $0\leq\Lambda(t)\leq 1$; \label{prop_change_var1}
    \item $\Lambda$ is non-decreasing in $\xR^+$ and increasing in $(0,t^*)$; \label{prop_change_var2}
    \item for all $t\in \xR^+$, $\Lambda'(t)t \leq\Lambda(t)\leq t$;  \label{prop_change_var4} 
    \item $\Lambda$ is 1-Lipschitz on $\xR$;  \label{prop_change_var5}
    \item $\Lambda'$ is $\alpha$-Lipschitz on $\xR^+$ for $\alpha\ge \tfrac{1}{2}$ and $\Lambda'\in C^{0,2\alpha}(\xR^+)$ for $0<\alpha< \tfrac{1}{2}$; \label{prop_change_var6}
     \item $\lim\limits _{t \rightarrow  0^+} \frac{\Lambda(t)}{t}=\lim\limits_{t \rightarrow  +\infty}  \Lambda(t)=1$. \label{prop_change_var3}
\end{enumerate}
\end{prop}

\begin{proof}
Property \ref{prop_change_var7} is a straightforward computation. Properties \ref{prop_change_var1}, \ref{prop_change_var2},  \ref{prop_change_var4} and \ref{prop_change_var5} are immediate from the  fact that $0\leq\Lambda'(t)\leq 1$ for all $t \in \xR$, $\Lambda''(t)\le 0$ for all $t>0$,  and $\Lambda(0)=0$.  For $\alpha\ge \tfrac{1}{2}$, property \ref{prop_change_var6} follows from  \ref{prop_change_var7} that implies  $|\Lambda''(t)|\leq \alpha$ for all $t\in (0,t^*)$. For $0<\alpha< \tfrac{1}{2}$, we remark that, for any $t\in (0,t^*)$,
\begin{equation*}
    |\Lambda'(t)-\Lambda'(0)|=1-\sqrt{1-\Lambda^{2\alpha}(t)}=\frac{\Lambda^{2\alpha}(t)}{1+\sqrt{1-\Lambda^{2\alpha}(t)}}\le |t-0|^{2\alpha}
\end{equation*}
which, together with $|\Lambda''(t)|\leq \alpha$ for all $t\in (0,t^*)$, leads to $\Lambda'\in C^{0,2\alpha}(\xR)$. Finally, the fact that 
$\lim\limits _{t \rightarrow  0^+}\frac{\Lambda(t)}{t}=1$ is a consequence of the mean value theorem, since $\lim\limits _{t \rightarrow  0^+}\frac{\Lambda(t)}{t}=\lim\limits _{x \rightarrow  0^+}\Lambda'(x)=1$.

\end{proof}

Now that our change of variables is well defined, we will inject it into the equation \eqref{SSNL} to transform the quasi-linear Schrödinger equation \eqref{SSNL} into the following semi-linear problem

\begin{equation}
\label{PSSL}
\left\{
\begin{aligned}
-\Delta u(x) &= f_\omega(u(x)) \qquad \text{ in } \xR^d\\   
  u\in& H^1(\xR^d),  \qquad u \neq 0
\end{aligned}
\right. \tag{SLS}
\end{equation}

with  
\begin{equation}
\label{f}
f_\omega(t)=\Lambda'(t)\left( \Lambda(t)^{2\alpha+1}-\omega\Lambda(t)\right) \text{ for all } t\in \xR\setminus \{0\},\quad f_\omega(0)=0.
\end{equation}
Note that the function $f:\xR\to \xR$ is a real continuous function which is odd and such that $f_\omega(0)=0$. Moreover, $f_\omega(t)=0$ for all $t\ge t^*$.

Our goal is then to establish the existence and uniqueness of positive solutions to this semi-linear problem. To this end, we first investigate some general properties of the non-linearity $f_\omega$ and of positive solutions to \eqref{PSSL}.

\subsection{Properties of the non-linearity to the semi-linear problem}

\begin{prop}
\label{prop_f}
Let $\Lambda $ be the change of variables defined in Definition \ref{defi_change_var}.
The function $f_\omega$ defined by \eqref{f}

satisfies:
\begin{enumerate}[label=\alph*)]
    \item for all $t\in \xR^+$, $f_\omega(t)\leq \min\{1,t\}$ and  $|f_\omega(t)|\leq (1+\omega)\min\{1,|t|\}$; \label{Prop_f1}
    \item for all $t\in (0,t^*)$,
    \begin{equation*}
         f_\omega'(t)=(\Lambda^{2\alpha}(t)-\omega)(1-(1+\alpha)\Lambda^{2\alpha}(t))+2\alpha\Lambda^{2\alpha}(t)(1-\Lambda^{2\alpha}(t))
    \end{equation*}
    and 
     \begin{equation*}
        f_\omega''(t)=2\alpha \Lambda^{2\alpha-1}(t)\Lambda'(t)\left(-(6\alpha+2)\Lambda^{2\alpha}(t)+(2\alpha+1)+\omega(1+\alpha)\right);
    \end{equation*}\label{Prop_f2}
    
    \item $f_\omega$ is k-Lipschitz with $k=\max\{1-\omega,\omega\}\max\{1,\alpha\}+ 2\alpha$; \label{Prop_f4}   
    
    \item for all $t \in (0, t^*)$,
    \begin{equation*}
      F_\omega(t) =  \int_0^t f_\omega (s) \xd s = \frac{1}{2(\alpha+1)}\Lambda^{2(\alpha+1)}(t) -\frac{\omega}{2} \Lambda^2(t)
    \end{equation*}\label{Prop_f5}   
\end{enumerate}
\end{prop}

\begin{proof}
Property \ref{prop_f}\ref{Prop_f1} is a direct result of the properties  \ref{prop_change_var1} and \ref{prop_change_var4} of Proposition \ref{prop_change_var}.
Properties \ref{prop_f}\ref{Prop_f2} and \ref{prop_f}\ref{Prop_f5} follow directly from the definition. For \ref{prop_f}\ref{Prop_f4}, we remark that, for any $t\in (0,t^*)$, 
\begin{equation*}
    |1-(1+\alpha)\Lambda^{2\alpha}(t)|\le \max\{1,\alpha\}\text{ and } |\Lambda^{2\alpha}(t)-\omega|\le \max\{1-\omega,\omega\}.
\end{equation*}

As a consequence,
 \begin{align*}
   |f_\omega'(t)|&\leq|(\Lambda^{2\alpha}(t)-\omega)(1-(1+\alpha)\Lambda^{2\alpha}(t))|+|2\alpha\Lambda^{2\alpha}(t)(1-\Lambda^{2\alpha}(t))|\\
         &\leq |(\Lambda^{2\alpha}(t)-\omega)|\max\{1,\alpha\}+ |2\alpha\Lambda^{2\alpha}(t)(1-\Lambda^{2\alpha}(t))|\\
         &\underbrace{\leq}_{\ref{prop_change_var}\ref{prop_change_var1}} \max\{1-\omega,\omega\}\max\{1,\alpha\}+2\alpha.
 \end{align*}
\end{proof} 

\subsection{Properties of the solutions to the semi-linear problem.}
Any positive solution to problem \eqref{PSSL}, if it exists, satisfies the following lemma.

\begin{lem}
\label{lem_properties}
	If $u\in H^1(\xR^d)$ is a positive solution of \eqref{PSSL}, then
	\begin{enumerate}
	    \item The solution $u$ is in $ L^\infty(\xR^d)$. Moreover $u\in C^1(\xR^d)$ and  $\lim\limits_{|x|\to\infty}u(x)= 0$ and $ \lim\limits_{|x|\to\infty} \nabla u(x)= 0$. \label{L1}
         \item  There exists $v:[0,+\infty)\to \xR$ a decreasing function
    and $T\in \xR^d$ such that $u(x)=v(|x+T|)$ for all  $x\in \xR^d$.\label{L2}

     \item The solution  $u(x)<t^*$ for all $x\in\xR^d$. \label{L3}
     \item The solution  $u$ belongs to $C^2(\xR^d)$. \label{L4}
	\end{enumerate}
\end{lem}

\begin{proof}
\begin{itemize}
    \item We start by proving \eqref{L1}. Let $u\in H^1(\xR^d)$ be a positive solution of $-\Delta u= f(u)$, we consider for $n\in \xN$ the following sequence 
 \begin{equation}
\label{qn}
\left\{
\begin{aligned}
q_{n+1}&=\frac{dq_n}{d-2q_n} \quad \text{ if } \quad q_n<\frac{d}{2} \\ 
q_{n+1}&=d+1 \quad \text{ if } \quad q_n\geq \frac{d}{2}\\ 
q_0&=2
\end{aligned}
\right.
\end{equation}
and we want to show by induction that $u \in W^{2,q_n}(\xR^d) $ for all $n\in \xN$.\\
For $n=0$, since $ u\in L^2(\xR^d)$ then by Prop.~\ref{prop_f}\ref{Prop_f1} we have $f_\omega(u)\in L^2(\xR^d) $, this implies $\Delta u\in L^2(\xR^d)$ and by the elliptical regularity result  \cite[Theorem 9.32]{brezis2011functional}, applied to $-\Delta u+ u = g(u)$ with $g(u)=f_\omega(u)-u \in L^2(\xR^d)$, we have 
$$u\in W^{2,q_0}(\xR^d).$$
Now, if $u \in W^{2,q_n}(\xR^d)$, then due to Sobolev embedding  \cite[Corollary 9.13]{brezis2011functional}, we have $u \in L^{q_{n+1}}(\xR^d)$, this implies that $g(u)\in L^{q_{n+1}}(\xR^d) $  by Prop.~\ref{prop_f}\ref{Prop_f1} and using the elliptical regularity again, we have 
$$u\in W^{2,q_{n+1}}(\xR^d).$$
In particular, there exists $N\in\xN$ such that  $q_{_{N+1}} =d+1$ so that $u\in W^{2,d+1}(\xR^d)$.
In fact, if $q_n\leq \frac{d}{2}$, the sequence $q_n$ is increasing and 
$$q_{n+1}-q_n\geq \frac{2}{d-2},$$
therefore 
$$\exists N\in \xN \quad  q_{_N}>\frac{d}{2}$$
 then $$q_{_{N+1}}=d+1.$$

Finally, by applying Morrey's inequality  \cite[Theorem 9.12]{brezis2011functional},
we have $$u\in W^{2,d+1}(\xR^d)\subset C^1(\xR^d).$$
Since $C^\infty_c(\xR^d)$ is dense in $W^{2,d+1}(\xR^d) $, we deduce that $\lim\limits_{|x|\to\infty}u(x)= 0$ and $ \lim\limits_{|x|\to\infty} \nabla u(x)= 0.$

\item Next, we prove \eqref{L2}. To this goal we use the result of Gidas, Ni, Nirenberg  \cite{gidas1981symmetry} as presented in  \cite[Theorem 4.1]{frank2013ground}. To be able to use this result which is an application of the method of moving planes, we have to show that $f_\omega$ is a function on $[0, \infty)$ satisfying
\begin{equation}
\label{Condition_F}
\frac{f_\omega(s)-f_\omega(t)}{s-t} \leq-\tau^2+C s^{\min\{1,2\alpha\}} \quad \text { for all  } 0 \leq t<s,
\end{equation}
for some $C \geq 0, \tau>0$.
In this case any positive solution $u \in H^1(\xR^d) \cap L^\infty(\xR^d)$ of $$	-\Delta u=f_\omega(u)$$ 
is radial with respect to some point and  decreasing with respect to the distance from this point.

Thanks to  \eqref{L1}, $u\in L^\infty(\xR^d) $, so it remains to show that $f$ satisfies \eqref{Condition_F}. Let $t^*\ge s > t\geq 0$,
\begin{align*}
    \frac{f_\omega(s)-f_\omega(t)}{s-t}=\int_0^1 f_\omega'(t+\sigma (s-t))\xd\sigma.
\end{align*}
We have $$f_\omega'(r)=-\omega -(3\alpha+1)\Lambda^{4 \alpha} (r)+(\alpha\omega+\omega+ 2\alpha +1)\Lambda^{2\alpha}(r)$$
which implies that 
$$ \frac{f_\omega(s)-f_\omega(t)}{s-t} \leq -\omega + \int_0^1 (\alpha\omega+\omega+ 2\alpha +1)\Lambda^{2\alpha}(\sigma)(t+\sigma(s-t))\xd\sigma.$$
Using the fact that $\Lambda$ is increasing, $\Lambda(s)\leq \min\{1,s\}$, and $t+\sigma(s-t)\leq s$, we obtain 
\begin{align*}    
    \frac{f_\omega(s)-f_\omega(t)}{s-t}&\leq -\omega + \int_0^1 (\alpha\omega+\omega+ 2\alpha +1)\Lambda^{2\alpha}(s)\xd\sigma \\  
    &\leq -\omega + (\alpha\omega+\omega+ 2\alpha +1)s^{\min\{1,2\alpha\}}
\end{align*}
with $\omega>0$. If $s>t\ge t^*\ge 1$,
\begin{align*}
    -\omega + (\alpha\omega+\omega+ 2\alpha +1)s^{\min\{1,2\alpha\}}\ge \omega((t^*)^{\min\{1,2\alpha\}}-1)\ge 0= \frac{f_\omega(s)-f_\omega(t)}{s-t}
\end{align*}
and \eqref{Condition_F} holds. Finally, if $0\le t<t^*<s$, we consider two different cases. Either $f_\omega(t)\ge 0$, so that
\begin{align*}
    \frac{f_\omega(s)-f_\omega(t)}{s-t}=\frac{-f_\omega(t)}{s-t}\le 0\le -\omega + (\alpha\omega+\omega+ 2\alpha +1)s^{\min\{1,2\alpha\}}
\end{align*}
as above. Either $f_\omega(t)< 0$, and we have 
\begin{align*}
    \frac{f_\omega(s)-f_\omega(t)}{s-t}&=\frac{f_\omega(t_*)-f_\omega(t)}{s-t}\le \frac{f_\omega(t_*)-f_\omega(t)}{t^*-t}\\
    &\le -\omega + (\alpha\omega+\omega+ 2\alpha +1)(t^*)^{\min\{1,2\alpha\}} \\
    &\le -\omega + (\alpha\omega+\omega+ 2\alpha +1)s^{\min\{1,2\alpha\}}.
\end{align*}
As a conclusion, \eqref{Condition_F} holds for any $s>t\ge 0$.

\item We prove \eqref{L3}. Let $ u $ be a positive solution of \eqref{PSSL}, then $u$ is spherically symmetric with respect to some point $T\in \xR^d$ and  decreasing with respect to the distance from $T$, thanks to \eqref{L2}. With no loss of generality one can take $T=0$. Now, we use the maximum principle in order to show that  $u(x)\leq t^*$, then, since $u $ is radial, instead of studying the PDE, we work on the ODE to prove that $u< t^*$.

Let us consider $$\Omega=\{x \in \xR^d\, |\, u(x)>t^*\}$$
and suppose by contradiction that $\Omega \ne \emptyset$. Then since $u$ is radial, $\Omega =B_r(0) \text{ with }  r>0.$ 
 We have 
$$-\Delta u(x)=f_\omega(u(x))=0 \qquad \forall x\in\Omega. $$ 
By elliptic regularity $u\in C^2(\Omega)\cap C(\bar{\Omega})$.
Now thanks to the weak maximum principle \cite[section 6.4, Theorem 1]{evans2009partial}
$$\max_{|x|\leq r} u(x)=\max_{|x|=r} u(x)= t^*$$
which contradicts the fact that $\Omega$ is not an empty set, then $\Omega = \emptyset$ and $u(x)\leq t^*$ for all $x\in \xR^d$. More precisely, since  $u$ is  decreasing 
$$t^*\geq u(0)> u(x) \qquad \forall x\in\xR^d\setminus \{0\}.$$
Then it remains to show that $t^*\ne u(0)$. Let us suppose by contradiction that $t^*= u(0)$. Then in the spherical coordinates $u$ is a solution of the following Cauchy problem for all $r\in(0,+\infty)$:
\begin{equation}
\label{Cauchy_P1}
\left\{
\begin{array}{ccc}
u^{\prime\prime} (r)+\frac{d-1}{r}u^\prime(r)+&f_\omega(u(r)) = & 0 \\
&u'(0) =&  0\\
&u(0) =& t^*.\\
\end{array}
\right.
\end{equation}
By integrating this ODE twice we have  
\begin{align*}
  0 &=u''(r)+\frac{d-1}{r}u'(r)+f_\omega(u(r))  \\
   &=r^{d-1} u''(r)+(d-1)r^{d-2}u'(r)+f_\omega(u(r)) r^{d-1} \\
   &= (r^{d-1} u'(r))'+f_\omega(u(r)) r^{d-1}\\
   u'(r)& = -\frac{1}{r^{d-1}}\int_0^r f_\omega(u(s))s^{d-1}~\xd s\\ 
 t^*-  u(r)& = \int_0^r\int_0^t\frac{s^{d-1}}{t^{d-1}} f_\omega(u(s))~\xd s~\xd t.\\ 
\end{align*}
Using that $f_\omega(t^*)=0$, we have 
\begin{align*}
t^*-u(r) &= \int_0^r\int_0^t\frac{s^{d-1}}{t^{d-1}} f_\omega(u(s))~\xd s~\xd t \\
&\leq \int_0^r\int_0^t\frac{s^{d-1}}{t^{d-1}} |f_\omega(u(s))-f_\omega(t^*)|~\xd s~\xd t .\\
\end{align*}
We know that $f_\omega$ is Lipschitz thanks to Proposition \ref{prop_f},\ref{Prop_f4}. As a consequence, for all $\delta>0$ and $r\in (0,\delta)$, we have 
\begin{align*}
|t^*-u(r)| &\leq K\| u-t^*\|_{L^\infty((0,\delta))}\int_0^r\int_0^t\frac{s^{d-1}}{t^{d-1}}~\xd s~\xd t.
\end{align*}
Then
\begin{equation}
\label{norm}
    \|t^*-u\|_{L^\infty((0,\delta))} \leq \frac{K\delta^2}{2d} \|t^*-u\|_{L^\infty((0,\delta))}.
\end{equation}
 In particular, if we choose  $\delta< \sqrt{\frac{2d}{K}}$, \eqref{norm} implies that  $u(r)=t^*$  for all $r\in(0,\delta)$. This contradicts the fact that $u $ is  decreasing.
 
 \item  Finally, it remains to prove \eqref{L4}. As above, if  $ u $ is a positive solution of \eqref{PSSL}, then $u$ is spherically symmetric, thanks to \eqref{L2}, and solves the ODE
 \begin{equation*}
     -u''(r)-\frac{d-1}{r}u'(r)=f_\omega(u(r))
 \end{equation*}
 for all $r\in (0,+\infty)$. Since $u$ is in $C^1$, thanks to \eqref{L1}, and $f_\omega$ is continuous, we  have that $u''$ is continuous except possibly at $0$. However, we have 
\begin{align*}
    u'(r)& = -\frac{1}{r^{d-1}}\int_0^r f_\omega(u(s))s^{d-1}~\xd s= -r\int_0^1 f_\omega(u(r\tau))\tau^{d-1}~\xd s
\end{align*}
so that 
\begin{align*}
    \frac{u'(r)}{r}& = -\int_0^1 f_\omega(u(r\tau))\tau^{d-1}~\xd s\xrightarrow[r\to 0]{}-\frac{f_\omega(u(0))}{d}
\end{align*}
since $u$ and $f_\omega$ are continuous functions. As a consequence, $u''(0)=-\frac{f_\omega(u(0))}{d}$ and $u\in C^2(\R^d)$.

 \end{itemize}
\end{proof}

We now state the theorem on the existence and uniqueness of a positive solution to the semi-linear problem.
\begin{thrm}
\label{Thrm_existence_uniqueness}
Let $\alpha\geq 0$ and $\omega>0$. 
\begin{itemize}
    \item For $\omega\geq \frac{1}{\alpha+1}$ the semi-linear problem \eqref{PSSL} has no non-trivial solution. 
\item For $0<\omega< \frac{1}{\alpha+1}$ the semi-linear problem \eqref{PSSL} has a unique solution $u>0 $, up to translations.\\
Moreover:
\begin{itemize}
    \item There exists $v:[0,+\infty)\to \xR$ a   decreasing function
    and $T\in \xR^d$ such that $u(x)=v(|x+T|)$ for all  $x\in \xR^d$.
    \item It is non-degenerate in the sense that $\ker (-\Delta -f'_\omega(u))=\mathrm{span}\{\partial_{x_1} u,\ldots,\partial_{x_d}u\}$   and $\ker (-\Delta -\frac{f_\omega(u)}{u})=\mathrm{span}\{u \}$.
\end{itemize}
\end{itemize}
\end{thrm}

To prove Theorem \ref{Thrm_existence_uniqueness}, we will decompose the result into three parts: existence, uniqueness and non-degeneracy for $\omega<\frac{1}{\alpha+1}$, and non-existence for $\omega\geq\frac{1}{\alpha+1}$ . We will prove each part separately.
   \begin{prop}[Existence]
   	\label{thrm_existence}
   	Let $\alpha >0$. For  $0<\omega<\frac{1}{\alpha+1}$, the problem \eqref{PSSL} has a positive solution $u$  such that $u$ is spherically symmetric  with respect to $0$, and decreases exponentially to $0$ as $|x|\to \infty$.
   \end{prop}

\begin{proof}
	
We use the result of Berestycki and Lions  \cite{berestycki1983nonlinear} when $d\geq3$ and the result of Berestycki, Gallouët and Kavian  \cite{berestycki1983equations} when $d=2$ in ordre to prove this proposition. So we need to check if the function $f_\omega$ satisfies the following conditions:
\begin{equation}
\label{condition1}
-\infty<\lim _{s \rightarrow 0^{+}} f_\omega(s) / s \leqq \varlimsup_{s \rightarrow 0^{+}} f_\omega(s) / s=-m<0 
\end{equation}
\begin{equation}
\label{condition3}
\text { There exists } \zeta>0 \text { such that } \quad F_\omega(\zeta)=\int_{0}^{\zeta} f_\omega(s) \xd s>0
\end{equation}
\begin{equation}
\label{condition2}
-\infty \leqq \lim _{s \rightarrow+\infty} f_\omega(s) / s^{l} \leqq 0, \quad \text { where } l=\frac{d+2}{d-2} \quad\text{and } \quad d\geq 3
\end{equation}
\begin{equation}
\label{condition4}
\forall \gamma >0  \quad\exists C_\gamma >0     \qquad  f_\omega(s)\leq C_\gamma e^{\gamma s^2} \qquad \text{when } \quad d=2
\end{equation}
We have that $f_\omega$ satisfies the condition \eqref{condition1} because 

$$
\begin{aligned}
\lim _{s \rightarrow 0^+} \frac{f_\omega(s)}{s} &=\lim _{s \rightarrow 0^+}\frac{\Lambda(s)}{s} \sqrt{1-\Lambda(s)^{2 \alpha}}\left(\Lambda(s)^{2 \alpha}-\omega\right) \\
&=-\omega < 0.
\end{aligned}
$$
For the condition \eqref{condition3}, by taking $\zeta=t^*$, we have 
$$
\begin{aligned}
F_\omega(t^*)=\int_{0}^{t^*} f_\omega(s) \xd s &=\int_{0}^{t^*} \Lambda^{\prime}(s)\left(\Lambda(s)^{2 \alpha+1}-\omega \Lambda(s)\right) \xd s \\
&=\frac{1}{2} \Lambda(t^*)^{2}\left(\frac{1}{(\alpha+1)} \Lambda(t^*)^{2\alpha}-\omega\right)>0
\end{aligned}
$$
provided that
$$\frac{1}{(\alpha+1)}  >\omega.$$
The conditions \eqref{condition2}, \eqref{condition4} are verified since $ f_\omega(s)\leq 1$ by \ref{Prop_f1} from Proposition \ref{prop_f}.

\end{proof}

We now turn to proving that the positive solution to \eqref{PSSL} is unique.
\begin{prop}[Uniqueness and non-degeneracy]
\label{thrm_uniqueness}
	Let $\alpha >0$ and  $0<\omega<\frac{1}{\alpha+1}$. If $v$ is a positive solution of \eqref{PSSL}, then there exists $T\in \xR^d$ such that, for all $x\in \xR^d$, $v(x)=u(x+T)$, where $u$ is the solution given by Proposition \ref{thrm_existence}. Moreover, this solution is non-degenerate.
\end{prop}

\begin{proof}

Let $v$ be a positive solution of \eqref{PSSL}, there exists $T\in \xR^d$ such that $v$ is spherically symmetric and  decreasing with respect to $T$, thanks to Lemma \ref{lem_properties}(\ref{L2}). Moreover, by Lemma \ref{lem_properties}(\ref{L1}), 
\begin{equation}
\label{C1}
v\in C^1(\xR^d) \qquad \text{and} \qquad v,\nabla v \to 0 \qquad \text{when} \qquad |x|\to \infty
\end{equation}
and, by Lemma \ref{lem_properties}(\ref{L3}), 
\begin{equation}
\label{condi_norm_borne}
    \|v\|_{L^\infty(\xR^d)}< t^*.
\end{equation}
Hence, we can apply the result of Lewin and Rota Nodari  \cite[Theorem 1]{lewin2020double} to prove the uniqueness and the non-degeneracy of positive solutions to \eqref{PSSL}. This result states that the problem \eqref{PSSL} admits at most one positive radial solution which satisfies \eqref{C1} and \eqref{condi_norm_borne} if the function $f$ is continuously differentiable  on $[0,t^*]$ and verifies  
\begin{enumerate}[label=\arabic*)]
	\item  $f_\omega(0)=f_\omega(t_0)=f_\omega(t^*)=0, f_\omega$ is negative on $(0,t_0)$ and positive on $(t_0,t^*)$ with $f_\omega^{\prime}(0)<0, f_\omega^{\prime}(t_0)>0$ and $f_\omega^{\prime}(t^*
		) \leqslant 0$ for some $t_0\in (0,t^*)$ ;\label{cond_uniq_1}
    \item the function $t\to  \frac{t f_\omega'(t)}{f_\omega(t)}$ is  decreasing on $(t_0, t^*)$;
\label{cond_uniq_3}
    \item there exists $t_1\in (0, t^*)$ such that $f_\omega'' > 0 $ on $(0, t_1)$ and $f_\omega'' < 0$ on
$(t_1, t^*)$. \label{cond_uniq_2}
\end{enumerate}
Here, we use the convention that $f'_\omega(0)=\lim_{t\to 0^+}f'_\omega(t)$ and $f'_\omega(t^*)=\lim_{t\to (t^*)^{-}}f'_\omega(t)$.
Let us check that $f_\omega$ verifies the conditions above. 
By definition, $f_\omega(0)=0$ and, for any $t\in (0,t^*)$, we have
$$f_\omega(t)=\Lambda'(t)\left( \Lambda^{2\alpha+1}(t)-\omega\Lambda(t)\right)\mathrm.$$
Hence, we can see that $f_\omega$ verifies condition \ref{cond_uniq_1} using that $\Lambda'(t^*)=0$ and by taking $t_0\in (0,t^*)$ such that $\Lambda^{2\alpha}(t_0)=\omega$.\\
Moreover, if we denote $$f_\omega(t)= \Lambda'(t)P(\Lambda(t)) \qquad \text{with} \qquad P(\xi)=\xi^{2\alpha+1}- \omega \xi$$
and we use $\Lambda''(t)=-\alpha \Lambda^{2\alpha-1}(t)$ we obtain
$$f_\omega'(t)=-\alpha \Lambda^{2\alpha-1}(t)P(\Lambda(t))+(\Lambda'(t))^2 P'(\Lambda(t))\mathrm.$$
This gives 
$ f_\omega'(0)=-\omega<0$,  $f_\omega'(t_0)=2\alpha \omega(1-\omega)>0$,  $f_\omega'(t^*)=-\alpha(1-\omega)<0.$
For condition \ref{cond_uniq_3}, we have 
\begin{align*}	
    \frac{t f_\omega'(t)}{f_\omega(t)}	&= -\frac{\alpha t\Lambda^{2\alpha}(t)}{\Lambda'(t)\Lambda(t)}+ \frac{t\Lambda'(t)P'(\Lambda(t))}{P(\Lambda(t))}.	
\end{align*}

A direct computation leads to
\begin{equation}
\begin{aligned}
\left(\frac{t f_\omega'(t)}{f_\omega(t)}\right)'&=-\alpha\frac{\Lambda^{2\alpha}(t)\Lambda'(t)(\Lambda(t)+(2\alpha-1)t\Lambda'(t))}{(\Lambda'(t)\Lambda(t))^2}-\alpha^2\frac{t(\Lambda^{2\alpha}(t))^2}{(\Lambda'(t)\Lambda(t))^2}\\
&+\left(\frac{\Lambda(t)P'(\Lambda(t))}{P(\Lambda(t))}\right)\left(\frac{\Lambda'(t)\Lambda(t)-t(1+(\alpha-1)\Lambda(t)^{2\alpha})}{\Lambda(t)^2}\right)\\
&+t\frac{\Lambda'(t)}{\Lambda(t)}\left(\frac{\Lambda(t)P'(\Lambda(t))}{P(\Lambda(t))}\right)^\prime.
\end{aligned}
    \label{con_f'_f}
\end{equation}
The first term in the first line of \eqref{con_f'_f} is clearly positive for all $\alpha\geq\frac{1}{2}$, and for all $0<\alpha<\frac{1}{2}$ we have that 
\begin{align}
\label{estimation}
   0< 1-2\alpha < 1 &\implies  (1-2\alpha) t \Lambda'(t) \leq t \Lambda'(t)  \nonumber \\ 
   & \underbrace{\implies }_{\ref{prop_change_var}\ref{prop_change_var4}} (1-2\alpha) t \Lambda'(t) \leq \Lambda(t)  \nonumber\\
   &\implies 0 \leq \Lambda(t) + (2\alpha-1) t \Lambda'(t).
\end{align}
The second line of \eqref{con_f'_f} is negative  on $(t_0,t^*)$  because, for all $ t>t_0$,
\begin{align*}
\label{P}
    \frac{\Lambda(t)P'(\Lambda(t))}{P(\Lambda(t))}&=(2\alpha+1)+\frac{2\alpha\omega}{\Lambda^{2\alpha}(t)-\omega}>  0,
\end{align*}    
and $h(t)=\Lambda'(t)\Lambda(t)-t(1+(\alpha-1)\Lambda(t)^{2\alpha})\leq0$ for all $t>0$. Indeed, we have $h(0)=0$ and as in \eqref{estimation}  using \ref{prop_change_var}\ref{prop_change_var4} we can show that 
$$h'(t)=-2\alpha\Lambda(t)^{2\alpha-1}( \Lambda(t)+(\alpha-1)t\Lambda'(t))\leq 0.$$
The third line of \eqref{con_f'_f} is  negative on $(t_0,t^*)$ because, for all $t\in(t_0,t^*)$, $\Lambda$ and $\Lambda'$ are positive and
\begin{equation*}
\left(\frac{\Lambda(t)P'(\Lambda(t))}{P(\Lambda(t))}\right)'=\frac{-4\alpha^2\omega \Lambda^{2\alpha-1}(t)\Lambda'(t)}{(\Lambda^{2\alpha}(t)-\omega)^2}<0.
\end{equation*}

For condition \ref{cond_uniq_2}, we have 
$$f_\omega''(t)=\underbrace{2\alpha \Lambda^{2\alpha-1}(t)\Lambda'(t)}_{\geq0}\left(-(6\alpha+2)\Lambda^{2\alpha}(t)+(2\alpha+1)+\omega(1+\alpha)\right)$$
for all $t\in (0,t^*)$.
Then the sign of $f_\omega''$ is the sign of $$Q(t)=-(6\alpha+2)\Lambda^{2\alpha}(t)+(2\alpha+1)+\omega(1+\alpha).$$
Using again that $\Lambda$ and $\Lambda'$ are positive, we deduce that the function $Q$ is decreasing on $(0,t^*)$, and we have $$Q(0)=2\alpha+1+\omega(1+\alpha)>0,\qquad  Q(t^*)=\omega(1+\alpha)-4\alpha-1<0.$$
Hence there exists $t_1$ such that $f_\omega''>0$ on $(0,t_1)$ and $f_\omega''<0$ on $(t_1,t^*)$.
Hence, \cite[Theorem 1]{lewin2020double} can be applied and we conclude that the solution $u$ is unique and non-degenerate in the sense that $\mathrm{ker} (-\Delta -f'_\omega(u))=\mathrm{span}\{\partial_{x_1} u,\ldots,\partial_{x_d}u\}$   and $\mathrm{ker} (-\Delta -\frac{f_\omega(u)}{u})=\mathrm{span}\{u \}$.
\end{proof}

\begin{rmrk}
\label{Remark_Linearized_op}
Let $\mathcal{L_\omega}=(-\Delta -f'_\omega(u))$ be the real part of the linearised operator related to the problem \eqref{PSSL} at our solution $u$. Then $\mathrm{ker} \mathcal{L}_{\omega{|\mathrm{rad}}}=\{0\}$.
\end{rmrk}
Now we are going to prove that there is no positive solution of \eqref{PSSL} if $\omega\geq \frac{1}{\alpha+1}$. 

\begin{prop}[Non-existence for $\omega\geq\frac{1}{1+\alpha}$]
Let $\alpha> 0$ and $\omega\geq \frac{1}{\alpha+1}$. The nonlinear problem \eqref{PSSL} has no non-trivial positive solution.
\end{prop}
\begin{proof}
Let $\omega\geq \frac{1}{\alpha+1}$, we suppose by contradiction that $u$ is a positive solution of \eqref{PSSL}.
By Lemma \ref{L2}, $u$ is radial with respect to $T\in\xR^d$, by a translation of coordinates  $r=|x-T|$, $u$ is a solution of the following problem:

\begin{equation}
\label{CPNE}
\left\{
\begin{array}{ccc}
u^{\prime\prime} (r)+\frac{d-1}{r}u^\prime(r)+f_\omega(u(r)) &= & 0 \\
u'(0) =  0, \qquad \lim\limits_{r\to+\infty} u(r) &= &0.
\end{array}
\right.
\end{equation}
We consider the following quantity 
\begin{equation}
    \label{Hamlt}
    H(r)=\frac{(u'(r))^2}{2}+F_\omega(u(r)),  \qquad \text{with} \qquad F_\omega(\xi)=\int_0^\xi f_\omega(t) \xd t.
\end{equation}
The function $H$ can be seen as an energy of the system \eqref{CPNE}, which is  decreasing. In fact,
$$H'(r)=-\frac{d-1}{r}(u^\prime(r))^2<0\mathrm.$$
Moreover, thanks to Lemma \ref{L1}, $u(r), u'(r) \to 0$
when $r\to \infty$ then 
$$\lim\limits_{r\to +\infty }H(r)=0$$
which implies that for all $r\in \xR^+ $, $H(r)\geq 0$.\\
On the other hand, by Lemma \ref{L3}, we have $u(0)<t^*$. Hence, using that $F_\omega$ is negative on $(0,t_0)$ and increasing on $(t_0,t^*)$ we have
$$H(0)=F_\omega(u(0))<F_\omega(t^*). $$
As a consequence, since $\omega\geq \frac{1}{\alpha+1}$ and $F_\omega(t^*)\le 0$, then $H(0)<0$, which contradicts $H(r)\geq0$ for all $r\in\xR^+$. 
\end{proof}
\subsection{Proof of Theorem \ref{Thrm_existence_uniqueness_quasi0}}
For the existence and the uniqueness of $\varphi$ the solution of the quasi-linear Schrödinger equation \eqref{SSNL}, we apply Theorem \ref{Thrm_existence_uniqueness},
which gives existence and uniqueness of the positive solution $u$ of \eqref{PSSL}. Next, we know that $ \varphi=\Lambda(u)$ where $\Lambda$ is the change of variables defined by \ref{defi_change_var}. Since $\Lambda$ is invertible from $(0,t^*)$ into $(0,1)$,
and, for all $x\in \xR^d$, $u(x)\in(0,t^*)$, this implies that there exists a unique $\varphi$ solution to \eqref{SSNL} such that $0<\varphi(x)<1$ for all $x\in \xR^d$.

Next, we prove the non-degeneracy of the unique positive solution $\varphi$. The linearised operator at our solution $\varphi$ is defined by
$$\mathcal{L}(\eta)=L_+ \eta_1+i L_-\eta_2   \qquad \text{with} \qquad \eta= \eta_1+i \eta_2 $$
\begin{align}
 L_+ w=&-\nabla \cdot\left(\frac{\nabla w}{1-\varphi^{2\alpha}}\right)- 2\alpha \nabla\cdot\left(\frac{\nabla\varphi}{(1-\varphi^{2\alpha})^2}\varphi^{2\alpha-1} \right)w +\alpha(2\alpha-1)\frac{\varphi^{2\alpha-2}|\nabla\varphi|^2}{(1-\varphi^{2\alpha})^2} w \nonumber\\
  &\quad + 4\alpha^2\frac{\varphi^{4\alpha-2}|\nabla\varphi|^2}{(1-\varphi^{2\alpha})^3} w -((2\alpha+1)\varphi^{2\alpha} -\omega)w
\end{align}
and
\begin{equation}
  L_- w = \omega w -\nabla \cdot \left(\frac{\nabla w}{1-|\varphi|^{2\alpha}} \right) +\alpha |\varphi|^{2\alpha-2}\frac{|\nabla\varphi|^2}{(1-|\varphi|^{2\alpha})^2}
w - |\varphi|^{2\alpha}w.
\end{equation}

\begin{prop}[Non degeneracy of the positive solution]
In $L^2(\xR^d)$, we have $$\mathrm{Ker}(L_+)=\mathrm{span}(\partial_{x_1}\varphi,\ldots,\partial_{x_d}\varphi) \qquad \text{and} \qquad \mathrm{ker}(L_-)=\mathrm{span}(\varphi).$$
\end{prop}
\begin{proof}
We note that the  first eigenvalue of the operator $L_-$ is simple
with a positive eigenfunction, this follows from \cite[Appendix, Theorem 2.4]{wolff2012topological},
  the fact that $\langle w, L_- w \rangle \ge  \langle|w|, L_-|w|\rangle$ and from Harnack’s inequality  \cite[Section 6.4, Theorem 5]{evans2009partial}.
   Since $L_-\varphi=0$ and since $\varphi$ is positive, $0$ is the first eigenvalue of $L_-$. 
   This implies that $\mathrm{ker}(L_-)=\mathrm{span}(\varphi)$.

   For $L_+$, we apply the change of variables $\varphi=\Lambda(u)$  introduced in Definition \ref{defi_change_var}, and by taking $v=\frac{w}{\Lambda'(u)}$ we obtain
   $$L_+w =\frac{-\Delta v-f'_\omega(u)v}{\Lambda'(u)} \qquad  \text{with} \qquad 0<\Lambda'(u)<1.$$
   So $w\in \mathrm{Ker}(L_+)$ if and only if $v\in \mathrm{Ker}(-\Delta-f'_\omega(u))$. Thanks to the result of Lewin and Rota Nodari  \cite{lewin2020double}, we have that 
  $ \mathrm{Ker}(-\Delta-f'_\omega(u))= \mathrm{span}(\partial_{x_1}u,\ldots,\partial_{x_d}u)$. Therefore  $$\mathrm{Ker}(L_+)= \mathrm{span}\{\partial_{x_1}\varphi,\ldots,\partial_{x_d}\varphi\}.$$
\end{proof}

Finally, using this non-degeneracy property, we prove the following proposition which completes the proof of Theorem \ref{Thrm_existence_uniqueness_quasi0}.

\begin{prop}
\label{prop_regularity}
For any $\omega\in (0,\omega^*)$, let $u_\omega$ be the unique positive solution of \eqref{PSSL}. The function $\omega \mapsto u_\omega$ is in $C^1((0,\omega^*), H^2_{\mathrm{rad}}(\xR^d))$.
\end{prop}

\begin{proof}
 The proof of this proposition is based on an implicit function argument. However, the non-linearity $f_\omega$ is not $C^1$ at $t^*$. To overcome this regularity issue, we consider the following different extension of $f_\omega$ for $|t|>t^*$:

\begin{equation}\label{smooth_non_linearity}
g_\omega(t)=
\begin{cases}
 \sqrt{1-\Lambda(t)^{2\alpha}}\left(\Lambda(t)^{2\alpha+1}-\omega\Lambda(t)\right)  &\text{ on } [0,t^*]\\   
   -\alpha(1-\omega)(t-t^*)   &\text{ on } (t^*,+\infty)\\
   -g_\omega (-t)      &\text{ on } (-\infty,0).\\
\end{cases}
\end{equation}

For any $\alpha>0$, the function $g_\omega$, defined in this way, is in $ C^1(\xR)$, and its first derivative is bounded as follows
\begin{equation}
\label{f'_smooth}
    |g_\omega'(t)|\leq \max\{1-\omega,\omega\}\max\{1,\alpha\} + 2 \alpha 
\end{equation}
for all $t\in \xR$. As a consequence, $g_\omega$ is a Lipschitz function on $\xR$. If $\alpha\ge \tfrac{1}{2}$, then $g_\omega$ is in $C^2(\xR)$, the second derivative satisfies  
\begin{equation}
\label{f''_smooth}
    |g_\omega''(t)|\leq 2\alpha (3+ (1+\alpha)\omega+ 8\alpha).
\end{equation}
and $g_\omega'$ is a Lipschitz function on $\xR$. If $0<\alpha<\tfrac{1}{2}$, the second derivative of $g_\omega$ is well-defined for any $t\in \xR\setminus\{0\}$ but $\lim_{t\to 0^{\pm}} g_\omega''(t)=\pm \infty$. Hence, $g_\omega'$ is not Lipschitz close to $0$. However, using the explicit definition of $f'_\omega$ given in Proposition \ref{prop_f}\ref{Prop_f2} and the properties of $\Lambda$ given in Proposition \ref{prop_change_var}, we deduce that $g_\omega'\in C^{0,2\alpha}(\xR)$. 
Now, let
\[
\begin{aligned}
    \Phi :\xR \times W^{2,q}(\xR^d) \cap H^2_\mathrm{rad}(\xR^d) &\to L^q(\xR^d)\cap L^2_\mathrm{rad}(\xR^d)\\
    (\omega, u) &\mapsto  -\Delta u - g_\omega(u)
\end{aligned}
\]
with $q\in \xN$, $q>\frac{d}{2}$.

For any $\omega_0\in (0,\omega^*)$, we know, thanks to Theorem \ref{Thrm_existence_uniqueness}, that there exists $u_{\omega_0}$ such that $\Phi(\omega_0,u_{\omega_0})=0$. Hence, we want to apply the implicit function theorem  \cite[Theorem 9.28]{rudinprinciples} 
on the function $\Phi$ in a neighbourhood of $(\omega_0,u_{\omega_0})$. Therefore, we need the function $\Phi$ to be  $C^1$ from $\xR\times W^{2,q}(\xR^d) \cap H^2_\mathrm{rad}(\xR^d)$ to $L^q(\xR^d)\cap L^2_\mathrm{rad}(\xR^d)$ and $D_u \Phi(\omega_0,u_{\omega_0})$ to be an isomorphism from $W^{2,q}(\xR^d) \cap H^2_\mathrm{rad}(\xR^d)$ to $L^q(\xR^d)\cap L^2_\mathrm{rad}(\xR^d)$. 

We start by proving that $\Phi$ is continuous from $\xR \times W^{2,q}(\xR^d) \cap H^2_\mathrm{rad}(\xR^d)$ to $L^q(\xR^d)\cap L^2_\mathrm{rad}(\xR^d)$.  Let $s\in\{2,q\}$ and let $(\omega_n,u_n)\to (\omega,u)$ in $\xR \times W^{2,q}(\xR^d) \cap H^2_\mathrm{rad}(\xR^d)$. We have 
\begin{align*}
    \|\Phi(\omega_n,u_n)&-\Phi(\omega,u)\|_{L^s(\xR^d)}\\
    &\leq \|\Delta u_n-\Delta u\|_{L^s(\xR^d)}+\|g_{\omega_n}(u_n)-g_\omega(u)\|_{L^s(\xR^d)}\\
    &\leq \| u_n- u\|_{W^{2,s}(\xR^d)}+\|g_{\omega_n}(u_n)-g_{\omega_n}(u)\|_{L^s(\xR^d)}+\|g_{\omega_n}(u)-g_\omega(u)\|_{L^s(\xR^d)}\\
    &\underbrace{\leq }_{\eqref{f'_smooth},\ref{prop_change_var}\ref{prop_change_var4}}\| u_n- u\|_{W^{2,s}(\xR^d)}+C\|u_n-u\|_{L^s(\xR^d)}+|\omega_n-\omega|\|u\|_{L^s(\xR^d)}.
\end{align*}

Next, we want to show that 
\[
\begin{aligned}
    D_u \Phi(\omega,u):W^{2,q}(\xR^d) \cap H^2_\mathrm{rad}(\xR^d) &\to L^q(\xR^d)\cap L^2_\mathrm{rad}(\xR^d)\\
    v &\mapsto  -\Delta v - g'_\omega(u)v,
\end{aligned}
\]
and 
\[
\begin{aligned}
   D_\omega \Phi(\omega,u):\xR &\to L^q(\xR^d)\cap L^2_\mathrm{rad}(\xR^d)\\
    h &\mapsto -\partial_\omega g_\omega(u)h.
\end{aligned}
\]
are continuous.

Let $s\in\{2,q\}$ and let $(\omega_n,u_n)\to (\omega,u)$ in $\xR \times W^{2,q}(\xR^d) \cap H^2_\mathrm{rad}(\xR^d)$.
For any $v\in W^{2,q}(\xR^d) \cap H^2(\xR^d)$ such that $\|v\|_{  W^{2,q}(\xR^d) \cap H^2(\xR^d)}=1$, we have 
\begin{align*}
    \|D_u \Phi(\omega_n,u_n)v &- D_u \Phi(\omega,u)v\|_{L^s(\xR^d)}= \|(g'_{\omega_n}(u_n)-g'_\omega(u))v\|_{L^s(\xR^d)}\\
    &\leq \|(g'_{\omega_n}(u_n)-g'_{\omega_n}(u))v\|_{L^s(\xR^d)}+\|g'_{\omega_n}(u)-g'_\omega(u)\|_{L^s(\xR^d)}\|v\|_{L^{\infty}(\xR^d)}\\
    &\leq \|(g'_{\omega_n}(u_n)-g'_{\omega_n}(u))v\|_{L^s(\xR^d)}+|\omega_n-\omega|((1+\alpha)\|u\|_{L^s(\xR^d)}+1)\|v\|_{L^{\infty}(\xR^d)}.
\end{align*} 
If $2\alpha\ge 1$, then $g_\omega'$ is a Lipschitz function and 
\begin{align*}
    \|(g'_{\omega_n}(u_n)-g'_{\omega_n}(u))v\|_{L^s(\xR^d)}&\le \|(g'_{\omega_n}(u_n)-g'_{\omega_n}(u))\|_{L^s(\xR^d)}\|v\|_{L^{\infty}(\xR^d)}\\
    &\le C\|u_n-u\|_{L^s(\xR^d)}\|v\|_{L^{\infty}(\xR^d)}
\end{align*}
for some positive constant $C$. If $0<2\alpha< 1$, then $f_\omega'$ is a $C^{0,2\alpha}$ function and, using Hölder inequality, we have
\begin{align*}
    \|(g'_{\omega_n}(u_n)-g'_{\omega_n}(u))v\|_{L^s(\xR^d)}&\le C\||u_n-u|^{2\alpha}|v|\|_{L^s(\xR^d)}\\
    &\le C\left(\int_{\xR^d}|u_n-u|^s\right)^{\frac{2\alpha}{s}}\left(\int_{\xR^d}|v|^{\frac{s}{1-2\alpha}}\right)^{\frac{1-2\alpha}{s}}\\
    &\le C\|u_n-u\|^{2\alpha}_{L^s(\xR^d)}\|v\|^{{2\alpha}}_{L^{\infty}(\xR^d)}\|v\|^{{1-2\alpha}}_{L^{s}(\xR^d)}.
\end{align*}

To prove the continuity of $D_\omega \Phi (\omega,u)$, we first show that the function $t\mapsto \partial_\omega g_\omega(t)$ is Lipschitz continuous. Indeed, from \eqref{smooth_non_linearity}, we deduce that
\begin{equation*}
\partial_\omega g_\omega(t)=
\begin{cases}
 -\Lambda'(t)\Lambda(t)  &\text{ on } [0,t^*]\\   
   \alpha(t-t^*)   &\text{ on } (t^*,+\infty)
\end{cases}
\end{equation*}
and $\partial_\omega g_\omega(t)=-\partial_\omega g_\omega (-t)$ for $t\in (-\infty,0)$.
A direct computation shows that the function $t\mapsto \partial_\omega g_\omega(t)$ is of class $C^1$ with bounded first derivative. 

As a consequence,
\begin{align*}
    \|D_\omega \Phi (\omega_n,u_n)h- D_\omega \Phi(\omega,u)h\|_{L^2(\xR^d)}&=  |h|\|\partial_{\omega} g_{\omega_n}(u_n)- \partial_{\omega} g_{\omega}(u)\|_{L^s(\xR^d)}\\
    &\le C|h| \| u_n- u\|_{L^s(\xR^d)}.
\end{align*}  
for some positive constant $C$ that does not depend on $\omega_n$ and $\omega$.
Next, note that $D_u \Phi(\omega_0,u_{\omega_0})= -\Delta - g_{\omega_0}'(u_{\omega_0})= \mathcal{L}_{\omega_0}$, where $\mathcal{L}_{\omega_0}$ is the real part of the linearised operator at $u_{\omega_0}$. By Remark \ref{Remark_Linearized_op},
$$\mathrm{ker}(\mathcal{L}_{\omega_0|\mathrm{rad}})=\{0\},$$
which implies that $D_u \Phi (\omega_0,u_{\omega_0})$ is injective into $L^q(\xR^d) \cap L^2_\mathrm{rad}(\xR^d)$.
It remains to show that the mapping is surjective. Let $h \in L^q(\xR^d)\cap L^2_\mathrm{rad}(\xR^d)$. We have to prove that there exists $v\in  W^{2,q}(\xR^d) \cap H^2(\xR^d)$ such that 
$$\mathcal{L}_{\omega_0}v=(-\Delta +\omega_0)v +V_{\omega_0}v =h\qquad   \text{with }  \qquad   V_{\omega_0} = -g_{\omega_0}'(u_{\omega_0})-\omega_0$$
or, equivalently,
$$ \left(\operatorname{Id} +(-\Delta +\omega_0)^{-1}V_{\omega_0}\right)v =(-\Delta +\omega_0)^{-1}h\mathrm.$$
Hence, it is sufficient to show that the operator $(-\Delta +\omega_0)^{-1}V_{\omega_0}:  W^{2,q}(\xR^d) \cap  H_{\mathrm{rad}}^2(\xR^d) \to  W^{2,q}(\xR^d) \cap  H_{\mathrm{rad}}^2(\xR^d)$ is compact. Since $(-\Delta +\omega_0)^{-1}: L^q(\xR^d)\cap L_{\mathrm{rad}}^2(\xR^d)\to W^{2,q}(\xR^d) \cap  H_{\mathrm{rad}}^2(\xR^d)$ is a bounded operator, it is enough to prove that $V_{\omega_0} :  W^{2,q}(\xR^d) \cap  H_{\mathrm{rad}}^2(\xR^d) \to L^q(\xR^d)\cap L_{\mathrm{rad}}^2(\xR^d)$ is a compact operator. 

To this goal, we start by noting that $V_{\omega_0}\in L^{\infty}(\xR^d)\cap L^{p}(\xR^d)$ for any $p\ge 1$, and vanishes at infinity exponentially because of the exponential 
decay of $u_{\omega_0}$. Then, let $\{u_n\}_n$ be a bounded sequence in 
$W^{2,q}(\xR^d) \cap H_{\mathrm{rad}}^2(\xR^d) 
\subset L^{\infty}(\xR^d)$ and let $s\in \{2,q\}$. On the one hand, thanks to Rellich--Kondrachov compactness theorem 
\cite[Chapter~5, Theorem~1]{evans2009partial}, there exists a subsequence $\{u_{n_k}\}_k$ that converges to $u$ in $W^{1,s}(B_{R}(0))$. As a consequence, since
\begin{align}\label{estim_V_0}
    \|V_{\omega_0}u_{n_k}-V_{\omega_0}u\|_{L^s(B_{R}(0))}\le \|V_{\omega_0}\|_{L^{\infty}(\xR^d)}\|u_{n_k}-u\|_{L^s(B_{R}(0))},
\end{align}
we deduce that $V_{\omega_0}u_{n_k}$ converges to $V_{\omega_0}u$ in $L^{q}(B_{R}(0)) \cap L_{\mathrm{rad}}^2(B_{R}(0))$. On the other hand, for any $k$,
\begin{align*}
    \|V_{\omega_0}u_{n_k}\|_{L^s(\xR^d\setminus B_{R}(0))}
    \le \|V_{\omega_0}\|_{L^\infty(\xR^d\setminus B_{R}(0))} \|u_{n_k}\|_{L^s(\xR^d)}\le C \|V_{\omega_0}\|_{L^\infty(\xR^d\setminus B_{R}(0))} 
\end{align*}
for some positive constant $C$ that does not depend on $k$.
This implies that
\begin{equation}
\label{V_conv}
    \|V_{\omega_0} u_{n_k}\|_{L^{s}(\xR^d \setminus B_R(0))}
    \xrightarrow[R \to +\infty]{} 0
\end{equation}
uniformly in $k$. Finally, thanks to \eqref{V_conv} and \eqref{estim_V_0}, we conclude that
\[
V_{\omega_0}: W^{2,q}(\xR^d) \cap H_{\mathrm{rad}}^2(\xR^d) \mapsto L^{q}(\xR^d) \cap L_{\mathrm{rad}}^2(\xR^d)
\]
is compact. 
  
By the Fredholm alternative, this implies that $\mathcal{L}_{\omega_0}$ is surjective from $W^{2,q}(\xR^d) \cap H_{\mathrm{rad}}^2(\xR^d)$ onto 
$L^q(\xR^d) \cap L^2_{\mathrm{rad}}(\xR^d)$.

Hence, we can apply the implicit function theorem and prove that there exists a $C^1$ function $\omega\mapsto u(\omega)$ defined on a neighbourhood of $\omega_0$, such that $\Phi(\omega,u(\omega))=0$. In other words $u(\omega)$ is a solution of 
$$
-\Delta u(\omega) = g_\omega (u(\omega)).
$$

To conclude, we have to show that $u_\omega=u(\omega)$. For this, it suffices to show that $u(\omega)$ is positive and $u(\omega)\le t^*$.

For each $\omega$, the operator $-\Delta - \frac{g_\omega(u(\omega))}{u(\omega)}$
has a zero eigenvalue, with eigenfunction $u(\omega)$. Moreover, when $ \omega = \omega_0$, zero is an isolated simple eigenvalue at the bottom of the spectrum. Thus, the zero eigenvalue of  $-\Delta - \frac{g_\omega(u(\omega))}{u(\omega)}$ must also be at the bottom of the spectrum for $|\omega - \omega_0|$ small. Then, thanks to \cite[Theorem XIII 46]{reed1978iv}, $u(\omega)$ is positive. Finally, thanks to Lemma \ref{lem_properties}(\ref{L3}), $u(\omega)< t^*$.
\end{proof}

\section{ On the limit $\omega \to \omega^*$, proof Theorem \eqref{Asymptotic_omegastar}}
\label{sec:limit_omega*}

In this Section we will prove Theorem \ref{Asymptotic_omegastar}, we investigate the asymptotic behaviour of the solution to \eqref{SSNL} and its mass as \(\omega \to \omega^*\), that we give in the following theorem.

\begin{thrm}
\label{thrm_Global_convergence_phi}
Let $\varphi_\omega$ be the unique  positive, radial solution  to \eqref{SSNL} and $t^*= \mathrm{min}\{t>0, \, \Lambda(t)=1\}$, with $\Lambda$ the change of variables given by Definition \ref{defi_change_var}. We define $R_\omega$ such that 
$u_\omega(R_\omega) = \frac{t^*}{2}$. For $\omega$ close enough to $\omega^*$,
we have the uniform convergence
\begin{equation}
\lim_{\omega\to\omega^*} \left\| \varphi_{\omega} - \Lambda( v_{\omega^*}(\cdot - R_{\omega}) )\right\|_{L^\infty(\xR^+)} = 0 \label{eq:uniform-convergence_phi}
\end{equation}
 where $v_{\omega^*}$ is the solution of \begin{equation}
\label{v_*}
\left\{
\begin{aligned}
v_{\omega^*}''(r)+&f_{\omega^*}(v_{\omega^*}(r))= 0\\
v_{\omega^*}(+\infty)&=0 \\ 
v_{\omega^*}(-\infty)&= t^* \\
v_{\omega^*}(0)&=\frac{t^*}{2} .\\ 
\end{aligned}
\right. 
\end{equation}
Also we have  the following local convergence. For all $R>0$ there exists $K>0$ so that $\frac{\omega^*}{2}<\omega<\omega^*$
\begin{equation}
\label{local_convergence_phi}
 \sup_{0<r<R}|\varphi_\omega(r)-1|\leq K|\omega-\omega^*|^{\frac{1}{2}} .
\end{equation}
Moreover, define  the mass of the solution to the problem \eqref{SSNL} as follows:
\begin{equation}
\label{Mass}
    M(\omega):=\|\varphi_\omega(x)\|^2_{L^2(\xR^d)},
\end{equation}
we have 
\begin{equation}
\label{Mass_assymp}
    M(\omega) =|\mathbb{S}^{d-1}| \frac{(2\sqrt{2})^{d}(d-1)^d}{d(\omega^*-\omega)^d} \left(\int_0^{t^*}|F_{\omega^*}(s)|^{\frac{1}{2}}\xd s\right)^d + o\left(\frac{1}{(\omega^*-\omega)^d}\right)
\end{equation}
where $F_{\omega}$ is given by  \ref{prop_f}\ref{Prop_f5}.
\end{thrm}

In order to prove this result, we  first study the behaviour of  $u$ the solution of \eqref{PSSL} when $  \omega$ converges towards the critical value $\omega^*=\frac{1}{1+\alpha}$ with $\alpha$ the power of the non-linearity of \eqref{SSNL}.

\begin{thrm}
\label{thrm_Global_convergence}
We define $R_\omega$ by 
$u_\omega(R_\omega) = \frac{t^*}{2}$ and let $v_{\omega^*}$ be the unique radial solution to the limiting problem \eqref{v_*}. We have the following uniform convergence
\begin{equation}
\lim_{\omega\to\omega^*} \left\| u_{\omega} - v_{\omega^*}(\cdot - R_{\omega}) \right\|_{L^\infty(\xR^+)} = 0 \label{eq:uniform-convergence}
\end{equation}
and the convergence of the derivatives
\begin{equation}
\lim_{\omega\to\omega^*} \left\| u'_{\omega} - v'_{\omega^*}(\cdot - R_{\omega}) \right\|_{L^p(\xR^+)} = 0 \label{eq:derivative-convergence}
\end{equation}
in $L^p(\xR^+)$ for all $1 \leq p \leq \infty$.
\end{thrm}

This theorem and its proof are similar to the results obtained in \cite{lewin2020double} for a double-power Schrödinger equation. In order to demonstrate Theorem \ref{thrm_Global_convergence}, we will require certain lemmas that contribute to establish the result.

\begin{lem}
\label{local_convergence}
For $0<\omega<\omega^*$, let $u_\omega$ be the unique positive, radial solution of \eqref{PSSL}. There exist $K>0$ so that for all $R>0$ and $\frac{\omega^*}{2}<\omega<\omega^*$
 $$ \sup_{r<R}|u_\omega(r)-t^*|\leq (R+1)K|\omega-\omega^*|^{\frac{1}{2}}\mathrm.$$
\end{lem}

\begin{proof}
We start with the ODE \eqref{CPNE}:  
\begin{equation}
\label{ODE}
u_\omega''(r)+ \frac{d-1}{r} u_\omega'(r)+ f_\omega (u_\omega(r))=0.
\end{equation}
Multiplying \eqref{ODE} by $u_\omega'(r)$ we get
 \begin{equation}
 \label{eq_conv_loc}
  \frac{ (u_\omega'(r))^2}{2} +(d-1)\int_0^r  \frac{ (u_\omega'(s))^2}{s} \xd s + F_\omega(u_\omega(r)) =F_\omega(u_\omega(0)).
 \end{equation}
  Evaluating at $r = +\infty$, knowing that $\lim\limits_{r\to \infty} u_\omega'(r)=0$ and $F_\omega(0)=0$, this gives
 \begin{equation}
 \label{eq_conv_1}
     (d-1)\int_0^\infty \frac{ (u_\omega'(s))^2}{s} \xd s  =F_\omega(u_\omega(0)).
 \end{equation}
On the other hand from \ref{prop_f},\ref{Prop_f5} we have 
  \begin{equation}
 \label{inq_F_omega}
    F_\omega(t)  =F_{\omega^*}(t)+\Lambda^2(t) \frac{\omega^*-\omega}{2}\leq \frac{\omega^*-\omega}{2}
 \end{equation}
 since $F_{\omega^*}(t)\le 0$ for all $t\in \R$.
 Then, for $t=u_\omega(0)$, we have 
 \begin{equation}
 \label{inq_F_omega1}
    F_\omega(u_\omega(0))  =F_{\omega^*}(u_\omega(0))+\Lambda^2(u_\omega(0)) \frac{\omega^*-\omega}{2}\leq \frac{\omega^*-\omega}{2}.
 \end{equation}
  Therefore
  \begin{equation}
  \label{A}
      |u_\omega(r)-u_\omega(0)|\leq \int_0^r (u_\omega'(s)) \xd s \leq \frac{r}{\sqrt{2}}\left(\int_0^r \frac{(u_\omega'(s))^2}{s} \xd s \right)^{\frac{1}{2}} \leq \frac{r\sqrt{\omega^*-\omega}}{2\sqrt{d-1}}.
  \end{equation}
 It remains to prove that $u_\omega(0)$ converges to $t^*$. 
Let $t_\omega=\Lambda^{-1}((\frac{\omega}{\omega^*})^{\frac{1}{2\alpha}})$ be the zero of $F_\omega(t)$ see as a function of $\omega$.
We have 
$$t^*- t_\omega =\int_{\omega}^{\omega^*}\left(\Lambda^{-1}\left(\frac{s}{\omega^*}\right)^{\frac{1}{2\alpha}}\right)'\xd s.$$
Knowing that $u_\omega(0)\geq t_\omega  $, we get
\begin{align*}
    0<t^*- u_\omega(0) &\leq\int_{\omega}^{\omega^*} \left(\Lambda^{-1}\left(\frac{s}{\omega^*}\right)^{\frac{1}{2\alpha}}\right)'\xd s\leq\frac{1+\alpha}{2\alpha} \int_{\omega}^{\omega^*}
\frac{ (\frac{s}{\omega^*})^{\frac{1}{2\alpha}-1}}{\sqrt{1-(\frac{s}{\omega^*})}}   \xd s\leq C_{\alpha,\omega^*} \int_{\omega}^{\omega^*} \frac{\sqrt{\omega^*}}{\sqrt{\omega^*-s}} \xd s.\\
\end{align*}
Finally we have:
\begin{equation}
\label{B}
    |t^*- u_\omega(0)|\leq C_{\alpha,\omega^*} |\omega-\omega^*|^{\frac{1}{2}}
\end{equation}
and then \eqref{A} and \eqref{B} give
\begin{equation*}
 \sup_{r<R}|u_\omega(r)-t^*|\leq (R+1)K|\omega-\omega^*|^{\frac{1}{2}}
\end{equation*}
where $K=\max\{C_{\alpha,\omega^*},\frac{1}{2\sqrt{d-1}}\}$.
Next, we note that the equation \eqref{eq_conv_loc} can be written in the form
\begin{equation*}
\frac{(u_\omega')^2}{2} + F_{\omega^*}(u_\omega)
= (d - 1) \int_r^\infty \frac{u_\omega'(s)^2}{s} \, \xd s
- \frac{1}{2} (\omega_* - \omega)\Lambda(u_\omega)^2.
\end{equation*}
Thanks to \eqref{eq_conv_1} and \eqref{inq_F_omega1} we obtain $$\left\|
(u_\omega')^2 + 2F_{\omega^*}(u_\omega)
\right\|_{L^\infty(\mathbb{R}_+)}
\le \omega_* - \omega.$$
Noticing that $|a^2 - b^2| \le \varepsilon^2$ implies 
$|a - b| \le \varepsilon$ whenever $a,b \ge 0$, we obtain 
\begin{equation}
\label{Conv_loc}
\left\|
u_\omega' + \sqrt{2|F_{\omega^*}(u_\omega)|}
\right\|_{L^\infty(\mathbb{R}_+)}
\le \sqrt{\omega_* - \omega}.
\end{equation}
\end{proof}

Now our aim is to prove that our solution converges uniformly on $\R^+$. To accomplish this, we will use the Arzelà–Ascoli argument. Once this convergence is established, we will proceed to prove that the limit corresponds to a solution of the Hamiltonian system \eqref{v_*}.
 We therefore introduce
$$v_\omega(r):=u_\omega(R_\omega + r),  \qquad r \in [-R_\omega, \infty)$$
where  $R_\omega$ is defined as 
$u_\omega(R_\omega) = \frac{t^*}{2}$, for $\omega$ close enough to $\omega^*$.

\begin{lem}
\label{conv_unif}
For all $R>0$, we have 
 $$\lim\limits_{\omega\to \omega^*} \|u_\omega-v_{\omega^*}(.-R_\omega)\|_{L^\infty(R_\omega-R,R_\omega+R)}=\lim\limits_{\omega\to \omega^*} \|u_\omega'-v_{\omega^*}'(.-R_\omega)\|_{L^\infty(R_\omega-R,R_\omega+R)}=0$$
  where $v_{\omega^*}$ is the unique solution to the limiting problem \eqref{v_*}.
\end{lem}

\begin{proof}
Based on \eqref{A}, we have that for $\frac{\omega^*}{2}<\omega<\omega^*$
\begin{equation}
    \label{estim_R}
     R_\omega\geq \frac{2\sqrt{d-1}(u_\omega(0)-\frac{t^*}{2})}{\sqrt{\omega^*-\omega}}.
\end{equation}
The function $v_\omega(r)$ satisfies the ODE
  \begin{equation}
  \label{ODEv}
       v_\omega''(r)+ \frac{d-1}{r+R_\omega} v_\omega'(r)+ f_\omega (v_\omega(r))=0.
  \end{equation}
Multiplying \eqref{ODEv} by $v_\omega'(r)$  and integrating on $[-R_\omega, r]$, we get
 \begin{equation*}
  \frac{ (v_\omega'(r))^2}{2} +(d-1)\int_{-R_\omega}^r  \frac{ (v_\omega'(s))^2}{R_\omega+s} \xd s + F_\omega(v_\omega(r)) =F_\omega(v_\omega(-R_\omega)).
 \end{equation*}
  Evaluating at $r = +\infty$, knowing that $\lim\limits_{r\to \infty} v_\omega'(r)=0$ and $F_\omega(0)=0$, and using \eqref{inq_F_omega1} gives us
 \begin{equation}
 \label{estim_1}
     (d-1)\int_{-R_\omega}^\infty \frac{ (v_\omega'(s))^2}{R_\omega+s} \xd s  =F_\omega(v_\omega(-R_\omega))\leq \frac{\omega^*-\omega}{2}.
 \end{equation}
 On the other hand, multiplying \eqref{ODEv} by $v_\omega'(r)$  and integrating on $[ r, \infty)$ for all $r\geq-R_\omega$, we get
 \begin{equation*}
  \frac{ (v_\omega'(r))^2}{2} + F_\omega(v_\omega(r))= (d-1)\int_r^\infty  \frac{ (v_\omega'(s))^2}{R_\omega+s} \xd s.
 \end{equation*}
 Using again \eqref{inq_F_omega}, one has
  \begin{equation}
  \label{B1}
  \frac{ (v_\omega'(r))^2}{2} + F_{\omega^*}(v_\omega(r))= (d-1)\int_r^\infty  \frac{ (v_\omega'(s))^2}{R_\omega+s} \xd s -\frac{1}{2} \Lambda^2(v_\omega(r))(\omega^*-\omega).
 \end{equation}
 Let $t_0=\Lambda^{-1}((\omega^*)^{\frac{1}{2\alpha}})$ be the minimum of $F_{\omega^*}$ and $F_{\omega^*}(t_0)= -\frac{\alpha}{2}(\omega^*)^{\frac{2\alpha+1}{\alpha}}$.

 Using $0<\Lambda(t)<1$ and  \eqref{estim_1}
 \begin{equation}
  \frac{ (v_\omega'(r))^2}{2}\leq(d-1)\int_r^\infty  \frac{ (v_\omega'(s))^2}{R_\omega+s} \xd s -F_{\omega^*}(t_0)  -\frac{1}{2} \Lambda^2(v(r))(\omega^*-\omega)\leq \frac{\omega^*-\omega}{2} +\frac{\alpha}{2}(\omega^*)^{\frac{2\alpha+1}{\alpha}}.
 \end{equation}
Therefore $v'$ is uniformly bounded, and more precisely
\begin{equation}
\label{born_v_prim}
    |v_\omega'(r)|\leq \sqrt{\alpha(\omega^*)^{\frac{2\alpha+1}{\alpha}}+\omega^*}.
\end{equation}
 This implies that the sequence  $\{v_\omega\}_{0<\omega<\omega^*}$ is uniformly equicontinuous.
 Also for all $R>0$ and all $r\in [-R, R]$, we have 
   
   $$ |v_\omega(r)-v_\omega(0)|\leq \int_0^r |v_\omega'(s)| \xd s \leq R\sqrt{\alpha(\omega^*)^{\frac{2\alpha+1}{\alpha}}+\omega^*}.$$
    Thanks to Ascoli-Arzela theorem there exists a subsequence still denoted by $\{v_\omega\}_{0<\omega<\omega^*}$ that converges uniformly. Let $v_{\omega^*}$ be the limit of $\{v_\omega\}_{0<\omega<\omega^*} $. \eqref{estim_R} and \eqref{ODEv} imply that $v_{\omega}''(r)$ is also uniformly bounded. As a consequence, up to a subsequence, ${v_\omega'}$ converges uniformly. Furthermore thanks to \eqref{ODEv}, ${v_\omega''}$ converges uniformly to $-f_{\omega^*}(v_{\omega^*})$. Hence $(v_{\omega^*}')'=-f_{\omega^*}(v_{\omega^*})$.
  In addition, $v_{\omega^*}$ satisfies
   \begin{itemize}
       \item $\lim\limits_{\omega\to \omega^*} \frac{d-1}{R_{\omega}+r}v_\omega'(r)=0$ (thanks to \eqref{born_v_prim} and \eqref{estim_R})
       \item  $v_{\omega^*}(0)=\lim\limits_{\omega\to \omega^*} v_\omega(0)= \lim\limits_{\omega\to \omega^*} u_\omega(R_\omega)=\frac{t^*}{2}$
    \end{itemize}
and, from \eqref{Conv_loc},
\[
\frac{v_{\omega^*}'}{\sqrt{2|F_{\omega^*}(v_{\omega^*})|}} = -1
\]
since $v_\omega' < 0$.
Therefore
\[
v_{\omega^*}(r) = \Theta^{-1}(r),
\qquad
 \Theta(v) = - \int_{\frac{t^*}{2}}^{v} \frac{\xd s}{\sqrt{2|F_{\omega^*}(s)|}}.
\]
Note that $ \Theta$ diverges logarithmically at $0$ and $t^*$, so that 
$v_{\omega^*}$ converges exponentially fast towards $t^*$ at $-\infty$ and $0$ at $+\infty$.

Finally, for all $R>0$
  $$\lim\limits_{\omega\to \omega^*} \|v_\omega-v_{\omega^*}\|_{L^\infty(-R,R)}=\lim\limits_{\omega\to \omega^*} \|v_\omega'-v_{\omega^*}'\|_{L^\infty(-R,R)}=0$$
   which means   
    $$\lim\limits_{\omega\to \omega^*} \|u_\omega-v_{\omega^*}(\cdot-R_\omega)\|_{L^\infty(R_\omega-R,R_\omega+R)}=\lim\limits_{\omega\to \omega^*} \|u_\omega'-v_{\omega^*}'(\cdot-R_\omega)\|_{L^\infty(R_\omega-R,R_\omega+R)}=0.$$
  \end{proof}
  Now that we have achieved locally uniform convergence on a moving compact subset of $\xR$, our aim is to extend this convergence on $\xR$. To accomplish this, we will establish exponential estimates for our solution, as provided by the following lemma.
\begin{lem}
\label{lem_exp_bound}
 We have the bounds
\begin{equation}
\label{exp_estim_u}
v_{\omega}(r) 
\begin{cases}
\leq Ce^{-c|r|} & \text{on } [0, \infty), \\
\geq t^* - Ce^{-c|r|} & \text{on } [-R_{\omega}, 0],
\end{cases}
\end{equation}
and
\begin{equation}
|v'_{\omega}(r)|+|v''_{\omega}(r)|\leq Ce^{-c|r|} \quad \text{on } [-R_{\omega}, \infty).
\end{equation}
for some $c>0$, $C > 0$ independent of $\omega \in [\frac{\omega^*}{2}, \omega^*]$.
\end{lem}

The proof is the same as that of \cite[Lemma 5.2]{lewin2020double}, so we omit it.

Now, we have all the necessary ingredients to prove the result of Theorem \ref{thrm_Global_convergence}.

\begin{proof}[Proof (Theorem \ref{thrm_Global_convergence})]
The exponential bounds from Lemma \ref{lem_exp_bound} give for $r \ge  -R_{\omega}$
$$|v'_{\omega}(r) - v'_{\omega^*}(r)| \leq Ce^{-c|r|}.$$
Hence, the dominated convergence theorem implies that, for all $1\leq p < +\infty $,
\begin{equation}
\lim_{\omega\to\omega^*} \int^{+\infty}_{-R_{\omega}} |v'_{\omega}(r) - v'_{\omega^*}(r)|^p \, \xd r = \lim_{\omega\to\omega^*} \int_0^{\infty} |u'_{\omega}(r) - v'_{\omega^*}(r - R_{\omega})|^p \, \xd r = 0.
\end{equation}
This gives the convergence of the derivatives in $L^p$, for all $1\leq p < +\infty $.
The convergence of the derivatives in $L^1$, and the fact that $u(R_\omega)=v_\omega(0) = v_{\omega^*}(0)$, gives 
\begin{equation}
   \int_0^{r} |u'_{\omega}(s+R_\omega) - v'_{\omega^*}(s)| \, \xd s \geq  |u_{\omega}(r+R_\omega) - v_{\omega^*}(r)|
\end{equation}
which implies the uniform convergence of $u_\omega$. Similarly, using the exponential bounds on $v''_\omega$, we obtain the $L^1$ convergence of $v''_\omega-v''_{\omega^*}$ on $(-R_\omega,+\infty)$ which leads to the uniform convergence of $ u'_{\omega} - v'_{\omega^*}(\cdot - R_{\omega})$ on $\R^+$.
\end{proof}
After establishing the convergence of \( u_\omega \) the solution of \eqref{PSSL}, we are now able to prove  Theorem \ref{thrm_Global_convergence_phi} for \( \varphi_\omega \) the solution of \eqref{SSNL} using the change of variables defined in Definition \ref{defi_change_var}.

The uniform convergence \eqref{eq:uniform-convergence_phi} of Theorem \ref{thrm_Global_convergence_phi} follows immediately from the fact that $\varphi_\omega(r)=\Lambda(u_\omega(r))$, the change of variables $\Lambda $ is Lipschitz by Proposition \ref{prop_change_var} and the uniform convergence of $u_\omega$ given by Theorem \ref{thrm_Global_convergence}. The local convergence \eqref{local_convergence_phi} of $\varphi_\omega$ can be proved thanks to Lemma \ref{local_convergence} 

To conclude the proof of the Theorem \ref{thrm_Global_convergence_phi}, it remains to prove the following lemma.

\begin{lem}
The mass of $\varphi_\omega$ the unique positive solution to \eqref{SSNL} behaves as follows
\begin{equation}
    M(\omega)=|\mathbb{S}^{d-1}| \frac{(2\sqrt{2})^{d}(d-1)^d}{d(\omega^*-\omega)^d} \left(\int_0^{t^*}|F_{\omega^*}(s)|^{\frac{1}{2}}\xd s\right)^d + o\left(\frac{1}{(\omega^*-\omega)^d}\right)
\end{equation}

when $\omega$ approaches $\omega^*$.
\end{lem}

\begin{proof}

Using the upper estimate \eqref{exp_estim_u}, the Lemma \ref{lem_properties}(\ref{L3}) and that $\varphi_\omega$ is radial decreasing we have
\begin{align*}
\int_0^{\infty} r^{d-1} \varphi_\omega(r)^2 \xd r= &\int_0^{R_\omega} r^{d-1} \varphi_\omega(r)^2 \xd r + \int_{R_\omega}^{\infty} r^{d-1} \varphi_\omega(r)^2 \xd r\\
&\leq \varphi_\omega(0)^2  \int_0^{R_\omega} r^{d-1} \xd r + \int_{R_\omega}^{\infty} r^{d-1} u_\omega(r)^2 \xd r\\
&\leq \frac{R_\omega^d}{d}+C\int_{0}^{\infty} (r+R_\omega)^{d-1} e^{-2cr} \xd r=  \frac{R_\omega^d}{d} + O(R_\omega^{d-1}).
\end{align*}
On the other hand, using Proposition \ref{prop_change_var}\ref{prop_change_var5}, Lemma \ref{lem_exp_bound}, and knowing that $v_\omega(r):=u_\omega(R_\omega + r)$ we have for all $r\in[0,R_\omega]$ 
\begin{align*}
    0\le 1-\varphi_{\omega}(r)=\Lambda(t^*)-u_\omega(r)\le (t^*-u_\omega(r))\le Ce^{-c|r-R_\omega|}.
\end{align*}
As a consequence,
\begin{align*}
\int_0^{\infty} r^{d-1} \varphi_\omega(r)^2 \xd r = &\int_0^{\infty} r^{d-1} ((1-\varphi_\omega(r))-1)^2 \xd r\\\ge  & \int_0^{R_\omega} r^{d-1}\xd r - 2 \int_0^{R_\omega} (1-\varphi_\omega(r)) r^{d-1}  \xd r
 +  \underbrace{\int_0^{R_\omega} (1-\varphi_\omega(r))^2 r^{d-1} \xd r}_{>0}\\
 \geq &\int_0^{R_\omega} r^{d-1}\xd r - 2 C\int^0_{-R_\omega} (r+R_\omega)^{d-1}e^{-c|r|} \xd r=\frac{R_\omega^d}{d} + O(R_\omega^{d-1}).
\end{align*}
Finally, 
\begin{equation}
    \label{Mass_Assymptotique}
    M(\omega)= |\mathbb{S}^{d-1}| \frac{R_\omega^d}{d} + O(R_\omega^{d-1}).
\end{equation}

To conclude, it remains to determine the asymptotic behaviour of $R_\omega$. This can be done exactly as in \cite[Lemma 5.5]{lewin2020double}, and we get

\begin{equation}
\label{R_assymptotique}
    R_\omega =\frac{(2\sqrt{2})(d-1)}{(\omega^*-\omega)} \int_0^{t^*}|F_{\omega^*}(s)|^{\frac{1}{2}}\xd s + o\left(\frac{1}{\omega^*-\omega}\right).
\end{equation}
By inserting equation \eqref{R_assymptotique} into equation \eqref{Mass_Assymptotique}, we immediately obtain \eqref{Mass_assymp}.

\end{proof}

\section{ On the limit $\omega \to 0$, proof Theorem \eqref{THRM_omega_0} }
 \label{sec:limit_0}

In this section, we prove Theorem \ref{THRM_omega_0} that describes the asymptotic behaviour of the positive solution \( \varphi_\omega \) to \eqref{SSNL} as \( \omega \to 0 \). The behaviour of \( \varphi_\omega \) in this regime depends crucially on the exponent of the non-linearity. Accordingly, we consider three distinct cases — sub-critical, critical, and super-critical — each leading to a different asymptotic profile. 

\subsection{Super-critical case}
\label{sec:SCC}

We examine, as \( \omega \) approaches 0, the behaviour of the unique positive solution \( \varphi_\omega \) to \eqref{SSNL} for $\alpha>\frac{2}{d-2}$. When $\omega=0 $ the equation \eqref{SSNL} becomes
\begin{equation}
\label{SSNL_0}
-\nabla \cdot \left(\frac{\nabla\varphi}{1-|\varphi|^{2\alpha}} \right) +\alpha |\varphi|^{2\alpha-2}\frac{|\nabla\varphi|^2}{(1-|\varphi|^{2\alpha})^2}
\varphi - |\varphi|^{2\alpha}\varphi = 0.
\end{equation}
Our aim is to show the following theorem.
\begin{thrm}
\label{Thrm_super_critical_phi}
Suppose $d\geq 3$ and $\alpha>\frac{2}{d-2}$. Then, as $\omega \to 0^+$, $\varphi_\omega$ the solution of \eqref{SSNL} converges in $\dot{H}^1(\mathbb{R}^d)$ to the unique positive, radial solution $\varphi_0$ of \eqref{SSNL_0}.
\end{thrm}
To this goal we consider the semi-linear problem \eqref{PSSL}  and  we use the variational characterisation as in  \cite{berestycki1983nonlinear}. More precisely, for any $\omega> 0$, the unique positive solution $u_\omega$ is obtained by considering the following minimisation problem
\begin{equation}
\label{Pb_optim_I_omega}
I_\omega= \inf \left\{\int_{\xR^d}|\nabla u|^2  \xd x :\,  u\in H^1(\xR^d),\, \frac{2d}{d-2}\int_{\xR^d} F_\omega (u) \xd x = 1\right\}
\end{equation}
with $$F_\omega(t)=\int_0^t f_\omega(s)\xd s = \frac{1}{2\alpha+2}\Lambda^{2\alpha+2}(t)-\frac{\omega}{2}\Lambda^2(t) .$$
where $\Lambda$ is the change of variables given by Definition \ref{defi_change_var}. In particular, let $v_\omega$ be a minimiser of \eqref{Pb_optim_I_omega}. There exists a Lagrange multiplier $\mu $ such that 
\begin{equation*}
    -\Delta v_\omega = \mu  f_\omega(v_\omega).
\end{equation*}
Thanks to Pohozaev identity
\begin{equation*}
    \int_{\xR^d}|\nabla v_\omega|^2 \,\xd x = \frac{2d}{d-2} \mu  \int_{\xR^d}F_\omega(v_\omega)\, \xd x
\end{equation*} 
 we have 
 \begin{equation*}
     \mu= I_\omega.
 \end{equation*}
Note that $I_\omega > 0$.
Then $v_\omega$ is a solution to 
\begin{equation*}
    -\Delta v_\omega =I_\omega f_\omega(v_\omega),
\end{equation*}
so that 
\begin{equation}
\label{rescaled_solut_v_omega}
    u_\omega(x)=v_\omega (I_\omega^{\frac{-1}{2}} x)\mathrm,
\end{equation}
with $u_\omega$ the solution to \eqref{PSSL}. 

When $\omega=0$, we have to deal with \eqref{SSNL_0}. In order to solve this equation, we proceed as before. We define $u$ by $\varphi=\Lambda(u)$ where $\Lambda$ is the change of variables given in Definition \ref{defi_change_var}. We remark that $\varphi_0$ is solution to \eqref{SSNL_0} if and only if $u_0$ is a solution to the following limiting problem 
\begin{equation}
    \label{Pb_omega_0}
    -\Delta u  = f_0(u), \text{ with }  f_0(t)=\Lambda'(t)\Lambda^{2\alpha+1}(t) \text { for } t\neq 0 \text{ and }\ f_0(0)=0.
\end{equation}
First, we will show that \eqref{Pb_omega_0} has a unique positive solution $u_0\in \dot{H}^1(\xR^d)$, and then we will prove that $u_\omega $ tends to $u_0$ in $\dot{H}^1(\xR^d)$.

Let us state the theorem of existence and uniqueness of a positive solution to \eqref{Pb_omega_0}.
\begin{thrm}
Let $d\geq 3$ and $\alpha>\frac{2}{d-2}$, the problem \eqref{Pb_omega_0} has a unique positive, radial decreasing solution $u_0\in \dot{H}^1(\xR^d)$. Moreover $u_0$ is a fast decaying function, meaning that
\begin{equation}
\label{estim_radial_u0}
 |x|^{d-2} u_0(|x|) \to c \geq 0 \qquad \text{when} \qquad  |x|\to \infty. 
\end{equation}
\end{thrm}

\begin{proof}
The existence of a positive solution to \eqref{Pb_omega_0} can be proved using the result of Berestycki and Lions   in \cite[Theorem 4]{berestycki1983nonlinear}. Similarly to the case $\omega>0$, the proof relies on the variational characterisation of the problem
\begin{equation}
\label{Pb_optim_I_0}
I_0= \inf \left\{\int_{\xR^d}|\nabla u|^2  \xd x :\,   u\in \dot{H}^1(\xR^d),\, \frac{2d}{d-2} \int_{\xR^d} F_0 (u) \xd x = 1\right\}
\end{equation}
with $$F_0 (t)=\int_0^t f_0(s)\xd s = \frac{1}{2\alpha+2}\Lambda^{2\alpha+2}(t).$$
The problem \eqref{Pb_optim_I_0} admits a positive, radial decreasing minimiser $v_0$. In particular there exists a Lagrange multiplier $\mu $ such that 
\begin{equation*}
    -\Delta v_0 = \mu  f_0(v_0).
\end{equation*}
Thanks to the Pohozaev identity 
\begin{equation*}
    \int_{\xR^d}|\nabla v_0|^2 \,\xd x = \frac{2d}{d-2} \mu  \int_{\xR^d}F_0(v_0)\, \xd x
\end{equation*} 
 we have 
\begin{equation*}
     \mu= I_0.
\end{equation*}
Note that $I_0> 0$. Finally, we define $u_0(x)= v_0(I_0^ {-\frac{1}{2}}x)$. Then $u_0$ is a positive radial decreasing solution to \eqref{Pb_omega_0}.

To prove the uniqueness, we want to use the result of Tang \cite{tang2001uniqueness}. The difficulty is that the non-linearity $f_0$ is not $C^1(\xR^+)$. In order to overcome this, we consider the following problem 
\begin{equation}
    \label{Pb_omega_0_tild}
    -\Delta u  = \tilde{f}_0(u)
\end{equation}
with
$$
\left\{
\begin{aligned}
&\tilde{f}_0(0)=0&\\
&\tilde{f}_0(t)=\Lambda'(t)\Lambda^{2\alpha+1}(t) &\text{if } 0<t\leq t^*\\ 
&\tilde{f}_0(t)=-\alpha(t-t^*) &\text{if } t> t^*
\end{aligned}
\right.
$$
where $t^*$ is defined by \eqref{t*}. The function $\tilde{f}_0$ is in $C^1((0,+\infty))$.

Our aim now is to show that all positive solutions to \eqref{Pb_omega_0} and \eqref{Pb_omega_0_tild} in $\dot{H}^1(\xR^d)$ are radial decreasing.  For this, we argue as in the proof of Lemma \ref{lem_properties}(\ref{L2}), using \cite[Theorem 4.1]{frank2013ground}.

Noticing that $0\leq \Lambda(t) \leq \min\{1,t\}$ for all $t\ge 0$, we have
\begin{align*}
    \frac{f_0(s)-f_0(t)}{s-t}=\int_0^1 f_0'(t+\sigma (s-t))d\sigma \leq (2\alpha+1) (\min \{1,s^{2\alpha}\}) 
\end{align*}
and
\begin{align*}
    \frac{\tilde{f}_0(s)-\tilde{f}_0(t)}{s-t}=\int_0^1 \tilde{f}_0'(t+\sigma (s-t))d\sigma \leq  (2\alpha+1) (\min \{1,s^{2\alpha}\}) .
\end{align*}
for any $s>t\ge 0$. Since  $\alpha>\frac{2}{d-2}$, 
we have that $u_0$, solution to \eqref{Pb_omega_0}, and $\tilde{u}_0$, solution to \eqref{Pb_omega_0_tild}, are in $\dot{H}^1(\xR^d)\cap L^{2d/d-2}(\xR^d)$. Then, all positive solutions of the  problem \eqref{Pb_omega_0} and the problem \eqref{Pb_omega_0_tild} are radial decreasing. 

Moreover, thanks to the maximum principle, we can prove that all positive solutions to \eqref{Pb_omega_0} and \eqref{Pb_omega_0_tild}  satisfy that 
$$\|u_0\|\leq t^*  \qquad \text{and} \qquad   \|\tilde{u}_0\|\leq t^*.$$
This implies that $\tilde{u}$ is a positive solution to \eqref{Pb_omega_0_tild} if and only if $u$ is a positive solution to \eqref{Pb_omega_0}, and it is enough to prove the uniqueness of positive radial solutions to \eqref{Pb_omega_0_tild}.

To this aim, we use the result of Tang \cite{tang2001uniqueness}. 
 
We have 
$$ \tilde{f}_0 \in C^1([0, \infty)),\ \tilde{f}_0(0) = 0,\text{ and } \tilde{f}_0(s) > 0 \text{ for $s$ near $0$,}
$$ 
and we want to prove that $g(s)=\frac{s f_0'(s)}{f_0(s)}$ is non-increasing on $(0,t^*)$. We have  
\begin{equation}
    \label{g}
    g(s)=\frac{(2\alpha+1)s\Lambda'(s)}{\Lambda(s)}- \frac{\alpha s \Lambda^{2\alpha}(s)}{\Lambda'(s)\Lambda(s)}.
\end{equation}
we have
\begin{align*}
g'(s)&=\frac{(2\alpha+1)}{\Lambda^2(s)}(\Lambda'(s)\Lambda(s)+(1-\alpha) s\Lambda^{2\alpha}(s)- s)\\
      &-\frac{\alpha\Lambda^{2\alpha}(s)}{(\Lambda'(s)\Lambda(s))^2} \left( \Lambda'(s)  \Lambda(s)+ (2\alpha-1)s(\Lambda'(s))^2 + \alpha s \Lambda^{2\alpha}(s) \right).
\end{align*}
For any $s\in (0,t^*)$, let $h(s)= \Lambda'(s)\Lambda(s)+(1-\alpha) s\Lambda^{2\alpha}(s)- s$. We have $\lim_{s\to 0^+}h(s)=0$ and
$$h'(s)=-2\alpha \Lambda^{2\alpha-1}(s)(\Lambda(s)+ (\alpha-1)s \Lambda'(s))  \leq 0 $$ for any $\alpha>0$ thanks to Proposition \ref{prop_change_var}\ref{prop_change_var4}. Similarly, 
$$\Lambda(s)+ (2\alpha-1)s\Lambda'(s)\ge 0$$ for all $s\in (0,t^*)$. As a consequence, $g'(s)\leq 0$ for all $s\in (0,t^*)$ and $g$ is non-increasing on $(0,t^*)$.

Next, we have to check that $-\infty=\lim\limits_{s \to (t^*)^-} g(s)<\frac{d+2}{d-2}<\lim\limits_{s \to 0^+} g(s)$. On the one hand, $\lim\limits_{s \to 0^+} g(s)= 2\alpha+1$. On the other hand, $\lim\limits_{s \to (t^*)^-} g(s)=-\infty$. Since we are in the super-critical case, this condition is satisfied.

Finally, to apply Tang's result, it remains to show that $u_0$ is a fast decaying function, meaning that
$$ |x|^{d-2} u_0(|x|) \to c \geq 0 \qquad \text{when} \qquad  |x|\to \infty. $$

First we have that $0 \leq F_0(t) \leq C |t|^{\frac{d}{d-2}}$ for some constant $C$, it is continuous and $F \not\equiv 0$. Since $u_0$ is a positive minimiser of $I_0$, using the result of Flucher and Müller \cite[Theorem 5]{flucher1998radial}, we have
\[
  \left(1 - O(|x|^{-2})\right) C K(|x|) \leq u_0(|x|) \leq C K(|x|) ,
\]
with
\[
K(|x|) := \frac{1}{(d - 2) |\mathbb{S}^{d-1}| |x|^{d-2}}, \quad \text{and}\quad C= \sqrt{2 \frac{(d - 1)}{d I_0} \int_{\mathbb{R}^d} \frac{F_0(u_0)} {K(|x|)}\,\xd x}
\]
which gives the decay of $u_0$ at infinity. 

As a consequence, we apply \cite[Theorem 2]{tang2001uniqueness} and conclude that $u_0$ is the unique positive solution to \eqref{Pb_omega_0}.

\end{proof}

Now, we can state the result of the convergence of $u_\omega$.

\begin{thrm}
\label{Thrm_super_critical}
Suppose that $d\geq 3$ and $\alpha>\frac{2}{d-2}$. Then, as $\omega \to 0^+$, $u_\omega$ the solution of \eqref{PSSL} converges in  $\dot{H}^1(\mathbb{R}^d)\cap L^q(\mathbb{R}^d)$ for any $q\ge \frac{2d}{d-2}$ to the unique positive radial solution $u_0$ of \eqref{Pb_omega_0}.
\end{thrm}

Before starting the proof, we give here a technical lemma that we will need after. The proof of this lemma can be found for instance in \cite[Lemma 4.1]{GenRot-24}.
\begin{lem}
\label{case_3_4}
 Let \( d \in \{3, 4\} \). Consider a function \( z \in C(\mathbb{R}^d) \), a number \( \rho > 0 \), and \( \eta_\rho \in C_0^{\infty}(\mathbb{R}^d) \), a cut-off function such that \( \eta_\rho \equiv 1 \) on \( B_\rho(0) \) and \( \eta_\rho \equiv 0 \) on \( \mathbb{R}^d \setminus B_{2\rho}(0) \). Suppose there exists a constant \( C > 0 \) such that
\[
\lim_{|x| \to \infty} \frac{|z(x)|}{|x|^{-d+2}} = C.
\]

Then we have the following asymptotic as \( \rho \to \infty \):
\[
\int_{\mathbb{R}^d} | \eta_\rho z |^2 \, \xd x = 
\begin{cases}
O(\rho) & \text{if } d = 3, \\
O(\log(\rho)) & \text{if } d = 4.
\end{cases}
\]
\end{lem}

\begin{proof}[Proof Theorem \ref{Thrm_super_critical}]
 The first step of the proof is to show that $\lim\limits_{\omega\to 0}I_\omega=I_0$. To this goal we consider the following functional
$$J_\omega(u)=\frac{\int_{\xR^d}|\nabla u|^2 \xd x}{\left(\frac{2d}{d-2} \int_{\xR^d} F_\omega(u) \xd x\right)^{\frac{d-2}{d}}}\qquad \text{and} \qquad K_\omega = \{u\in H^1(\xR^d):   \int_{\xR^d} F_\omega(u) \xd x >0 \},$$
and we show that  
\begin{equation}
\label{Pb_optim_J_omega}
\inf\limits_{v\in K_\omega} J_\omega(v)=I_\omega.
\end{equation}

Let $v\in H^1(\xR^d)$ be the minimiser of \eqref{Pb_optim_I_omega}, then
$$ \int_{\xR^d} |\nabla v|^2\,\xd x =I_\omega \text{ and } \int_{\xR^d} F_\omega(v)\xd x = \frac{d-2}{2d}. $$
This implies that $v\in K_\omega$ and $\inf\limits_{w\in K_\omega}J_\omega (w)\leq J_\omega(v)= I_\omega$. 

By contradiction, we suppose that $\inf\limits_{w\in K_\omega}J_\omega (w)< I_\omega$. Then, there exists  \( v \in K_\omega \) such that \( J_\omega(v) < I_\omega \). For all $x\in \xR^d$, let
$v_\lambda (x)=v(\lambda x)$ be a dilatation of $v$ with $\lambda = \left(\frac{2d}{d-2}  \int_{\xR^d} F_\omega(v) \,\xd x \right)^{\frac{1}{d}}$,so that $v_\lambda$ satisfies 
$\int_{\xR^d} F_\omega(v_\lambda) \,\xd x  = \frac{d-2}{2d}$.
Then 
\[I_\omega \leq \int_{\xR^d} |\nabla v_\lambda|^2 \, \xd x = J_\omega(v_\lambda) = J_\omega(v) < I_\omega\]
since $J_\omega$ is invariant under dilatation. This leads to a contradiction. Thus, $\inf\limits_{v\in K_\omega} J_\omega(v)=I_\omega$. 
Similarly we consider the following functional
\begin{equation*}
  J_0(u)=\frac{\int_{\xR^d}|\nabla u|^2 \xd x}{\left(  \frac{2d}{d-2}\int_{\xR^d} F_0(u) \xd x\right)^{\frac{d-2}{d}}} \text{ and } K_0 = \{u\in \dot{H}^1(\xR^d):  \int_{\xR^d} F_0(u) \xd x >0 \}.
\end{equation*}
 Using the same argument as before we prove that  
 \begin{equation}
 \label{Pb_optim_J_0}
 \inf\limits_{v\in K_0} J_0(v)=I_0.
 \end{equation}
Next, we  prove that
 \begin{equation}
 \label{convergence_minima}
     1\le \frac{I_\omega}{I_0}< 1+o(1).
 \end{equation}
First notice that, for all $\omega \in (0,\omega^*)$ we have $F_\omega(u)\leq F_0(u)$. So we can use the minimiser \( v_\omega \) of $I_\omega$ as a test function for \( J_0 \), then we have
\begin{align*}
    I_0 \leq J_0(v_\omega) = \frac{ \int_{\mathbb{R}^d}|\nabla v_\omega|^2  \xd x }{\left( \frac{2d}{d-2}\int_{\mathbb{R}^d} F_0(v_\omega) \xd x \right)^{\frac{d-2}{d}} }&= \frac{ \int_{\mathbb{R}^d}|\nabla v_\omega|^2  \xd x }{\left(\frac{2d}{d-2} \int_{\mathbb{R}^d} F_\omega(v_\omega) \xd x \right)^{\frac{d-2}{d}} } \left(\frac{ \int_{\mathbb{R}^d} F_\omega(v_\omega) \xd x}{ \int_{\mathbb{R}^d} F_0(v_\omega) \xd x} \right)^{\frac{d-2}{d}}\\
    &\leq J_\omega(v_\omega)= \int_{\mathbb{R}^d}|\nabla v_\omega|^2  \xd x = I_\omega.
\end{align*}
For $d \geq 5$, we have that $v_0 \in L^2(\xR^d)$. Hence $v_0\in H^1(\xR^d)$. Moreover, for $\omega > 0$ small enough, $\int_{\mathbb{R}^d} F_\omega(v) \xd x =\int_{\mathbb{R}^d} (F_0(v)-\frac{\omega}{2} \Lambda^2(v)) \,\xd x >0 $. Since $v_0 \in K_\omega$, we can use $v_0$ as a test function for $J_\omega$. This gives:

\begin{align*}
 I_\omega \leq J_\omega(v_0)& = \frac{ \int_{\mathbb{R}^d} |\nabla v_0|^2 \xd x}{  \left(\frac{2d}{d-2}\int_{\mathbb{R}^d} F_\omega(v_0) \xd x\right)^{1 - \frac{2}{d}}}\\
 &=  \frac{ \int_{\mathbb{R}^d}|\nabla v_0|^2 \xd x}{ \left(\frac{2d}{d-2}\int_{\mathbb{R}^d} F_0(v_0) \xd x\right)^{1 - \frac{2}{d}} }\cdot \left(\frac{\int_{\mathbb{R}^d} F_0(v_0) \xd x} {\int_{\mathbb{R}^d} F_\omega(v_0)\xd x }\right)^{1 - \frac{2}{d}} = I_0 \cdot \left( \frac{\int_{\mathbb{R}^d} F_0(v_0) \,\xd x}{\int_{\mathbb{R}^d} F_\omega(v_0)\,\xd x} \right)^{1 - \frac{2}{d}}.
 \end{align*}

Then we have that
\begin{align*}
    \frac{\int_{\mathbb{R}^d} F_0(v_0)\,\xd x}{\int_{\mathbb{R}^d} F_\omega(v_0)\,\xd x}&= \frac{\int_{\xR^d} \frac{1}{2\alpha+2} \Lambda(v_0)^{2\alpha+2} \xd x}{\int_{\xR^d} \frac{1}{2\alpha+2}\Lambda(v_0)^{2\alpha+2} \xd x -\omega  \int_{\xR^d} \frac{1}{2}\Lambda(v_0)^{2} \xd x}\\
    &= \frac{1}{1 -(\alpha+1)\omega\frac{  \int_{\xR^d} \Lambda(v_0)^{2} \xd x}{\int_{\xR^d} \Lambda(v_0)^{2\alpha+2} \xd x}} =1 + O(\omega ).
\end{align*}

For \( d \in \{3, 4\} \), we let \( R > 0 \) and introduce a cut-off function \( \eta_R \in C^\infty_0(\mathbb{R}) \) such that:
\[
\eta_R(s) =
\begin{cases}
1, & \text{for } |s| < R, \\
0, & \text{for } |s| > 2R,
\end{cases}
\qquad \text{and } \quad 0 < \eta_R < 1 \text{ for } R < |s| < 2R,
\]
with \( |\eta_R'(s)| \leq \frac{2}{R} \text{ for all } s \in \mathbb{R}. \) We have \( \eta_R v_0 \in L^2(\mathbb{R}^d) \), thus  we use it as a test function for \( J_\omega \). As above, obtain:
\[
I_\omega \leq J_\omega(\eta_R v_0) =
\frac{\int_{\xR^d} |\nabla (\eta_R v_0)|^2 \, \xd x}{\left(\frac{2d}{d-2}
\int_{\xR^d} F_0(\eta_R v_0) \, \xd x\right)^{1-\frac{2}{d}}}
\left( \frac{\int_{\xR^d} F_0(\eta_R v_0) \, \xd x}{\int_{\xR^d} F_\omega(\eta_R v_0) \, \xd x} \right)^{1-\frac{2}{d}}.
\]
First, by dominated convergence
\[
\int_{\mathbb{R}^d} |\nabla \eta_R v_0|^2 \, \xd x \to \int_{\xR^d} |\nabla v_0|^2 \, \xd x = I_0, \quad \text{as } R \to \infty.
\]
Then, we have 
\begin{align*}
\int_{\mathbb{R}^d} \Lambda(\eta_R v_0(x))^{2\alpha+2} \, \xd x & =
\int_{\mathbb{R}^d} \Lambda(v_0(x))^{2\alpha+2} \, \xd x\\
&+\int_{\mathbb{R}^d \setminus B_R(0)} \big( \Lambda(\eta_R v_0(x))^{2\alpha+2} - \Lambda(v_0(x))^{2\alpha+2} \big) \, \xd x.    
\end{align*}
For any \( x \in \mathbb{R}^d \setminus B_R(0) \), there exists \( \tau(x) \in \big(v_0(x) - (1 - \eta_R(|x|))v_0(x), v_0(x)\big) \) such that
\[
\Lambda(v_0(x))^{2\alpha+2} - \Lambda(\eta_R v_0(x))^{2\alpha+2} = (2\alpha+2)\Lambda'(\tau(x))\Lambda(\tau(x))^{2\alpha+1}(1 - \eta_R(|x|))v_0(x).
\]
Since \( v_0(x) \) decays like \( |x|^{-(d-2)} \) as \( |x| \to \infty \), so does \( \tau(x) \). Then,
\[
\Lambda(\eta_R v_0(x))^{2\alpha+2} - \Lambda(v_0(x))^{2\alpha+2} = O(|x|^{-(2\alpha+2)(d-2)}).
\]
Then, we have 
\[
\int_{\mathbb{R}^d} \Lambda(\eta_R v_0(x))^{2\alpha+2} \, \xd x = \int_{\mathbb{R}^d} \Lambda(v_0(x))^{2\alpha+2} \, \xd x + O(R^{d-(2\alpha+2)(d-2)}).
\]
Since \( 2\alpha+1 > \frac{d+2}{d-2} \), \( d - (2\alpha+2)(d-2) < 0 \) then
\[
\int_{\xR^d} F_0(\eta_R v_0) \,\xd x \to \int_{\xR^d} F_0( v_0) \,\xd x, \quad \text{as } R \to \infty.\]
This gives
\[
\frac{\int_{\xR^d} |\nabla \eta_R v_0|^2 \, \xd x}{(\frac{2d}{d-2}\int_{\xR^d} F_0(\eta_R v_0) \,\xd x)^{1 - 2/d}} \to I_0, \quad \text{as } R \to \infty.
\]
By Lemma~\ref{case_3_4},
\[
g_d(R) := \int_{\mathbb{R}^d} \Lambda(\eta_R v_0)^2 \, \xd x =
\begin{cases}
O(R), & \text{if } d = 3, \\
O(\log(R)), & \text{if } d = 4,
\end{cases}
\quad \text{as } R \to \infty. 
\]
Then we can show for all \( \omega\) and \(R > 0 \) that
\begin{align*}
\frac{\int_{\xR^d} F_0(\eta_R v_0) \,\xd x}{\int_{\xR^d} F_\omega(\eta_R v_0)\,\xd x}&= \frac{\int_{\xR^d} \Lambda(v_0)^{2\alpha+2} \,\xd x + O(R^{d-(2\alpha+2)(d-2)})}{\int_{\xR^d} \Lambda(v_0)^{2\alpha+2} \,\xd x + O(R^{d-(2\alpha+2)(d-2)})-\frac{2\alpha+2}{2} \omega g_d(R)}\\
&=\left( 1- \frac{2\alpha+2}{2} \omega \frac{g_d(R)}{\int_{\xR^d} \Lambda(v_0)^{2\alpha+2} \,\xd x + O(R^{d-(2\alpha+2)(d-2)})}     \right)^{-1}.
\end{align*}
For \( d = 3 \), we let \( R = \omega^{-\frac{1}{2}} \) and we have
\[
\frac{\int_{\xR^d} F_0(\eta_R v_0) \,\xd x}{\int_{\xR^d} F_\omega(\eta_R v_0)\,\xd x} = 1 + O(\omega^{\frac{1}{2}}), \quad \text{as } \omega \to 0.
\]
For \( d = 4 \), we let \( R = \omega^{-1} \) and we have
\[
\frac{\int_{\xR^d} F_0(\eta_R v_0) \,\xd x}{\int_{\xR^d} F_\omega(\eta_R v_0)\,\xd x} = 1 + O(\omega \log(\omega^{-1})), \quad \text{as } \omega \to 0.
\]
As a conclusion, 
\begin{align*}
    I_\omega \leq 
\frac{\int_{\xR^d} |\nabla (\eta_R v_0)|^2 \, \xd x}{\left(\frac{2d}{d-2}
\int_{\xR^d}F_0(\eta_R v_0) \, \xd x\right)^{1-\frac{2}{d}}}(1+o(1))\le I_0(1+o(1)).
\end{align*}
This proves~\eqref{convergence_minima} for any $d\ge 3$. The second  step is to show that $v_\omega \to v_0$ in $\dot{H}^1(\xR^d)$ when $\omega \to 0$. Let $v_{\omega}$ be a minimiser for \eqref{Pb_optim_I_omega}. Then by \eqref{convergence_minima}
\begin{equation}
\|\nabla v_\omega\|_{L^2}^2 \to I_0 = \|\nabla v_0\|_{L^2}^2.
\end{equation}
Then $ v_\omega$ is bounded in $\dot{H}^1(\xR^d)$, we can extract a subsequence noted also $ v_{\omega}$ and there exists $v^*$ such that
$$ v_\omega\rightharpoonup v^* \quad \text{in  } \dot{H}^1(\xR^d) \qquad    \text{and } \quad  v_n\to v^*  \text{ almost everywhere in } \xR^d .$$
Using Lemma \ref{lem_properties}(\ref{L3}) and \eqref{rescaled_solut_v_omega}, we deduce that $\|v_\omega\|_{L^\infty(\xR^d)}$ is uniformly bounded for $\omega$ close to $0^+$. Moreover, thanks to Sobolev embedding $\dot{H}^1(\xR^d)\hookrightarrow L^{\frac{2d}{d-2}}(\xR^d)$, we have that $\|v_\omega\|_{L^{\frac{2d}{d-2}}(\xR^d)}$ is also uniformly bounded for $\omega$ close to $0^+$. As a consequence,
using Holder interpolation inequality, we conclude that there exists $\eta>0$ such that, for all $s\geq \frac{2d}{d-2}$, there exists a constant $K_s>0$ such that
\[
\sup_{\omega\in (0,\eta)}\|v_\omega\|_{L^s(\mathbb{R}^d)} \leq K_s.\]

Let $q > \frac{2d}{d-2}$ and  $w_\omega := v_\omega - v^*$. By the Radial Lemma \cite{Strauss1977}, there exists a subsequence also noted $w_\omega$ such that

\begin{equation}
\label{w_omega_convergence}
w_\omega \to 0 \text{ in } L^q(\mathbb{R}^d \setminus B_1(0)).
\end{equation}
Now, let $s > q$. Since $\{v_n\}$ is bounded in $L^s(\mathbb{R}^d)$, there exists $v^{**} \in L^s(\mathbb{R}^d)$ such that, up to a subsequence, $v_\omega \rightharpoonup v^{**}$ weakly in $L^s$ and $v_\omega\to v^{**}$ almost everywhere in $\mathbb{R}^d$. Therefore $v^{**} = v^*$ almost everywhere and $v^* \in L^s(\mathbb{R}^d)$. Thanks to the Radial Lemma \cite[Lemma 4.4]{GenRot-24}, there exists a constant $C > 0$ such that
\[ \forall x \neq 0, \quad |w_n(x)| \leq C_{s, d} \|v_\omega - v^*\|_{L^s} |x|^{-d/s} \leq C |x|^{-d/s}.
\]
As a consequence,\[
|w_\omega(x)|^q \leq C^q |x|^{-dq/s} \quad \text{for } x \in B_1(0) \setminus \{0\}.
\]
Using the dominated convergence theorem, since $|x|^{-dq/s} \in L^1(B_1(0))$ and $|w_\omega|^q \to 0$ almost everywhere in $B_1(0)$, we have that
\begin{equation}
\label{w_omega_convergence_B}
    w_\omega \to 0 \text{ in } L^q(B_1(0)).
\end{equation}
From \eqref{w_omega_convergence} and \eqref{w_omega_convergence_B}, we have $v_{\omega} \to v^* \text{ in } L^q(\mathbb{R}^d)$ for all $q>\frac{2d}{d-2}$. 
In particular, for $q=2\alpha+2$, this implies, using \ref{prop_change_var}\ref{prop_change_var5},
\begin{equation}
\label{convergence_v*_v}
 \|\Lambda(v_\omega)-\Lambda(v^*)\|_{L^{2\alpha+2}(\xR^d)}\leq \|v_\omega -v^*\|_{L^{2\alpha+2}(\xR^d)  } \to 0
\end{equation} 
which means that 
\[  
\lim_{\omega \to 0} \|\Lambda(v_\omega)\|_{L^{2\alpha+2}}=\|\Lambda(v^*)\|_{L^{2\alpha+2}} .\]
Next, we want to show that $\lim\limits_{\omega\to 0}\omega \|\Lambda(v_\omega)\|^2_{L^2(\xR^d)}=0$. Let us suppose, by contradiction, that $\limsup\limits_{\omega\to 0^+}\omega \|\Lambda(v_\omega)\|^2_{L^2(\xR^d)}>0$. Then there exists a sequence $\{\omega_n\}_n$ such that 
\begin{align*}
    I_0\le J_0(v_{\omega_n})=
    \frac{ \int_{\mathbb{R}^d}|\nabla v_{\omega_n}|^2\, \xd x }
    {\left( \frac{2d}{d-2}\int_{\mathbb{R}^d} F_0(v_{\omega_n})\,\xd x \right)^{\frac{d-2}{d}}}= 
    \frac{ I_{\omega_n} }{\left( 1+\frac{d}{d-2}\omega_n\int_{\mathbb{R}^d} \Lambda^2(v_{\omega_n}) \xd x \right)^{\frac{d-2}{d}}}<I_{0}
\end{align*}
for $n$ large enough. This contradiction prove that 
\begin{equation}
    \label{convergence_omega_mass}
    \lim\limits_{\omega\to 0}\omega \|\Lambda(v_\omega)\|^2_{L^2(\xR^d)}=0.
\end{equation}
From \eqref{convergence_v*_v} and \eqref{convergence_omega_mass}, it follows that

\[
\int_{\xR^d} F_0(v^*) \, \xd x = \frac{1}{2\alpha+2} \|\Lambda(v^*)\|_{L^{2\alpha+2}}^{2\alpha+2} = \frac{1}{2\alpha+2} \lim_{\omega \to 0} \|\Lambda(v_\omega)\|_{L^{2\alpha+2}}^{2\alpha+2} = \lim_{\omega \to 0} \int_{\xR^d} F_{\omega}(v_\omega) \, \xd x = \frac{d-2}{2d}.
\]
We deduce that $v^*\in K_0$. In order to prove that $v^* $ is a minimiser of \eqref{Pb_optim_I_0}, it remains to show that $\int_{\mathbb{R}^d} |\nabla v^*|^2 \,\xd x=I_0$. By the weak lower semi-continuity of $v\to \|\nabla v\|_{L^2(\xR^d)}$ in $\dot{H}^1(\xR^d)$, we have 
\[
I_0 \leq \| \nabla v^* \|_{L^2}^2 \leq \liminf_{\omega\to \infty} \| \nabla v_\omega \|_{L^2}^2 = \lim_{\omega \to 0} I_{\omega} = I_0.
\]
Then, $v^*$ is a minimiser for \eqref{Pb_optim_I_0}, which means that $v^*$ is a positive solution of \eqref{Pb_omega_0}. As a consequence, 
$v^* = v_0$.

Finally, since $v_\omega$ converges weakly to $v_0$ in $\dot{H}^1(\xR^d)$
and $I_\omega= \| \nabla v_\omega\|_{L^2}^2$ converges to 
$I_0= \| \nabla v_0\|_{L^2}^2$, we conclude that
\[
v_\omega \to v_0 \text{ in } \dot{H}^1(\mathbb{R}^d) \cap L^q(\mathbb{R}^d), \quad \forall q \geq \frac{2d}{d-2}.
\]
Observing that
\[
u_{\omega}(x) = v_\omega(I_{\omega}^{-\frac{1}{2}} x) \text{ and } u_{0}(x) = v_0(I_{0}^{-\frac{1}{2}} x) 
\]
we deduce from the conclusions obtained for the sequence $\{v_\omega\}$ that, up to a subsequence,
\[
u_{\omega} \to u_0 \text{ in } \dot{H}^1(\mathbb{R}^d) \cap L^q(\mathbb{R}^d) \quad \forall q \geq \frac{2d}{d-2}.
\]
\end{proof}

From Theorem \ref{Thrm_super_critical}, we deduce the following result on the convergence of $\varphi_\omega$.

\begin{corol}
\label{Cor_super_critical_phi}
Suppose that $d\geq 3$ and $\alpha>\frac{2}{d-2}$. Then, as $\omega \to 0^+$, the unique solution $0<\varphi_\omega<1$ to \eqref{SSNL}
$u_\omega$ the solution of \eqref{PSSL} converges in  $\dot{H}^1(\mathbb{R}^d)\cap L^q(\mathbb{R}^d)$ for any $q\ge \frac{2d}{d-2}$ to the unique positive radial solution $0<\varphi_0<1$ of \eqref{SSNL_0}.
\end{corol}

\begin{proof}\label{convergence_subcritical}
    For any $\omega\in (0,\omega^*)$, let $0<\varphi_\omega<1$ the unique positive solution to \eqref{SSNL}. 

    Thanks to Proposition \ref{prop_change_var}\ref{prop_change_var5} and Theorem \ref{Thrm_super_critical}, we have
    \begin{equation}\label{convergence_phi_Lq}
        \|\varphi_\omega-\varphi_0\|_{L^q(\xR^d)}=\|\Lambda(u_\omega)-\Lambda(u_0)\|_{L^q(\xR^d)}\le \|u_\omega-u_0\|_{L^q(\xR^d)} \to 0 \text{ as } \omega\to 0
    \end{equation}
    for any $q\ge \frac{2d}{d-2}$. In particular, $\varphi_\omega$ converges to $\varphi_0$ in $L^{\infty}(\xR^d)$.

    To prove the convergence of $\varphi_\omega$ in $\dot{H}^{1}(\xR^d)$, we have to prove that $\|\nabla \varphi_\omega-\nabla \varphi_0\|_{L^2(\xR^d)}\to 0$ as $\omega\to 0$. First of all, we remark that
    \begin{align*}
        \|\nabla \varphi_\omega-\nabla \varphi_0\|_{L^2(\xR^d)}=&\,\|\Lambda'(u_\omega)\nabla u_\omega-\Lambda'(u_0)\nabla u_0\|_{L^2(\xR^d)}\\
        \le&\,  \|(\Lambda'(u_\omega)-\Lambda'(u_0))\nabla u_\omega\|_{L^2(\xR^d)}+\|\Lambda'(u_0)(\nabla u_\omega-\nabla u_0)\|_{L^2(\xR^d)}
    \end{align*}

    On the one hand, since $u_\omega$ converges to $u_0$ in $\dot{H}^{1}(\xR^d)$ and using Proposition \ref{prop_change_var}\ref{prop_change_var1}, we have 
    \begin{equation*}
        \|\Lambda'(u_0)(\nabla u_\omega-\nabla u_0)\|_{L^2(\xR^d)}\le \|\nabla u_\omega-\nabla u_0\|_{L^2(\xR^d)}\to 0
    \end{equation*}
    as $\omega\to 0$ and the sequence $\{\|\nabla u_\omega\|_{L^2(\xR^d)}\}_{\omega}$ is uniformly bounded for $\omega$ close to $0$. On the other hand, 
    \begin{align*}
        \int_{\xR^d}|(\Lambda'(u_\omega)-\Lambda'(u_0))\nabla u_\omega|^2\,\xd x\le&\, \int_{\xR^d}|\Lambda'(u_\omega)-\Lambda'(u_0)|^2|\nabla u_\omega|^2\,\xd x\\
        \le & \int_{\xR^d}|\Lambda^{2\alpha}(u_\omega)-\Lambda^{2\alpha}(u_0)||\nabla u_\omega|^2\,\xd x\\
        \le &\,\max\{2\alpha,1\}\int_{\xR^d}|\Lambda(u_\omega)-\Lambda(u_0)|^{\min\{2\alpha,1\}}|\nabla u_\omega|^2\,\xd x\\
        \le &\,\max\{2\alpha,1\}\|\Lambda(u_\omega)-\Lambda(u_0)\|_{L^{\infty}(\xR^d)}^{\min\{2\alpha,1\}}\int_{\xR^d}|\nabla u_\omega|^2\,\xd x\to 0
    \end{align*}
    as $\omega\to 0$. 

    As a conclusion $\varphi_\omega$ converges to $\varphi_0$ in $\dot{H}^1(\xR^d)\cap L^{q}(\xR^d)$ for $q\ge\frac{2d}{d-2}$.
\end{proof}

We conclude this section with the asymptotic behaviour of the mass of $\varphi_\omega$ as $\omega\to 0$.

\begin{prop}[Asymptotic behaviour of the mass of $\varphi_\omega$ as $\omega\to 0$]
Let $\varphi_\omega$ be the solution to \eqref{SSNL}. The asymptotic behaviour of the $L^2$-norm of $\varphi_\omega$ as $\omega \to 0$ is given by:
\[
\lim_{\omega \to 0} \|\varphi_\omega\|_{L^2(\mathbb{R}^d)} = 
\begin{cases} 
+\infty, & \text{if } d \in \{3, 4\}, \\
\|\varphi_0\|_{L^2(\mathbb{R}^d)}, & \text{if } d \geq 5.
\end{cases}
\]
\end{prop}
\begin{proof}

Let us consider the case \( d\in \{3,4\} \). Our goal is to establish that \( \| \varphi_{\omega} \|_{L^2(\xR^d)} \to \infty \) as \( \omega \to 0 \).  
Assume, by contradiction, that there exists a sequence $\omega_n$ that converges to $0$ such that $\| \varphi_{\omega_n} \|_{L^2(\xR^d)}$ remains bounded. Hence, the sequence $(\varphi_{\omega_n})$ is bounded in  $H^1(\xR^d)$ by Theorem \ref{Thrm_super_critical}. Then, applying the Rellich-Kondrachov compactness theorem and the radial lemma, $\varphi_{\omega_n}  \to \varphi_0$ in  $L^2(\xR^d)$ up to a subsequence. This implies that $\varphi_0 \in L^2(\xR^d)$, leading to a contradiction because $\varphi_0 \notin L^2(\xR^d)$ for $d\in\{3,4\}$.
Then $\lim\limits_{\omega\to 0}\|\varphi_\omega\|_{L^2(\xR^d)}= +\infty$.

For  $d\ge 5$, from Proposition \ref{prop_change_var}\ref{prop_change_var5}, we know that 
\[
    \|\varphi_\omega - \varphi_0\|_{L^q(\xR^d)} = \|\Lambda(u_\omega) - \Lambda(u_0)\|_{L^q(\xR^d)} \leq \|u_\omega - u_0\|_{L^q(\xR^d)}.
\]
It remains to prove that $\|u_\omega - u_0\|_{L^q} \to 0$ when $\omega\to 0$ for all $q \ge \frac{d}{d-2}$. We know that $u_\omega \to u_0$ almost everywhere and \{$u_\omega$\} is bounded in $L^\infty(\xR^d)$, a dominated convergence argument gives that
\[
\|u_\omega - u_0\|_{L^q(B_R(0))} \to 0 \quad \text{as } \omega \to 0,
\]
for any $R > 0$ and $q > \frac{d}{d-2}$. Thanks to radial lemma \cite[Lemma 4.4]{GenRot-24}, we have 
\begin{equation}
\label{estim_radial}
    u_\omega(x) \leq \frac{C}{|x|^{\frac{d-2}{2}}}, \quad (|x| \geq 1).
\end{equation}
Furthermore, we can show that 
\begin{equation}
\label{apper_bound_u_omega}
      u_\omega(x) \leq \frac{C_\varepsilon}{|x|^{d-2-\varepsilon}}, \quad (|x| \geq R_\varepsilon)
\end{equation}
with $R_\varepsilon, C_\varepsilon$ independent of $\omega$.

In fact, using \eqref{estim_radial} together with the estimates $f_\omega(t) < \Lambda'(t)\Lambda^{2\alpha+1}(t)$, $\Lambda'(t) \leq 1$, and $\Lambda(t) \leq t$, we obtain
\begin{equation*}
    \left(-\Delta -\frac{C}{|x|^{(d-2)\alpha}}\right)u_\omega(x) = f_\omega(u)  - \frac{C}{|x|^{(d-2)\alpha}} u \leq \left(\frac{C}{|x|^{(d-2)\alpha}}-\frac{C}{|x|^{(d-2)\alpha}} \right) u \leq 0,
\end{equation*}
where $C$ denotes the constant from \eqref{estim_radial}.

We now aim to apply the maximum principle. However, the potential $-\frac{C}{|x|^{(d-2)\alpha}}$ is negative, which prevents a direct application of the principle. To overcome this difficulty, we introduce
\[
\Phi_\omega(x) = |x|^{\frac{d-1}{2}}\, u_\omega(x),
\qquad
\Psi(x) = |x|^{\frac{d-1}{2}} \frac{1}{|x|^{d-2-\varepsilon}}
       = |x|^{\frac{3-d}{2}+\varepsilon}.
\]

From Theorem~\ref{thrm_existence}, $\Phi_\omega$ is radial, decays exponentially to $0$, and satisfies
\begin{equation}
\label{L_Phi}
\left(-\Delta - \frac{C}{r^{(d-2)\alpha}}\right)\Phi_\omega(r)
= -\Phi_\omega''(r)
+ \left(\frac{(d-1)(d-3)}{4}
       - \frac{C}{r^{(d-2)\alpha-2}}\right)
\frac{1}{r^{2}}\,\Phi_\omega(r)
\leq 0.
\end{equation}

Moreover, $\Psi$ satisfies
\begin{equation}
\label{L_Psi}
\left(-\Delta - \frac{C}{r^{(d-2)\alpha}}\right)\Psi(r)
= \left(\varepsilon(d-2-\varepsilon)
        - \frac{C}{r^{(d-2)\alpha-2}}\right)
\frac{1}{r^{2}}\,\Psi(r)
\geq 0,
\end{equation}
for all $r \geq R_\varepsilon$, where
\[
R_\varepsilon \geq
\left(\frac{C}{\varepsilon(d-2-\varepsilon)}\right)^{\frac{1}{(d-2)\alpha-2}}.
\]

Define the operator
\[
Lg(r)
= -g''(r)
+ \left(\frac{(d-1)(d-3)}{4}
       - \frac{C}{r^{(d-2)\alpha-2}}\right)
\frac{1}{r^{2}}\,g(r).
\]

From \eqref{L_Phi} and \eqref{L_Psi}, we observe that $L(\Phi_\omega - C(R_\varepsilon)\Psi) \leq 0$ for sufficiently small $\omega$. Applying the maximum principle \cite[Theorem~4, Section~6.4]{evans2009partial} and imposing the boundary condition $\Phi_\omega (R_\varepsilon) - C(R_\varepsilon)\Psi(R_\varepsilon)=0$, we obtain
\[
\Phi_\omega (r) \leq C(R_\varepsilon) \Psi(r),
\]
which yields \eqref{apper_bound_u_omega}.

Now, for sufficiently small $\varepsilon>0$, using estimates \eqref{estim_radial} and \eqref{estim_radial_u0}, and noting that $q(d-2)>d$, it follows from the dominated convergence theorem that
\[
    \|u_\omega - u_0\|_{L^q(\mathbb{R}^d\setminus B_{R_\varepsilon}(0))} \to 0 \quad \text{as } \omega \to 0.
\]
Then, thanks to Proposition \ref{prop_change_var}\ref{prop_change_var5}, we have
\[
    \|\varphi_\omega - \varphi_0\|_{L^q} = \|\Lambda(u_\omega) - \Lambda(u_0)\|_{L^q} \leq \|u_\omega - u_0\|_{L^q} \to 0, \quad \omega \to 0.
\]
\end{proof}

\subsection{Sub-critical case}
 We want to study the behaviour of the positive solutions as $\omega$ approaches zero when $\alpha<\frac{2}{d-2}$. For any $\omega\in (0,\omega^*)$, let $0<\varphi_\omega<1$ the unique positive solution to \eqref{SSNL}. We consider the rescaling $\varphi_\omega (x)=\omega^{\frac{1}{2\alpha}}\tilde\varphi_\omega (\omega^{\frac{1}{2}}x)$. As consequence $\tilde\varphi_\omega$ is a positive solution to the equation  
\begin{equation}
\label{RSNL}
  -\psi = -\nabla. \left(\frac{\nabla\psi}{1-\omega |\psi|^{2\alpha}} \right) +\alpha\omega |\psi|^{2\alpha-2}\frac{|\nabla\psi|^2}{(1-\omega|\psi|^{2\alpha})^2}\psi - |\psi|^{2\alpha}\psi.
\end{equation}  
When $\omega=0$, we obtain
    \begin{equation}
    \label{limit_RSNL}
-\Delta \psi+ \psi - |\psi|^{2\alpha}\psi=0,
    \end{equation}
the usual nonlinear Schrödinger equation with a power non-linearity, which has a unique positive, radial solution $\psi_0\in H^2(\xR^d)$ (see \cite{tao2006nonlinear}). Our aim is to show that $\omega^{\frac{-1}{2\alpha}}\varphi_\omega (\omega^{\frac{-1}{2}}x)$,  where $\varphi_\omega$ is the unique positive solution to \eqref{SSNL}, tends to $\psi_0$ in $H^1(\xR^d)\cap L^{\infty}(\xR^d)$ when $\omega$ approaches zero. To achieve this goal, it is more convenient to work on the semi-linear problem. For this purpose, we will apply the same rescaling to the problem \eqref{PSSL}. Let $u_\omega $ be the unique positive solution to \eqref{PSSL}. We consider the rescaling $u_\omega(x)=\omega^{\frac{1}{2\alpha}}\tilde{u}_\omega (\omega^{\frac{1}{2}}x)$ with $\omega>0$. Then $\tilde{u}_\omega$ is a solution to 
    $$-\Delta \tilde{u}_\omega (x)=\Lambda'(\omega^{\frac{1}{2\alpha}} \tilde{u}_\omega (x))
    \left(\omega^{-\frac{1}{2\alpha}-1}\Lambda^{2\alpha+1}(\omega^{\frac{1}{2\alpha}} \tilde{u}_\omega (x))-   \omega^{-\frac{1}{2\alpha}}\Lambda(\omega^{\frac{1}{2\alpha}} \tilde{u}_\omega (x)) \right).$$
 
 \begin{defi}
 \label{defi_change_var_omega0}
 Let $t\in \xR^+$ and  $\Lambda$ be the change of variables given in Definition \ref{defi_change_var}. For all $\omega>0$, we define $\Lambda_\omega (t) = \omega^{\frac{-1}{2\alpha}} \Lambda (\omega^{\frac{1}{2\alpha}}t)$, and $\Lambda_0(t)=t$.
\end{defi}

Since $\Lambda'_\omega (t) =\Lambda'(\omega^{\frac{1}{2\alpha}}t)$, we have  the following semi-linear problem for $\tilde{u}$:
\begin{equation}
\label{PSSL_omega0}
\left\{
\begin{aligned}
-\Delta \tilde{u}(x) &= \tilde{f}_\omega (\tilde{u}(x)) \qquad \text{ in } \xR^d\\   
  \tilde{u}\in& H^1(\xR^d),  \qquad u \neq 0,
\end{aligned}
\right.
\end{equation}
with 
\begin{equation}\label{def_f_tilde}
    \tilde{f}_\omega(t)=\Lambda_\omega'(t)\left( \Lambda_\omega(t)^{2\alpha+1}-\Lambda_\omega (t)\right) \text{ for all } t\in \xR^+.
\end{equation}

When $\omega=0$,  we recover
\begin{equation}
\label{PSSL_0}
    -\Delta v(x) = \tilde{f}_0 (v(x))  \qquad  \text{with} \qquad  \tilde{f}_0 (t)=t^{2\alpha+1}-t.
\end{equation}
 and the function $\psi_0$ defined above is the unique positive solution to \eqref{PSSL_0}.

In Proposition \ref{prop_change_variable_omega0} and Lemma \ref{convergence_uniform_Lambda_omega}, we state some properties of $\Lambda_\omega$ and  $\tilde{f}_\omega$ that will be used later.

\begin{prop}
\label{prop_change_variable_omega0}
For any $\omega>0$, let $t^*_\omega= t^*\omega^{\frac{-1}{2\alpha}}$, with $t^*$ given by \eqref{t*} and $t^*_0=+\infty$.
\begin{enumerate}[label=\alph*)]

\item \label{prop_change_variable_omega0_1} For all $t\in (0,t_\omega^*)$, $\Lambda_\omega'(t)=\sqrt{1-\omega \Lambda_\omega^{2\alpha}(t)}  $ and $\Lambda_\omega''(t)=-\alpha \omega \Lambda_\omega^{2\alpha-1}(t) $.

\item For all $t\ge 0$, $0\le \Lambda_\omega (t)\leq \mathrm{min}\{t,\omega^{-\frac{1}{2\alpha}}\}$.  \label{prop_change_variable_omega0_2}

\item \label{prop_change_variable_omega0_4} For all $\omega\geq 0$, the function \( \tilde{f}_\omega \)  in \eqref{PSSL_omega0}, is of class \( C^2 ((0,t^*_\omega))\), and its first and second derivatives in $(0,t^*_\omega)$ are given by 
\begin{equation}\label{df_tilde_omega}
\tilde{f}_\omega'(t)= \left( \Lambda^{2\alpha}_\omega(t)-1\right)\left( 1-\omega(\alpha+1)\Lambda^{2\alpha}_\omega(t)\right) + 2\alpha \Lambda^{2\alpha}_\omega(t)\left(1-\omega \Lambda^{2\alpha}_\omega(t)\right),
\end{equation}

\begin{equation}\label{ddf_tilde_omega}
    \tilde{f}_\omega ''(t)=2\alpha \Lambda'_\omega(t) \Lambda^{2\alpha-1}_\omega (t) \left( 1+\omega(\alpha+1)+ 2\alpha - 2\omega(3\alpha+1) \Lambda^{2\alpha}_\omega(t)\right).
\end{equation}
In particular, for all $\omega>0$, $\tilde f_\omega \in C^{1,\min\{1,2\alpha\}}((0,t^*_\omega))$.
\item \label{prop_change_variable_omega0_3} For all $\omega\ge 0$, the first derivative of $\tilde{f}$ is bounded as follows
\begin{equation}
\label{Lipshitz_caracter_f}
    |\tilde{f}_\omega'(t)|\leq (2\alpha+\max\{\alpha,1\}) t^{2\alpha}+ \max\{\alpha,1\}
\end{equation}
for all $t\in(0,t^*_\omega)$.
\end{enumerate}

\end{prop}
\begin{proof}
Properties \ref{prop_change_variable_omega0_1} and \ref{prop_change_variable_omega0_2} follow from Definition \ref{defi_change_var_omega0}, $\Lambda'_\omega (t) =\Lambda'(\omega^{\frac{1}{2\alpha}}t)$ and Proposition \ref{prop_change_var}\ref{prop_change_var1}. A straightforward computation gives \eqref{df_tilde_omega} and \eqref{ddf_tilde_omega}. As a consequence, for $\omega>0$ and $2\alpha\ge 1$, 
\begin{equation}
    |\tilde{f}_\omega ''(t)|\le \frac{2\alpha(3+\omega(\alpha+1)+8\alpha)}{\omega^{1-\frac{1}{2\alpha}}}
\end{equation}
and $\tilde f'_\omega \in C^{0,1}((0,t^*_\omega))$. For $0<2\alpha<1$, $\tilde{f}_\omega ''$ has a singularity at $0$. However, since 
\begin{equation*}
    |1-\omega \Lambda^{2\alpha}_\omega(t)|\le 1 \text{ and } |1-(\alpha+1)\omega \Lambda^{2\alpha}_\omega(t)|\le \max\{\alpha,1\}=1,
\end{equation*}
we obtain, for all $s,t\in (0,t^*_\omega)$
\begin{align*}
    |\tilde f'_\omega(s)-&\tilde f'_\omega(t)|\\
    \le&\, \omega(\alpha+1)|\Lambda^{2\alpha}_\omega(s)-\Lambda^{2\alpha}_\omega(t)|+|\Lambda^{2\alpha}_\omega(s)\left( 1-\omega(\alpha+1)\Lambda^{2\alpha}_\omega(s)\right) -\Lambda^{2\alpha}_\omega(t)\left( 1-\omega(\alpha+1)\Lambda^{2\alpha}_\omega(t)\right)|\\
    &+2\alpha|\Lambda^{2\alpha}_\omega(s)\left(1-\omega \Lambda^{2\alpha}_\omega(s)\right)-\Lambda^{2\alpha}_\omega(t)\left(1-\omega \Lambda^{2\alpha}_\omega(t)\right)|\\
    \le&\, \left(\omega(\alpha+1)+|1-\omega(\alpha+1)\Lambda^{2\alpha}_\omega(s)|+\omega(\alpha+2) |\Lambda^{2\alpha}_\omega(t)|+|1-\omega \Lambda^{2\alpha}_\omega(s)|\right)|\Lambda_\omega(s)-\Lambda_\omega(t)|^{2\alpha}  \\
   \le&\, \left(\omega(\alpha+1)+(\alpha+4)\right)|s-t|^{2\alpha}.
\end{align*}

Finally, starting from \eqref{df_tilde_omega}, we have, for all $t\in (0,t^*)$,
\begin{align*}
    |\tilde f'_\omega(t)|\le (|\Lambda^{2\alpha}_\omega(t)|+1)\max\{\alpha,1\}+2\alpha |\Lambda^{2\alpha}_\omega(t)|\le (2\alpha+\max\{\alpha,1\})t^{2\alpha}+\max\{\alpha,1\}
\end{align*}
which gives the property \ref{prop_change_variable_omega0_3}.

\end{proof}

Let $\omega\ge 0$ and consider $\Lambda_\omega$ and $\tilde f_\omega$ as defined in Definition \ref{defi_change_var_omega0} and equation \eqref{def_f_tilde} respectively. 
For any $\omega<0$, we define 
\begin{equation*}
    \label{def_change_variable0_neg_omega}
    t^*_{\omega}=t^*_{-\omega},\ \Lambda_{\omega}(t)=\Lambda_{-\omega}(t),\text{ and } \tilde f_{\omega}(t)=\tilde f_{-\omega}(t)\text{ for all }t\in (0,t^*_\omega).
\end{equation*}
Moreover, we extend the definition of the change of variable and of the non-linearity to the interval $(-t^*_\omega,0)$ by setting 
\begin{equation*}
    \Lambda_\omega(t)=\Lambda_\omega(-t) 
\end{equation*}
so that 
\begin{equation*}
    \tilde f_\omega(t)=-\tilde f_\omega(-t).
\end{equation*}
\begin{rmrk}\label{properties_change_variable_omega}
    Note that $\Lambda_\omega(t)$ is now defined for any $\omega\in \R$. The function $\Lambda_\omega$ is continuous on $[-t^*_\omega,t^*_\omega]$ and differentiable on $(-t^*_\omega,0)\cup (0,t^*_\omega)$. Moreover, if $\omega\neq 0$, $|\Lambda'_\omega(t)|\le 1$ on $(-t^*_\omega,0)\cup (0,t^*_\omega)$, so that $\Lambda_\omega$ is a Lipschitz function on $[-t^*_\omega,t^*_\omega]$ for any $\omega\neq 0$.
\end{rmrk}

We have the following lemma.

\begin{lem}
\label{convergence_uniform_Lambda_omega}
Let $s\geq 2$ and $u \in  L^s(\xR^d)\cap L^\infty(\xR^d)$. Let $\{\omega_n\}_n$ be a sequence in $(-\delta,\delta)$ that converges to $\omega$ as $n\to +\infty$. Assume that  $\|u\|_{L^{\infty}(\xR^d)}< t^*_{\omega}$. Then 
\begin{equation}
       \|\Lambda_{\omega_n}(u)-\Lambda_\omega(u)\|_{L^s(\xR^d)}\to 0
\end{equation}
as $n \to +\infty$.
\end{lem}

\begin{proof}
    We start by proving the result for $\omega\neq 0$. Since $\Lambda_\omega=\Lambda_{|\omega|}$, we can assume $\omega>0$, without loss of generality. As a consequence, $\omega_n>0$ at least for $n$ large enough. 
    For any $t\in \xR$,
    \begin{align*}
        |\Lambda_{\omega_n}(t)-\Lambda_\omega(t)|&=|   \omega_n^{\frac{-1}{2\alpha}} \Lambda(\omega_n^{\frac{1}{2\alpha}} t)-\omega^{\frac{-1}{2\alpha}}\Lambda(\omega^{\frac{1}{2\alpha}}t)|\\
        &\leq \omega_n^{\frac{-1}{2\alpha}}| \Lambda(\omega_n^{\frac{1}{2\alpha}} t)-\Lambda(\omega^{\frac{1}{2\alpha}}t)|+|\omega_n^{\frac{-1}{2\alpha}}-\omega^{\frac{-1}{2\alpha}}|\Lambda(\omega^{\frac{1}{2\alpha}}t)\\
        &\le \omega_n^{\frac{-1}{2\alpha}}| \omega_n^{\frac{1}{2\alpha}} -\omega^{\frac{1}{2\alpha}}||t|+\omega^{\frac{1}{2\alpha}}|\omega_n^{\frac{-1}{2\alpha}}-\omega^{\frac{-1}{2\alpha}}||t|
    \end{align*}
thanks to Proposition \ref{prop_change_var}. 
As a consequence, for any $u\in L^{s}(\xR^d)$,
\begin{align*}
     \|\Lambda_{\omega_n}(u)-\Lambda_\omega(u)\|_{L^s(\xR^d)}\le \left(\omega_n^{\frac{-1}{2\alpha}}| \omega_n^{\frac{1}{2\alpha}} -\omega^{\frac{1}{2\alpha}}|+\omega^{\frac{1}{2\alpha}}|\omega_n^{\frac{-1}{2\alpha}}-\omega^{\frac{-1}{2\alpha}}|\right)\|u\|_{L^s(\xR^d)}\to 0
\end{align*}
as $n\to +\infty$.

For $\omega = 0$, we have $\Lambda_0(t) = t$ for all $t\ge 0$, and our goal is to estimate the quantity $|\Lambda_{\omega_n}(t) - \Lambda_0(t)|$ for all $t\in (-t^*_{\omega_n},t^*_{\omega_n})$. Since $\Lambda_\omega$ is an even function and $\Lambda_\omega(0)=0$ for all $\omega$, it is enough to get the estimate for $t\in (0,t^*_{\omega_n})$. Let $t\in (0,t^*_{\omega_n})$, we have 
\begin{align}\label{uniform_bound_Lambda_omega}
    |\Lambda_{\omega_n}(t) -\Lambda_0(t)|=&\,|\Lambda_{\omega_n}(t)-\Lambda_{\omega_n}(0) -t|=\left|\int_{0}^t\Lambda'_{\omega_n}(s)\,\xd s-t\right|\le \int_{0}^t|\Lambda'_{\omega_n}(s)-1|\,\xd s \nonumber\\
    \le&\,  \int_0^t\frac{\omega_n\Lambda_{\omega_n}^{2\alpha}(s)}{1+\sqrt{1-\omega_n\Lambda_{\omega_n}^{2\alpha}(s)}}\,\xd s\le \frac{\omega_n}{1+2\alpha}|t|^{2\alpha+1}.
\end{align} 
Hence, for any $u\in L^s(\xR^d)\cap L^{\infty}(\xR^d)$ such that $\|u\|_{L^{\infty}(\xR^d)}< t^*_{\omega_n}$ for all $n\ge n_0$, we have 
\begin{equation*}
    \|\Lambda_{\omega_n}(u)-\Lambda_\omega(u)\|_{L^s(\xR^d)}\le \frac{\omega_n}{1+2\alpha}\|u\|^{2\alpha}_{L^{\infty}(\xR^d)}\|u\|_{L^s(\xR^d)}\to 0
\end{equation*}
as $n\to +\infty$.

\end{proof}

Now, we can state the result on the convergence of $\tilde u_\omega$.

\begin{thrm}\label{thm_convergence_tilde_u_subcritical}
Suppose $d=2$, or $d \geq 3$ and $\alpha < \frac{2}{d-2}$. Then, as $\omega \to 0, \, \tilde{u}_\omega (x)= \omega^\frac{-1}{2\alpha} u_\omega (\omega^\frac{-1}{2\alpha} x) $ with $u_\omega$ the unique positive

solution to \eqref{PSSL} converges in $L^{\infty}(\xR^d)\cap H^2_{\mathrm{rad}}(\xR^d)$ to $\psi_0$ the unique positive solution of \eqref{limit_RSNL}.
\end{thrm}

\begin{proof}

To prove this theorem, we will use the implicit function argument in a neighbourhood of $(0,\psi_0)$. 

Given $q\in \xN$ with $q>\frac{d}{2}$, we define the open set 
\begin{equation}\label{def_open_set_implicit_arg}
   \mathcal{V}= \left\{ v \in W^{2,q}(\mathbb{R}^d) \cap H^2_{\mathrm{rad}}(\mathbb{R}^d) \ : \ \|v-\psi_0\|_{L^\infty(\mathbb{R}^d)} < \|\psi_0\|_{L^\infty(\mathbb{R}^d)} \right\}
\end{equation}
and we choose  $\delta = \left(\frac{t^*}{4\|\psi_0\|_{L^\infty(\xR^d)}}\right)^{2\alpha}$.

This choice implies that, for any $(\omega,u)\in (-\delta,\delta)\setminus\{0\}\times \mathcal{V}$, we have 
\begin{equation*}
    \|u\|_{L^\infty(\xR^d)} \leq \frac{t^* |\omega|^{-\frac{1}{2\alpha}}}{2}<t^*_{\omega}
\end{equation*}
so that $\tilde f_\omega(u(x))$ is well-defined for any $x\in \xR^d$. 

Hence, we define the map 
\[
\begin{aligned}
    \Phi: (-\delta,\delta) \times  \mathcal{V}\subset \xR\times W^{2,q}(\mathbb{R}^d) \cap H^2_{\mathrm{rad}}(\mathbb{R}^d) &\to L^q(\xR^d)\cap L^2_\mathrm{rad}(\xR^d) \\
    (\omega, u) &\mapsto  -\Delta u - \tilde{f}_\omega(u).
\end{aligned}
\]
Since $\Phi(0,\psi)=0$, we aim to apply the implicit function theorem in $(-\delta,\delta) \times  \mathcal{V}$ which is an open subset of $\xR\times W^{2,q}(\mathbb{R}^d) \cap H^2_{\mathrm{rad}}(\mathbb{R}^d)$.
To proceed, we must show that the application  $\Phi$ is continuous from $(-\delta,\delta)\times \mathcal{V} $ to $L^q(\xR^d)\cap L^2_\mathrm{rad}(\xR^d)$, for $s\in\{q,2\}$. Let $\{(\omega_n,u_n)\}_n$ be a sequence in $(-\delta,\delta) \times  \mathcal{V}$ that converges to $(\omega,u)$ in $\xR\times  W^{2,q}(\mathbb{R}^d) \cap H^2_{\mathrm{rad}}(\mathbb{R}^d)$. We have  
 \begin{align*}
    \|\Phi(\omega_n,u_n)-\Phi(\omega,u)\|_{L^s(\xR^d)}=&\|\Delta u_n-\Delta u\|_{L^s(\xR^d)}+\|\tilde{f}_{\omega_n}(u_n)-\tilde{f}_\omega(u)\|_{L^s(\xR^d)}\\
    \leq& \| u_n- u\|_{W^{2,s}(\xR^d)}+\|\tilde{f}_{\omega_n}(u_n)-\tilde{f}_{\omega_n}(u)\|_{L^s(\xR^d)}+\\
    &\|\tilde{f}_{\omega_n}(u)-\tilde{f}_\omega(u)\|_{L^s(\xR^d)}.
\end{align*}
Note that, since $\|u_n-u\|_{L^\infty(\xR^d)}\to 0$ as $n\to \infty$, $\|u\|_{L^\infty(\xR^d)}< t_{\omega_n}^*$ for $n$ large enough. Hence, for the second term on the right hand side, using \eqref{Lipshitz_caracter_f}, we have
\begin{align*}
\|\tilde{f}_{\omega_n}(u_n)-\tilde{f}_{\omega_n}(u)\|^s_{L^s(\xR^d)}=&\int_{\xR^d} |\tilde{f}_{\omega_n}(u_n(x))- \tilde{f}_{\omega_n}( u(x))|^s \xd x\\
\leq&\int_{\xR^d} ((2\alpha+\max\{\alpha,1\}) \max \{|u_n(x)|^{2\alpha},|u(x)|^{2\alpha}\}\\ &+\max\{\alpha,1\})^s |u_n(x)- u(x)|^s \, \xd x\\
\leq &((2\alpha+\max\{\alpha,1\}) (2\|\psi_0\|_{L^\infty(\xR^d)})^{2\alpha}+\max\{\alpha,1\})^s\|u_n- u\|^s_{L^s(\xR^d)}.
\end{align*}
To prove the continuity, it remains to show that $\|\tilde{f}_{\omega_n}(u)-\tilde{f}_{\omega}(u)\|_{L^s(\xR^d)} \to  0$ as $n \to +\infty$. First,
$$
\tilde{f}_{\omega_n}(u(x))-\tilde{f}_{\omega}(u(x))\to 0 \text{ a.e. in } \xR^d
$$
as $n\to +\infty$. Indeed, if $\omega\neq 0$, this follows from the continuity of the map $\omega\in (0,+\infty)\mapsto \tilde f_\omega(u(x))\in \R$. If $\omega=0$, since $\tilde f_{\omega_n}(0)=0$ and $\tilde f_{\omega_n}$ is an odd function, we can restrict our analysis to $x\in \xR^d$ such that $u(x)>0 $. In this case, we write 
\begin{align*}
    \tilde{f}_{\omega_n}(u(x))-\tilde{f}_{0}(u(x))=&\Lambda'(\omega_n^{\frac{1}{2\alpha}}u(x))\left(u(x)^{2\alpha+1}\frac{\Lambda^{2\alpha+1}(\omega_n^{\frac{1}{2\alpha}}u(x))}{(\omega_n^{\frac{1}{2\alpha}}u(x))^{2\alpha+1}}-u(x)\frac{\Lambda(\omega_n^{\frac{1}{2\alpha}}u(x))}{\omega_n^{\frac{1}{2\alpha}}u(x)}\right)\\
    &-(u(x)^{2\alpha+1}-u(x)),
\end{align*}
and we use that $\lim\limits_{t\to 0^+}\frac{\Lambda(t)}{t}=1=\lim\limits_{t\to 0^+}\Lambda'(t)$ (see Proposition \ref{prop_change_var}\ref{prop_change_var3}) to conclude. 
Next, since $0\le\Lambda'_\omega(t)\le 1$ in $(-t^*_\omega,t^*_\omega)$ and thanks to Proposition \ref{prop_change_variable_omega0}\ref{prop_change_variable_omega0_2}, for all $x\in \xR^d$, 
$$|\tilde{f}_{\omega_n}(u(x))-\tilde{f}_{\omega}(u(x))|\leq 2 |u(x)|(|u(x)|^{2\alpha}+1)\in L^{s}(\xR^d), $$
since  $u\in W^{2,q} (\xR^d) \hookrightarrow L^\infty(\xR^d)$. Then using the dominated convergence theorem, we deduce that $\|\tilde{f}_{\omega_n}(u)-\tilde{f}_{\omega}(u)\|_{L^s(\xR^d)} \to  0$ as $n \to +\infty$,
and $\Phi$ is continuous.

Now, we want to show that for all $\omega \in (-\delta, \delta)$, the map $u\in  \mathcal{V} \mapsto \Phi(\omega, u)$ is Fréchet differentiable, and its differential is the linear continuous operator given by  
\[
\begin{aligned}
    D_u \Phi(\omega, u) : W^{2,q}(\mathbb{R}^d) \cap H^2_{\mathrm{rad}}(\mathbb{R}^d) &\to L^q_{\mathrm{rad}}(\mathbb{R}^d) \cap L^2_{\mathrm{rad}}(\mathbb{R}^d), \\
    v &\mapsto -\Delta v - \tilde{f}'_\omega(u)v.
\end{aligned}
\]  
This means that we need to show that  
\[
\lim_{\|v\|_{W^{2,q}(\mathbb{R}^d) \cap H^2(\mathbb{R}^d)} \to 0} \frac{\|\Phi(\omega, u+v) - \Phi(\omega, u) - D_u \Phi(\omega, u)v\|_{L^q(\mathbb{R}^d) \cap L^2(\mathbb{R}^d)}}{\|v\|_{W^{2,q}(\mathbb{R}^d) \cap H^2(\mathbb{R}^d)}} = 0.
\]  
We have:  
\begin{align*}
\|\Phi(\omega, u+v) - \Phi(\omega, u) -& D_u \Phi(\omega, u)v\|_{L^q (\mathbb{R}^d)\cap L^2 (\mathbb{R}^d)}\\
=& \|\tilde{f}_\omega(u+v) - \tilde{f}_\omega(u) - \tilde{f}'_\omega(u)v\|_{L^q (\mathbb{R}^d)\cap L^2 (\mathbb{R}^d)}.
\end{align*}
Since $\tilde f_\omega$ is an odd function, we deduce from Proposition \ref{prop_change_variable_omega0}\ref{prop_change_variable_omega0_4} that $\tilde f_\omega\in C^{1,\min\{1,2\alpha\}}((-t^*_\omega,t^*_\omega))$ for all $\omega\neq 0$. Hence,
\begin{equation*}
    |\tilde{f}_\omega(u+v) - \tilde{f}_\omega(u)- \tilde{f}'_\omega(u) v|\le C_{\alpha,\omega}|v|^{1+\min\{1,2\alpha\}}
\end{equation*}
for some positive constant $C_{\alpha,\omega}$ that depends on $\omega$ and $\alpha$. As a consequence, 
\begin{align*}
    \|\tilde{f}_\omega(u+v) - \tilde{f}_\omega(u)-\tilde{f}'_\omega(u) v\|_{L^q (\mathbb{R}^d)\cap L^2(\mathbb{R}^d)} &\leq C_{\alpha,\omega}\|v\|_{L^\infty(\xR^d)}^{\min\{1,2\alpha\}}\|v\|_{L^q (\mathbb{R}^d)\cap L^2(\mathbb{R}^d)} \\
    &\le C \|v\|^{1+\min\{1,2\alpha\}}_{W^{2,q} (\mathbb{R}^d)\cap H^2(\mathbb{R}^d)} 
\end{align*}
using Sobolev embedding. This implies
\begin{equation*}
     \frac{\|\Phi(\omega, u+v) - \Phi(\omega, u) - D_u \Phi(\omega, u)v\|_{L^q(\mathbb{R}^d) \cap L^2(\mathbb{R}^d)}}{\|v\|_{W^{2,q} (\mathbb{R}^d)\cap H^2(\mathbb{R}^d)}} \leq C \|v\|^{\min\{1,2\alpha\}}_{W^{2,q}(\mathbb{R}^d) \cap H^2(\mathbb{R}^d)} \to 0
\end{equation*}
for all $\omega\neq 0$. 

For $\omega=0$, we use the explicit expression of $\tilde f_0$. If $0<2\alpha\le 1$, we remark that $\tilde f_0\in \mathcal C^{1,2\alpha}(\xR)$. As above, this implies
\begin{equation*}
    |\tilde{f}_0(u+v) - \tilde{f}_0(u)- \tilde{f}'_0(u) v|\le C_{\alpha,\omega}|v|^{1+2\alpha}
\end{equation*}
and 
\begin{equation*}
     \frac{\|\Phi(\omega, u+v) - \Phi(\omega, u) - D_u \Phi(\omega, u)v\|_{L^q(\mathbb{R}^d) \cap L^2(\mathbb{R}^d)}}{\|v\|_{W^{2,q} (\mathbb{R}^d)\cap H^2(\mathbb{R}^d)}} \leq C \|v\|^{2\alpha}_{W^{2,q}(\mathbb{R}^d) \cap H^2(\mathbb{R}^d)} \to 0.
\end{equation*}
If $2\alpha>1$, we use the fact that, for all $t,s\in \xR$ 
\begin{equation*}
    |\tilde{f}_0'(t)-\tilde{f}_0'(s)|\le 2\alpha(2\alpha+1) \max\{|t|,|s|\}^{2\alpha-1}|t-s|.
\end{equation*}
This gives
\begin{align*}
    \left|\tilde{f}_0(u+v) - \tilde{f}_0(u)- \tilde{f}'_0(u) v\right|=\left|\int_{0}^1 \tilde{f}'_0(u+\sigma v)v \,\xd \sigma- \tilde{f}'_0(u) v\right|\le \alpha(2\alpha+1)|v|^2 {(|u|+|v|)^{2\alpha-1}}
\end{align*}
which implies 
\begin{equation*}
     \frac{\|\Phi(\omega, u+v) - \Phi(\omega, u) - D_u \Phi(\omega, u)v\|_{L^q(\mathbb{R}^d) \cap L^2(\mathbb{R}^d)}}{\|v\|_{W^{2,q} (\mathbb{R}^d)\cap H^2(\mathbb{R}^d)}} \leq C \|v\|_{W^{2,q}(\mathbb{R}^d) \cap H^2(\mathbb{R}^d)} \to 0
\end{equation*}
since $\|u\|_{L^{\infty}(\xR^d)}\le 2 \|\psi_0\|_{L^{\infty}(\xR^d)}$ and $\|v\|_{L^{\infty}(\xR^d)}\le \|v\|_{W^{2,q}(\mathbb{R}^d) \cap H^2(\mathbb{R}^d)} \to 0$.

It remains to show that $D_u \Phi(\omega,u)v  $ is continuous. For any $v \in W^{2,q}(\mathbb{R}^d)$,
\begin{align*}
     \|D_u \Phi(\omega,u)v \|_{L^s(\xR^d)}&\leq \|\Delta v\|_{L^s(\xR^d)} +\|\tilde{f}'_\omega(u)v\|_{L^s(\xR^d)}\\
     &\leq (1+\|\tilde{f}'_\omega(u)\|_{L^\infty(\xR^d)})\| v\|_{W^{2,s}(\xR^d)}\\
     &\le C \| v\|_{W^{2,s}(\xR^d)} 
\end{align*}
since $\|\tilde{f}'_\omega(u)\|_{L^\infty(\xR^d)}$ is bounded thanks to \eqref{Lipshitz_caracter_f}.

Now we must show that $(\omega,u)\in (-\delta,\delta)\times \mathcal V\mapsto D_u\Phi(\omega,u)$ is continuous.
Let $\{(\omega_n,u_n)\}_n$ be a sequence in $(-\delta,\delta) \times  \mathcal{V}$ that converges to $(\omega,u)$ in $\xR\times  W^{2,q}(\mathbb{R}^d) \cap H^2_{\mathrm{rad}}(\mathbb{R}^d)$, and $v\in W^{2,q}(\mathbb{R}^d) \cap H^2_{\mathrm{rad}}(\mathbb{R}^d)$ such that 
$\|v\|_{H^2(\xR^d)\cap W^{2,q}(\xR^d)}= 1$. Let $s\in \{2,q\}$, we have 

\begin{align*}
    \|D_u \Phi(\omega_n,u_n)v - D_u \Phi(\omega,u)v\|_{L^s(\xR^d)}=&\, \|(\tilde{f}'_{\omega_n}(u_n)-\tilde{f}'_\omega(u))v\|_{L^s(\xR^d)}\\
    \leq&\|(\tilde{f}'_{\omega_n}(u_n)-\tilde{f}'_{\omega_n}(u))v\|_{L^s(\xR^d)}\\
    & + \|(\tilde{f}'_{\omega_n}(u)-\tilde{f}'_\omega(u))v\|_{L^s(\xR^d)}.
\end{align*} 
If $0<2\alpha\le 1$, we deduce from (the proof of) Proposition \ref{prop_change_variable_omega0} and the explicit expression of $\tilde f'_0$ that, for all $s,t\in (-t^*_\omega,t^*_\omega)$,
\begin{equation*}
    |\tilde f'_\omega(s) -\tilde f'_\omega(t)|\le C_{\alpha}(|\omega|+1)|s-t|^{2\alpha}
\end{equation*}
with $C_\alpha$ a positive constant for all $\omega \in\xR$.
Hence, proceeding as in the proof of Proposition \ref{prop_regularity}, we deduce that
\begin{align*}
    \|(\tilde{f}'_{\omega_n}(u_n)-\tilde{f}'_{\omega_n}(u))v\|_{L^s(\xR^d)}\le C_{\alpha}(|\omega_n|+1) \|u_n-u\|_{L^s(\xR^d)}^{2\alpha}\|v\|_{L^\infty(\xR^d)}^{2\alpha}\|v\|_{L^s(\xR^d)}^{1-2\alpha}\le C  \|u_n-u\|_{L^s(\xR^d)}^{2\alpha}
\end{align*}
since $|\omega_n|\le \delta$ for all $n$.
If $2\alpha> 1$, we deduce from \eqref{ddf_tilde_omega} and the explicit expression of $\tilde f'_0$ that
\begin{equation*}
    |\tilde f'_\omega(s) -\tilde f'_\omega(t)|\le C_{\alpha}(|\omega|+1)\max\{|s|,|t|\}^{2\alpha-1}|s-t|
\end{equation*}
This leads to 
\begin{align*}
    \|(\tilde{f}'_{\omega_n}(u_n)-\tilde{f}'_{\omega_n}(u))v\|_{L^s(\xR^d)}&\le C \|\tilde{f}'_{\omega_n}(u_n)-\tilde{f}'_{\omega_n}(u)\|_{L^s(\xR^d)} \\
    &\leq C \|\psi_0\|_{L^\infty(\xR^d)}^{2\alpha -1}\|u_n-u\|_{L^s(\xR^d)}
\end{align*}
using that $\max\{\|u_n\|_{L^{\infty}(\xR^d)},\|u\|_{L^{\infty}(\xR^d)}\}\le C\|\psi_0\|$ for $n$ large enough. Consequently, as $n\to +\infty$ we have  $\|(\tilde{f}'_{\omega_n}(u_n)-\tilde{f}'_{\omega_n}(u))v\|_{L^s(\xR^d)}\to 0$.

Next, we have to prove that 
$\|(\tilde{f}'_{\omega_n}(u)-\tilde{f}'_\omega(u))v\|_{L^s(\xR^d)}\to 0$ as $n\to +\infty$.

For any $\omega_n,\omega\in [-\delta,\delta]$ and for any $t\in \bigcap_n(-t^*_{\omega_n},t^*_{\omega_n})$, we have 
\begin{align*}
    |\tilde f'_{\omega_n}(t)-\tilde f'_{\omega}(t)|
    =&\, |\left(2\alpha+1+\omega(\alpha+1)-(3\alpha+1)({\omega_n}\Lambda^{2\alpha}_{\omega_n}(t)+{\omega}\Lambda^{2\alpha}_{\omega}(t))\right)\left(\Lambda^{2\alpha}_{\omega_n}(t)-\Lambda^{2\alpha}_\omega(t)\right)\\
    &\,+({\omega_n}-\omega)\Lambda^{2\alpha}_{\omega_n}(t)\left(\alpha+1 
    -(3\alpha+1)\Lambda^{2\alpha}_\omega(t)\right)|\\
    \le&\, C_{\alpha}\left((1+|\omega|)|\Lambda^{2\alpha}_{\omega_n}(t)-\Lambda^{2\alpha}_\omega(t)|+|\omega-\omega_n||t|^{2\alpha}(1+|t|^{2\alpha})\right)
\end{align*}
for some positive constant $C_\alpha$ that depends only on $\alpha$, thanks to Proposition \ref{prop_change_variable_omega0}. Hence, if $0<2\alpha\le 1$, we obtain 
\begin{equation*}
    |\tilde f'_{\omega_n}(t)-\tilde f'_{\omega}(t)|\le C_{\alpha}\left((1+|\omega|)|\Lambda_{\omega_n}(t)-\Lambda_\omega(t)|^{2\alpha}+|\omega-\omega_n||t|^{2\alpha}(1+|t|^{2\alpha})\right)
\end{equation*}
and 
\begin{align*}
    \|(\tilde{f}'_{\omega_n}(u)-\tilde{f}'_{\omega}(u))v\|_{L^s(\xR^d)}
    &\le
     C_{\alpha}\|\Lambda_{\omega_n}(u)-\Lambda(u)\|_{L^s(\xR^d)}^{2\alpha}\|v\|_{L^\infty(\xR^d)}^{2\alpha}\|v\|_{L^s(\xR^d)}^{1-2\alpha}\\
     &\,\,\,+ C_{\alpha}|\omega_n-\omega| \|u\|^{2\alpha}_{L^{\infty}(\xR^d)}(\|u\|^{2\alpha}_{L^\infty(\xR^d)}+1)\|v\|_{L^s(\xR^d)}.
\end{align*}
If $2\alpha>1$, we have
\begin{equation*}
    |\tilde f'_{\omega_n}(t)-\tilde f'_{\omega}(t)|\le C_{\alpha}\left((1+|\omega|)|t|^{2\alpha-1}|\Lambda_{\omega_n}(t)-\Lambda_\omega(t)|+|\omega-\omega_n||t|^{2\alpha}(1+|t|^{2\alpha})\right)
\end{equation*}
and 
\begin{align*}
    \|(\tilde{f}'_{\omega_n}(u)-\tilde{f}'_{\omega}(u))v\|_{L^s(\xR^d)}&\le
     C_{\alpha}\|\Lambda_{\omega_n}(u)-\Lambda(u)\|_{L^s(\xR^d)}\|u\|_{L^\infty(\xR^d)}^{2\alpha-1}\|v\|_{L^\infty(\xR^d)}\\
     &\,\,\, + C_{\alpha}|\omega_n-\omega|\|u\|^{2\alpha}_{L^{\infty}(\xR^d)}(\|u\|^{2\alpha}_{L^\infty(\xR^d)}+1)\|v\|_{L^s(\xR^d)}.
\end{align*}
In both cases, $\|(\tilde{f}'_{\omega_n}(u)-\tilde{f}'_\omega(u))v\|_{L^s(\xR^d)}\to 0$ as $n\to +\infty$ thanks to Lemma \ref{convergence_uniform_Lambda_omega}.

This shows that the application $$(\omega,u)\mapsto D_u\Phi(\omega,u)$$ is continuous as an operator in $\mathcal{L}(W^{2,q}(\mathbb{R}^d) \cap H^2_{\mathrm{rad}}(\mathbb{R}^d), L^q(\xR^d)\cap L^2_\mathrm{rad}(\xR^d)) $. 

Finally, it remains to prove that $D_u\Phi(0,\psi_0)$ is an isomorphism.
Note that $D_u \Phi(0,\psi_0)= -\Delta - \tilde{f}_0'(\psi_0)= \mathcal{L}_{0}$,
 where $\mathcal{L}_{0}$ is the linearisation of \eqref{PSSL_0} at $\psi_0$. Since $\psi_0$ is non-degenerate, see \cite{mcleod1993uniqueness}, then
$$\mathrm{ker}(\mathcal{L}_{0|\mathrm{rad}})=\{0\}.$$
This implies that $D_u \Phi(0,\psi_0)$ is injective into $L^q(\xR^d) \cap L^2_\mathrm{rad}(\xR^d)$.
It remains to show that the mapping  
$$\mathcal{L}_0 =-\Delta  +1 + V_0 \text{ with }  V_0 =- (2\alpha+1) \psi_0^{2\alpha}$$
is surjective. As in the proof of Proposition \ref{prop_regularity}, to conclude it is enough to remark that $V_0\in L^{\infty}(\xR^d)\cap L^p(\xR^d)$ for any $p\ge 1$ and decays exponentially to $0$ at infinity.

By applying the implicit function theorem \cite[Theorem 3.4.10]{krantz2002implicit}, we obtain the existence of a continuous function \( h:\omega \in U\subset (-\delta,\delta) \to h(\omega)\in \mathcal{V} \) defined on a neighbourhood $U$ of $0$, such that \( \Phi(\omega, h(\omega)) = 0 \). In other words, \( h(\omega) \) is a solution of \eqref{PSSL_omega0}. Moreover, for any $\omega\in U$, $\|h(\omega)\|_{L^\infty}< {t^*_\omega}$ and, arguing as in the proof of Proposition \ref{prop_regularity}, we conclude that $h(\omega)$ is positive for $\omega$ close to $0$. This implies  $\tilde{u}_\omega = h(\omega)$ for all $\omega$ close to $0$, and 
\begin{equation*}
    \|\tilde u_{\omega}-\psi_0\|_{W^{2,q}(\xR^d)\cap H^2(\xR^d)}\to 0
\end{equation*}
as $\omega\to 0$.

\end{proof}

From Theorem \ref{thm_convergence_tilde_u_subcritical}, we deduce the following result on the convergence of $\tilde \varphi_\omega$

\begin{corol}\label{cor_convergence_tilde_phi_subcritical}
Suppose $d=2$, or $d \geq 3$ and $\alpha < \frac{2}{d-2}$. Then, as $\omega \to 0, \, \tilde{\varphi}_\omega (x)= \omega^\frac{-1}{2\alpha} \varphi_\omega (\omega^\frac{-1}{2\alpha} x) $ with $0<\varphi_\omega<1$ the unique positive solution to \eqref{SSNL} converges in $L^{\infty}(\xR^d)\cap H^1_{\mathrm{rad}}(\xR^d)$ to $\psi_0$ the unique positive solution of \eqref{limit_RSNL}.
\end{corol}

\begin{proof}
\label{proof_convergence_subcritical}
    For any $\omega\in (0,\omega^*)$, let $0<\varphi_\omega<1$ the unique positive solution to \eqref{SSNL}. The function $\tilde\varphi_\omega(x)=\omega^{-\frac{1}{2\alpha}}\varphi_\omega(\omega^{-\frac{1}{2}}x)$ can be written in terms of $\tilde u_\omega$ as $\tilde\varphi_\omega=\Lambda_\omega(\tilde u_\omega)$.

    Since $\tilde u_\omega\in L^s(\R^d)$ for any $s\ge 2$ and $\|\tilde u_\omega\|_{L^\infty(\xR^d)}<t^*_\omega$, we can apply \eqref{uniform_bound_Lambda_omega} and deduce that
    \begin{equation*}
        \|\Lambda_\omega(\tilde u_\omega)-\tilde u_\omega\|_{L^s(\xR^d)}\le\frac{\omega}{2\alpha}\|\tilde u_\omega\|_{L^\infty(\xR^d)}^{1+2\alpha}\|\tilde u_\omega\|_{L^s(\xR^d)}.
    \end{equation*}
    For $\omega$ close to $0$, $\tilde u_\omega\in \mathcal V$, which implies that $\|\tilde u_\omega\|_{L^\infty(\xR^d)}$ is uniformly bounded by $2\|\psi_0\|_{L^\infty(\xR^d)}$. Moreover, the sequence $\{\|\tilde u_\omega\|_{L^s(\xR^d)}\}_\omega$ is uniformly bounded for $\omega$ close to $0$, since $\tilde u_\omega$ converges to $\psi_0$ in $L^s(\R^d)$ for any $s\ge 2$. As a consequence, 
    \begin{equation}
        \label{convergence_Lambda_u_tilde}
        \|\Lambda_\omega(\tilde u_\omega)-\tilde u_\omega\|_{L^s(\xR^d)}\to 0 \text{ as } \omega\to 0
    \end{equation}
    for any $s\ge 2$. This, together with Theorem \ref{thm_convergence_tilde_u_subcritical}, implies that
    \begin{equation}\label{convergence_psi_Ls}
        \|\tilde\varphi_\omega-\psi_0\|_{L^s(\xR^d)}\to 0 \text{ as } \omega\to 0
    \end{equation}
    for any $s\ge 2$. In particular, $\tilde\varphi_\omega$ converges to $\psi_0$ in $L^{\infty}(\xR^d)$ as $\omega\to 0$.

    To prove the convergence of $\tilde\varphi_\omega$ in $H^{1}(\xR^d)$, it remains to prove that $\|\nabla \tilde\varphi_\omega-\nabla \psi_0\|_{L^2(\xR^d)}\to 0$ as $\omega\to 0$. On the one hand, since $\tilde u_\omega$ converges to $\psi_0$ in $H^1(\xR^d)$, we have $\|\nabla \tilde u_\omega-\nabla \psi_0\|_{L^2(\xR^d)}\to 0$ as $\omega\to 0$ and the sequence $\{\|\nabla \tilde u_\omega\|_{L^2(\xR^d)}\}_{\omega}$ is uniformly bounded for $\omega$ close to $0$. On the other hand, 
    \begin{align*}
        \int_{\xR^d}|\nabla \tilde\varphi_\omega-\nabla \tilde u_\omega|^2\,\xd x\le&\, \int_{\xR^d}(1-\Lambda'_\omega(\tilde u_\omega))^2|\nabla \tilde u_\omega|^2\,\xd x\\ 
        \le& \omega^2\int_{\xR^d}\Lambda^{4\alpha}_\omega(\tilde u_\omega)|\nabla \tilde u_\omega|^2\,\xd x\\
        \le& \omega^2\int_{\xR^d}|\tilde u_\omega|^{4\alpha}|\nabla \tilde u_\omega|^2\,\xd x \\
        \le &\, \omega^2 (2\|\psi_0\|_{L^\infty(\xR^d)})^{4\alpha}\|\nabla \tilde u_\omega\|^2_{L^2(\R^d)} \to 0
    \end{align*}
    as $\omega\to 0$.

    As a conclusion $\tilde\varphi_\omega$ converges to $\psi_0$ in $H^1(\xR^d)\cap L^{\infty}(\xR^d)$.
\end{proof}

We conclude this section with the asymptotic behavior of the mass $\varphi_\omega$ as $\omega \to 0$.

\begin{prop}[Asymptotic behaviour of the mass of $\varphi_\omega$ as $\omega \to 0$]
Let $\varphi_\omega$ be the solution to \eqref{SSNL}, the mass $M(\omega)$ has the following asymptotic behaviour as $\omega \to 0$.
\begin{itemize}
    \item \textbf{$L^2$-critical case} ($\alpha = \frac{2}{d}$): 
    \[
    M(\omega) \xrightarrow[\omega \to 0]{} \|\psi_0\|_{L^2(\mathbb{R}^d)}^2
    \]
    where $\psi_0$ is the unique positive solution to \eqref{limit_RSNL}
    \item \textbf{$L^2$-super-critical case} ($\alpha > \frac{2}{d}$): 
    \[
    M(\omega) \xrightarrow[\omega \to 0]{} +\infty.
    \]
    
    \item \textbf{$L^2$-sub-critical case} ($\alpha < \frac{2}{d}$): 
    \[
    M(\omega) \xrightarrow[\omega \to 0]{} 0.
    \]
\end{itemize}
\end{prop}

\begin{proof}
We have that 
\begin{align*}
   M(\omega)= \|\varphi_\omega\|^2_{L^2(\xR^d)} = \omega^{\frac{1}{\alpha}-\frac{d}{2}} \int_{\xR^d} |\tilde\varphi_\omega|^2 \xd x 
\end{align*}

Thanks to \eqref{convergence_psi_Ls}, we have
\begin{equation}
\label{convergence_mass_phi}
   \omega^{-(\frac{1}{\alpha}-\frac{d}{2})}\|\varphi_\omega\|^2_{L^2(\xR^d)} \xrightarrow[\omega \to 0]{}  \|\psi_0\|^2_{L^2(\xR^d)}.
\end{equation}
This leads to the following three different cases
\begin{itemize}
    \item $L^2$-critical case $\alpha = \frac{2}{d}$: $\|\varphi_\omega\|^2_{L^2(\xR^d)} = \|\tilde \varphi_\omega\|^2_{L^2(\xR^d)} \xrightarrow[\omega \to 0]{} \|\psi_0\|^2_{L^2(\xR^d)}$.
    
    \item $L^2$-super-critical case $\alpha>\frac{2}{d}$: $\|\varphi_\omega\|^2_{L^2(\xR^d)} \xrightarrow[\omega \to 0]{} +\infty$.
    
    \item $L^2$-sub-critical case $\alpha<\frac{2}{d}$: $\|\varphi_\omega\|^2_{L^2(\xR^d)} \xrightarrow[\omega \to 0]{} 0$.
\end{itemize}

\end{proof}

\subsection{Critical case}

In  this section we are going to study the behaviour of the positive solution to \eqref{SSNL}, for $\alpha= \frac{2}{d-2}$ and as $\omega$ goes to $0$. Our goal is to prove the following theorem.
\begin{thrm}
\label{Asymptotic_behaviour_PHI_critical}
Let $\varphi_\omega$ be the positive solution to \eqref{SSNL}. There exists a function $\omega\mapsto\lambda_\omega \in (0,+\infty)$ such that, as $\omega \to 0^+$, $\lambda_\omega\to +\infty$ and the rescaled function  $\tilde \varphi_\omega (x)= \lambda_\omega^{\frac{d-2}{2}} \varphi(\lambda_\omega x)$
converges to the  positive solution $S$ to 
\begin{equation}\label{eq_critical_limit}
    -\Delta S = S^{\frac{d+2}{d-2}},
\end{equation}
in $\dot{H}^1(\xR^d)\cap L^{\frac{2d}{d-2}}(\xR^d)$.
Moreover,
\begin{equation*}
    \lim_{\omega \to 0} \|\varphi_{\omega} \|_{L^2(\mathbb{R}^d)} = +\infty.
\end{equation*}
\end{thrm}
To prove this result, we first establish a similar one for $u_\omega$ the solution of the semi-linear problem \eqref{PSSL}, involving the critical power of the non-linearity, as stated in the following theorem.

\begin{thrm}
\label{thrm_convergence_critical}
Suppose that $d\geq 3$ and $\alpha=\frac{2}{d-2}$. Let $u_\omega$ the positive solution to \eqref{PSSL}. Then, there exists a function $\omega\mapsto\lambda_\omega \in (0,+\infty)$ such that, as $\omega \to 0^+$, the rescaled function $\lambda_\omega^{\frac{d-2}{2}}u_\omega(\lambda_\omega x)$ converges in $\dot{H}^1(\mathbb{R}^d)\cap L^\frac{2d}{d-2}(\xR^d)$ to $S$ the positive solution of \eqref{eq_critical_limit}.
\end{thrm}

\begin{proof}
The idea is to use the variational characterisation as in super-critical case (see Section \ref{sec:SCC}). In particular, we consider $v_\omega$ a minimiser of \eqref{Pb_optim_I_omega} and we prove that, up to a dilation, $(v_\omega)_\omega$ is a minimising sequence for 
\begin{equation}
\label{Pb_optim_I_critical}
I_{\#}= \inf \left\{\int_{\xR^d}|\nabla u|^2  \xd x :\quad   u\in \dot{H}^1(\xR^d),  \quad \int_{\xR^d}u^{\frac{2d}{d-2}}\xd x= 1\right\}
\end{equation}
From \cite{lions1985concentration}, we know that every minimising sequence of \(I_\#\)  is relatively compact (up to dilations and translations) in $\dot H^1(\xR^d)\cap L^{\frac{2d}{d-2}}(\xR^d)$ and converges, up to a subsequence, to a minimiser of \(I_\#\). Moreover, \(I_\#\) is attained by the family of positive radial functions 
\[
\{W_\lambda\}_{\lambda > 0} \subset \dot{H}^1(\mathbb{R}^N), \quad W(x):=W_1(x) = S(\sqrt{I_\#}x)\]
with $S(x) := \left( 1+\frac{|x|^2}{d(d-2)}\right)^{1-\frac{d}{2}}$ the Aubin-Talenti function and $W_\lambda=\lambda^{-\frac{d-2}{2}} W(\lambda^{-1}x)$. Each $W_\lambda$ is solution to the following equation
\begin{equation}
\label{equation_W}
    -\Delta w = I_\# w^{\frac{d+2}{d-2}} \quad \text{in } \mathbb{R}^d,
\end{equation}
and verifies 
\begin{equation}
\label{propreties_W}
    \int_{\mathbb{R}^d} |\nabla W_{\lambda}|^2 \,\xd x = \int_{\mathbb{R}^d} |\nabla W|^2 \,\xd x = I_\#, \quad
\text{and} \quad
    \int_{\mathbb{R}^d} W_{\lambda}^{\frac{2d}{d-2}} \,\xd x = \int_{\mathbb{R}^d} W^{\frac{2d}{d-2}} \,\xd x = 1, \quad \forall \lambda > 0.
\end{equation}
Similarly to the super-critical case, we consider the following functional

$$ J_\#(u)= \frac{\int_{\xR^d}|\nabla u|^2 \xd x}{\left(\int_{\xR^d}u^{\frac{2d}{d-2}}\xd x\right)^{\frac{d-2}{d}}}$$
defined on 
$$K_\#= \{u\in \dot{H}^1(\xR^d):   \int_{\xR^d} u^{\frac{2d}{d-2}}\xd x >0 \}.$$
Noticing that $J_\#$ is invariant under the scaling $u(x) \mapsto u_\mu(x) :=u(\mu x)$ and using the same argument as in the super-critical case we can show that 

\begin{equation}
\label{PB_optim_IJcritical}
    I_\#=\inf\limits_{u\in K_{\#}} J_\#(u).
\end{equation}
As in the super-critical case, we first show that as \( \omega \to 0 \), there holds  
\begin{equation}\label{limitIomegacritical}
1 \leq \frac{I_{\omega}}{I_\#} \leq 1 + o(1).
\end{equation}
Let $\omega>0$ and $v_{\omega}$ be a minimiser of \eqref{Pb_optim_I_omega}. We have $v_\omega \in H^1 $, $v_\omega\ge 0$,
\[
\int_{\mathbb{R}^d} |\nabla v_{\omega}|^2 \,\xd x = I_{\omega} \quad\text{and} \quad \frac{2d}{d-2}\int_{\mathbb{R}^d} F_{\omega}(v_{\omega}) \,\xd x=1.
\]
Using that $0\le \Lambda(t)\leq t$ for all $t\in \xR^+$, we obtain
\begin{align}
\frac{d-2}{2d}=\int_{\mathbb{R}^d} F_{\omega}(v_{\omega}) \,\xd x =&  \int_{\mathbb{R}^d}\left(\frac{d-2}{2d}
 \Lambda(v_{\omega})^{\frac{2d}{d-2}} - \frac{\omega}{2} \Lambda(v_{\omega})^2 \right)\,\xd x \notag\\
&< \frac{d-2}{2d} \int_{\mathbb{R}^d} \Lambda(v_{\omega})^{\frac{2d}{d-2}} \,\xd x \notag\\
&\leq \frac{d-2}{2d}\int_{\mathbb{R}^d} v_{\omega}^{\frac{2d}{d-2}} \,\xd x.
\end{align}
Then, we have 
\begin{equation}
\label{bounds_optimums}
I_\# \leq J_\#(v_{\omega}) = \frac{\int_{\mathbb{R}^d} |\nabla v_{\omega}|^2 \xd x }{\left(\int_{\mathbb{R}^d} v_{\omega}^{\frac{2d}{d-2}}\xd x\right)^{\frac{d-2}{d}}} =
\frac{I_{\omega}}{\left(\int_{\mathbb{R}^d} v_{\omega}^{\frac{2d}{d-2}} \xd x\right)^{\frac{d-2}{d}}} < I_{\omega}.
\end{equation}
Now, for \( d \geq 5 \), we have \( W \in L^2(\xR^d) \), so we can use the family \( \{W_{\lambda} : \lambda > 0\} \subset H^1(\xR^d) \) as test functions for \( J_{\omega} \). Thanks to Proposition \ref{prop_change_var}, we can show that, if  $\lambda \gg 1$ and $ \omega \lambda^2 \ll 1$

\begin{align}\label{lowerboundFcritical}
\int_{\xR^d} F_{\omega}(W_{\lambda}) \,\xd x 
&= \frac{d-2}{2d} \int_{\xR^d} \Lambda(W_{\lambda})^{\frac{2d}{d-2}} \,\xd x - \frac{\omega}{2} \int_{\xR^d} \Lambda(W_{\lambda})^2 \,\xd x \nonumber\\
 &\underbrace{\geq}_{\ref{prop_change_var}\ref{prop_change_var4}} \frac{d-2}{2d}  \int_{\xR^d} W_{\lambda}^{\frac{2d}{d-2}} \Lambda'(W_{\lambda})^{\frac{2d}{d-2}} \,\xd x - \frac{\omega}{2} \int_{\xR^d} W_{\lambda}^2 \,\xd x \nonumber\\
 &\geq \frac{d-2}{2d}  \int_{\xR^d} W_{\lambda}^{\frac{2d}{d-2}} (1-W_{\lambda}^{\frac{4}{d-2}})^{\frac{d}{d-2}} \,\xd x - \frac{\omega}{2} \int_{\xR^d} W_{\lambda}^2 \,\xd x \nonumber\\
 &\geq \frac{d-2}{2d} \int_{\xR^d} W^{\frac{2d}{d-2}} (1-\lambda^{-2} W^{\frac{4}{d-2}})^{\frac{d}{d-2}} \,\xd x - \frac{\omega\lambda^2}{2} \int_{\xR^d} W^2 \,\xd x \nonumber\\
 &\underbrace{\geq}_{\eqref{propreties_W}} \frac{d-2}{2d} (1-\lambda^{-2})^\frac{d}{d-2} - \frac{\omega\lambda^2}{2} \int_{\xR^d} W^2 \,\xd x >0.
\end{align}
As a consequence \( W_{\lambda} \in K_{\omega} \). Using \eqref{propreties_W} and choosing $\lambda= \omega^{-\frac{1}{d}}$, we obtain
\begin{align*}
I_{\omega} \leq J_{\omega}(W_{\lambda}) &=
\frac{\int_{\xR^d} |\nabla W|^2 \,\xd x}{\left(\int_{\xR^d}  W^{\frac{2d}{d-2}} \,\xd x\right)^{\frac{d-2}{d}}}
\left(\frac{\int_{\xR^d}  W^{\frac{2d}{d-2}}\,\xd x}{ \frac{2d}{d-2}\int_{\xR^d} F_\omega(W_\lambda(x))\,\xd x}\right)^{\frac{d-2}{d}}\\
&= I_\# \left(\frac{1}{ \frac{2d}{d-2}\int_{\xR^d} F_\omega(W_\lambda(x))\,\xd x}\right)^{\frac{d-2}{d}}\leq I_\# \left( \frac{1}{(1-\lambda^{-2})^\frac{d}{d-2} - \frac{d}{d-2}\omega\lambda^2 \int_{\xR^d} W^2 \,\xd x }\right)^{\frac{d-2}{d}}\\
&\leq I_\# \left(1+O(\lambda^{-2})+ O(\omega\lambda^2)\right)^{-\frac{d-2}{d}}\\
&=I_\#\,(1+O(\omega^{\frac{2}{d}} )+O(\omega^{1-\frac{2}{d}}))= I_{\#}\,(1+o(1)).
\end{align*}
For \( d \in \{3, 4\} \), we use the same cut-off function \( \eta_R \) as in the proof of Theorem \ref{Thrm_super_critical}. We have \( \eta_R W_\lambda \in H^1(\mathbb{R}^d) \), and we use it as a test function for \( J_\omega \). We obtain
\[
I_\omega \leq J_\omega(\eta_R  W_\lambda) =
\frac{\int_{\xR^d} |\nabla (\eta_R  W_\lambda)|^2 \, \xd x}{\left(
\int_{\xR^d}(\eta_R  W_\lambda)^{\frac{2d}{d-2}} \, \xd x\right)^{1-\frac{2}{d}}}
\left( \frac{\int_{\xR^d}(\eta_R  W_\lambda)^{\frac{2d}{d-2}} \, \xd x}{\frac{2d}{d-2}\int_{\xR^d} F_\omega(\eta_R  W_\lambda) \, \xd x} \right)^{1-\frac{2}{d}}.
\]
First, note that by dominated convergence we have as $R \to +\infty$ 
\[
\int_{\mathbb{R}^d} |\nabla \eta_R W_\lambda|^2 \, \xd x \to \int |\nabla W_\lambda|^2 \, \xd x = I_\#.
\]
Next, proceeding as in the Proof of Theorem \ref{Thrm_super_critical} and using the same bounds as in~\eqref{lowerboundFcritical}, we have
\begin{align*}
\int_{\mathbb{R}^d} \Lambda(\eta_R W_\lambda(x))^{\frac{2d}{d-2}} \, \xd x &= \int_{\mathbb{R}^d} \Lambda(W_\lambda(x))^{\frac{2d}{d-2}} \, \xd x + O((\lambda^{-1}R)^{d-(\frac{2d}{d-2})(d-2)}) \\
&\ge (1-\lambda^{-2})^\frac{d}{d-2} + O((\lambda^{-1}R)^{-d}) .
\end{align*}
provided that $\frac{R}{\lambda}\gg 1$. Similarly, using~\eqref{propreties_W}, we obtain
\[
\int_{\mathbb{R}^d} (\eta_R W_\lambda(x))^{\frac{2d}{d-2}} \, \xd x = \int_{\mathbb{R}^d} (W_\lambda(x))^{\frac{2d}{d-2}} \, \xd x + O((\lambda^{-1}R)^{d-(\frac{2d}{d-2})(d-2)})= 1 + O((\lambda^{-1}R)^{-d})
\]
whenever $\frac{R}{\lambda}\gg 1$.
Finally, from Proposition \ref{prop_change_var} and Lemma \ref{case_3_4}, we deduce that
\[
h_d(\lambda,R)=\int_{\mathbb{R}^d} \Lambda(\eta_R W_\lambda)^2 \, \xd x \leq \int_{\mathbb{R}^N} (\eta_R W_\lambda)^2 \,\xd x =
\begin{cases}
O(\lambda R), & \text{if } d = 3, \\
O(\lambda^2\log(\lambda^{-1}R )), & \text{if } d = 4,
\end{cases}
\text{ as } R \to \infty. 
\]
As a consequence, we have 
\begin{align*}
 \frac{\int_{\xR^d}(\eta_R  W_\lambda)^{\frac{2d}{d-2}} \, \xd x}{\int_{\xR^d} \Lambda^{\frac{2d}{d-2}}_\omega(\eta_R  W_\lambda)\xd x - \frac{d}{d-2}\omega\int_{\xR^d}\Lambda^{2}_\omega(\eta_R  W_\lambda) \, \xd x}\leq \frac{1+O((\lambda^{-1}R)^{-d})}{(1-\lambda^{-2})^\frac{d}{d-2} + O((\lambda^{-1}R)^{-d})- \omega \frac{d}{d-2} h_d(\lambda,R)}.
\end{align*}
For $d=3$, we can choose $\lambda=\omega^{-\frac{1}{3}}$ and $R=\omega^{-\frac{1}{2}}$, so that $R\lambda^{-1}=\omega^{-\frac{1}{6}}\to +\infty$ as $\omega$ goes to $0$, and
\begin{align*}
    \left( \frac{\int_{\xR^d}(\eta_R  W_\lambda)^{\frac{2d}{d-2}} \, \xd x}{\frac{2d}{d-2}\int_{\xR^d} F_\omega(\eta_R  W_\lambda) \, \xd x} \right)^{1-\frac{2}{d}}&\le (1+O(\omega^{\frac{1}{2}}))(1+O(\omega^{\frac{2}{3}})+O(\omega^{\frac{1}{2}})+O(\omega^{\frac{1}{6}}))^{-\frac13}\\
    &=(1+O(\omega^{\frac{1}{2}}))(1+O(\omega^{\frac{1}{6}}))=1+o(1).
\end{align*}
For $d=4$, we can choose $\lambda=\omega^{-\frac{1}{4}}$ and $R=\omega^{-\frac{1}{2}}$, so that $R\lambda^{-1}=\omega^{-\frac{1}{4}}\to +\infty$ as $\omega$ goes to $0$, and
\begin{align*}
    \left( \frac{\int_{\xR^d}(\eta_R  W_\lambda)^{\frac{2d}{d-2}} \, \xd x}{\frac{2d}{d-2}\int_{\xR^d} F_\omega(\eta_R  W_\lambda) \, \xd x} \right)^{1-\frac{2}{d}}&\le (1+O(\omega))(1+O(\omega^{\frac{1}{2}})+O(\omega)+O(\omega^{\frac{1}{2}}\log(\omega^{-1}))^{-\frac{1}{2}}\\
    &=(1+O(\omega))(1+O(\omega^{\frac{1}{2}}\log(\omega^{-1}))=1+o(1).
\end{align*}
As a conclusion, 
\begin{align*}
    I_\omega \leq 
\frac{\int_{\xR^d} |\nabla (\eta_R  W_\lambda)|^2 \, \xd x}{\left(
\int_{\xR^d}(\eta_R  W_\lambda)^{\frac{2d}{d-2}} \, \xd x\right)^{1-\frac{2}{d}}}(1+o(1))\le I_\#\,(1+o(1)).
\end{align*}
This proves~\eqref{limitIomegacritical} for any $d\ge 3$.

Next, we show that  
\begin{equation*}
    \int_{\mathbb{R}^d} |\nabla v_\omega|^2 \,\xd x \to I_\#  
    \quad \text{and} \quad
    \int_{\mathbb{R}^d} |v_\omega|^{\frac{2d}{d-2}} \,\xd x \to 1.  
\end{equation*}
First, we note that  
\[
\int_{\mathbb{R}^d} |\nabla v_\omega|^2 \,\xd x = I_\omega.
\]
From \eqref{limitIomegacritical}, we have \( I_\omega \to I_\# \), which directly implies  
\[
\int_{\mathbb{R}^d} |\nabla v_\omega|^2 \,\xd x \to I_\#.
\]

Next, taking the limit as \( \omega \to 0 \) in \eqref{bounds_optimums} and using \( I_\omega \to I_\# \), we obtain  
\begin{equation}
    \label{convergence_nonlinearty_critical}
\int_{\mathbb{R}^d} v_\omega^{\frac{2d}{d-2}} \,\xd x \to 1.
\end{equation}
Thus, \( v_\omega \) is a minimising sequence for \eqref{Pb_optim_I_critical}. Then by the concentration-compactness principle \cite{lions1985concentration}
there exist \( (\lambda_\omega)_\omega \) in \( (0, \infty) \) such that the new minimising sequence  
\[
\tilde{v}_\omega = \lambda_\omega^{\frac{d-2}{2}} v_\omega \left( \lambda_\omega \cdot\right)
\]  is relatively compact in  \( \dot{H}_{1}(\xR^d)\cap L^{\frac{2d}{d-2}}(\xR^d) \).  Then $\tilde{v}_\omega$  converges to $W$ in $\dot{H}_{1}(\xR^d)\cap L^{\frac{2d}{d-2}}(\xR^d)$.

To conclude, we define $\tilde u_\omega= \lambda_\omega^{\frac{d-2}{2}} u_\omega \left( \lambda_\omega \cdot\right)$. Since $u_\omega=v_\omega(I_\omega^{-\frac{1}{2}}\cdot)$, we deduce that $\tilde u_\omega=\tilde v_\omega(I_\omega^{-\frac{1}{2}}\cdot)$. As a consequence, using that $W(x)=S(\sqrt{I_{\#}}x)$ for all $x\in \R^d$, we obtain 
\begin{align*}
    \|\tilde u_\omega-S\|_{L^{\frac{2d}{d-2}}(\xR^d)}\le&\, \|\tilde v_\omega(I_\omega^{-\frac{1}{2}}\cdot)-W(I_\omega^{-\frac{1}{2}}\cdot)\|_{L^{\frac{2d}{d-2}}(\xR^d)}+\|W(I_\omega^{-\frac{1}{2}}\cdot)-S\|_{L^{\frac{2d}{d-2}}(\xR^d)}\\
    =&\,I_\omega^{\frac{d-2}{4}}\|\tilde v_\omega-W\|_{L^{\frac{2d}{d-2}}(\xR^d)}+\|S((I_{\#}/I_\omega)^{\frac{1}{2}}\cdot)-S\|_{L^{\frac{2d}{d-2}}(\xR^d)}.
\end{align*}
Using dominated convergence theorem for the second term, we conclude that $\|\tilde u_\omega-S\|_{L^{\frac{2d}{d-2}}(\xR^d)}\to 0$ as $\omega\to 0$.

Next, using that $\tilde{v}_\omega$  converges to $W$ in $\dot{H}_{1}(\xR^d)$ and $\lim\limits_{\omega\to 0}I_\omega=I_\#$, we have $$\lim\limits_{\omega\to 0}\|\nabla \tilde u_\omega\|_{L^2(\R^d)}^2=\|\nabla S\|_{L^2(\R^d)}^2.$$ As a consequence,
\begin{align*}
   \|\nabla \tilde u_\omega-\nabla S\|_{L^2(\R^d)}^2&=\|\nabla \tilde u_\omega\|_{L^2(\R^d)}^2+ \|\nabla S\|_{L^2(\R^d)}^2-2\int_{\R^d}\nabla \tilde u_\omega\cdot\nabla S\,\xd x\\
    &=\|\nabla \tilde u_\omega\|_{L^2(\R^d)}^2+ \|\nabla S\|_{L^2(\R^d)}^2-2\left(\int_{\R^d} S( -\Delta S)\,\xd x+\int_{\R^d} (\tilde u_\omega-S)(-\Delta S)\,\xd x\right)\\
    &=\|\nabla \tilde u_\omega\|_{L^2(\R^d)}^2- \|\nabla S\|_{L^2(\R^d)}^2-2\int_{\R^d} (\tilde u_\omega-S) S^{\frac{d+2}{d-2}}\,\xd x \to 0
\end{align*}
as $\omega\to 0$, since 
\begin{align*}
    \left|\int_{\R^d} (\tilde u_\omega-S) S^{\frac{d+2}{d-2}}\,\xd x\right|\le \|\tilde u_\omega-S\|_{L^{\frac{2d}{d-2}}(\R^d)}\|S\|_{L^{\frac{2d}{d-2}}(\R^d)}^{\frac{d+2}{d-2}}\to 0.
\end{align*}
This conclude the proof of the theorem.
\end{proof}

Now we are able to prove Theorem \ref{Asymptotic_behaviour_PHI_critical}.

\begin{proof}[Proof of Theorem \ref{Asymptotic_behaviour_PHI_critical}]
Proving this result boils down to show that the dilatation $\lambda_\omega \to \infty $ when $\omega \to 0$.  To this goal we analyse the behaviour of solutions to the rescaled semi-linear equation using Pohozaev and Nehari type identities. The rescaled function $\tilde v_\omega$ solves the equation 
\begin{equation} \label{rescaled_semilinear}
    -\Delta \tilde{v}_\omega = \lambda_\omega ^{\frac{d+2}{2}} I_\omega f_\omega\left(\lambda_\omega^{-\frac{d-2}{2}} \tilde{v}_\omega\right).
\end{equation}
Applying the Pohozaev identity to \eqref{rescaled_semilinear} yields:
\begin{align}
\label{Pohosaev_critical}
    \int_{\mathbb{R}^d} |\nabla \tilde{v}_\omega|^2 \, \xd x &= \frac{2d}{d-2} I_\omega \lambda_\omega^d \int_{\mathbb{R}^d} F_\omega\left( \lambda_\omega^{-\frac{d-2}{2}} \tilde{v}_\omega \right) \, \xd x \notag \\
    &= I_\omega \lambda_\omega^d \int_{\mathbb{R}^d} \Lambda_\omega^{\frac{2d}{d-2}}\left( \lambda_\omega^{-\frac{d-2}{2}} \tilde{v}_\omega \right) \, \xd x - \frac{d}{d-2} \omega I_\omega \lambda_\omega^d \int_{\mathbb{R}^d} \Lambda^2\left( \lambda_\omega^{-\frac{d-2}{2}} \tilde{v}_\omega \right) \, \xd x.
\end{align}
To derive a Nehari type identity, we multiply both sides of \eqref{rescaled_semilinear} by $\left( \frac{\Lambda(\lambda_\omega^{-\frac{d-2}{2}}\tilde{v}_\omega)}{\Lambda'(\lambda_\omega^{-\frac{d-2}{2}}\tilde{v}_\omega)} \right)$ and integrate, knowing that $\Lambda'(t)=\sqrt{1-\Lambda^{2\alpha}(t)}$ and $\Lambda''(t)=-\alpha \Lambda^{2\alpha-1}(t)$ for any $t\in (0,t^*)$. We have 
\begin{align*}
\int_{\mathbb{R}^d} (-\Delta \tilde{v}_\omega) 
\left( \frac{\Lambda(\lambda_\omega^{-\frac{d-2}{2}}\tilde{v}_\omega)}{\Lambda'(\lambda_\omega^{-\frac{d-2}{2}}\tilde{v}_\omega)} \right) \xd x 
=& \lambda_\omega^{-\frac{d-2}{2}}\int_{\mathbb{R}^d} |\nabla \tilde{v}_\omega|^2 
\left( 1 + \frac{\alpha \Lambda^{2\alpha} 
\left( \lambda_\omega^{-\frac{d-2}{2}} \tilde{v}_\omega \right)}
{\left( \Lambda' \left( \lambda_\omega^{-\frac{d-2}{2}} 
\tilde{v}_\omega \right) \right)^2} \right) \xd x \\
=& I_\omega \lambda_\omega^{\frac{d+2}{2}} 
\int_{\mathbb{R}^d} \Lambda^{\frac{2d}{d-2}} 
\left( \lambda_\omega^{-\frac{d-2}{2}} \tilde{v}_\omega \right) \xd x \\
 &- \omega I_\omega \lambda_\omega ^{\frac{d+2}{2}}
\int_{\mathbb{R}^d} \Lambda^2 
\left( \lambda_\omega^{-\frac{d-2}{2}} \tilde{v}_\omega \right) \xd x, 
\end{align*}
which leads to 
\begin{align}
\label{nehari_critical}
\int_{\mathbb{R}^d} |\nabla \tilde{v}_\omega|^2 
\left( 1 + \frac{\alpha \Lambda^{2\alpha} 
\left( \lambda_\omega^{-\frac{d-2}{2}} \tilde{v}_\omega \right)}
{\left( \Lambda' \left( \lambda_\omega^{-\frac{d-2}{2}} 
\tilde{v}_\omega \right) \right)^2} \right) \xd x 
=& I_\omega \lambda_\omega^{d} 
\int_{\mathbb{R}^d} \Lambda^{\frac{2d}{d-2}} 
\left( \lambda_\omega^{-\frac{d-2}{2}} \tilde{v}_\omega \right) \xd x  \nonumber\\ 
&- \omega I_\omega \lambda_\omega^d
\int_{\mathbb{R}^d} \Lambda^2 
\left( \lambda_\omega^{-\frac{d-2}{2}} \tilde{v}_\omega \right) \xd x. 
\end{align}
Subtracting the Pohozaev identity \eqref{Pohosaev_critical} from the Nehari-type identity \eqref{nehari_critical}, we obtain
\begin{align}\label{Pohozaev_Nehari_critical}
    \int_{\mathbb{R}^d} |\nabla \tilde{v}_\omega|^2 \left( \frac{\alpha \Lambda^{2\alpha} \left( \lambda_\omega^{-\frac{d-2}{2}} \tilde{v}_\omega \right)}{\left( \Lambda' \left( \lambda_\omega^{-\frac{d-2}{2}} \tilde{v}_\omega \right) \right)^2} \right) \, \xd x 
    &= \frac{2}{d-2} \omega I_\omega \lambda_\omega^d \int_{\mathbb{R}^d} \Lambda^2 \left( \lambda_\omega^{-\frac{d-2}{2}} \tilde{v}_\omega \right) \, \xd x.
\end{align}
Suppose, by contradiction, that there exist $\ell \in [0, \infty)$ such that
\[
\lambda_{\omega} \to \ell \quad \text{as } \omega\to 0 .
\]
up to a subsequence.

\medskip
\noindent\textbf{Case 1:} $\ell = 0$. 

Thanks to Lemma \ref{lem_properties}(\ref{L3}) and the relation \eqref{rescaled_solut_v_omega}, we have 
\[
\|v_\omega\|_{L^\infty} \le t^*,
\]
 It follows that
\[
\left| \lambda_\omega^{\frac{d-2}{2}} v_\omega(\lambda_\omega x) \right| \le t^* \lambda_\omega^{\frac{d-2}{2}} \to 0 \quad \text{for all } x \in \mathbb{R}^d,
\]
as $\omega \to 0$. Therefore, $\tilde{v}_\omega\to 0$ pointwise on $\mathbb{R}^d$, which contradicts the fact that $\tilde{v}_\omega  \to W$ with $W \not\equiv 0$.
\medskip

\noindent\textbf{Case 2:} $\ell \in (0, \infty)$. 

Since $\tilde{v}_\omega$ converges (up to a subsequence) to $W$ a.e. in $\R^d$ and the function 
$$x\mapsto |\nabla \tilde{v}_\omega|^2 \left( \frac{\alpha \Lambda^{2\alpha} \left( \lambda_\omega^{-\frac{d-2}{2}} \tilde{v}_\omega \right)}{\left( \Lambda' \left( \lambda_\omega^{-\frac{d-2}{2}} \tilde{v}_\omega \right) \right)^2} \right)$$ 
is measurable, we can apply Fatou's Lemma to derive the following inequality
\begin{align} \label{ineq:fatou}
    \int_{\mathbb{R}^d} |\nabla W|^2 \left( \frac{\alpha \Lambda^{2\alpha} \left( \ell^{-\frac{d-2}{2}} W \right)}{1 - \Lambda^{2\alpha} \left( \ell^{-\frac{d-2}{2}} W \right)} \right) \, \xd x 
    \leq \liminf_{\omega \to 0} \int_{\mathbb{R}^d} |\nabla \tilde{v}_\omega|^2 \left( \frac{\alpha \Lambda^{2\alpha} \left( \lambda_\omega^{-\frac{d-2}{2}} \tilde{v}_\omega \right)}{\left( \Lambda' \left( \lambda_\omega^{-\frac{d-2}{2}} \tilde{v}_\omega \right) \right)^2} \right) \, \xd x.
\end{align}
Our goal is now to show that the right-hand side of \eqref{ineq:fatou} tends to zero as \( \omega \to 0 \). Thanks to \eqref{Pohozaev_Nehari_critical}, it is enough to show that
\begin{align*}
\omega I_\omega \lambda_\omega^d \int_{\mathbb{R}^d} \Lambda\left( \lambda_\omega^{-\frac{d-2}{2}} \tilde{v}_\omega(x) \right)^2 \, \xd x
&= I_\omega \omega \lambda_\omega^d \int_{\mathbb{R}^d} \Lambda\left( v_\omega(\lambda_\omega x) \right)^2 \, \xd x \\
&= I_\omega \omega \int_{\mathbb{R}^d} \Lambda(v_\omega(y))^2 \, \xd y \longrightarrow 0 \text{ as } \omega \to 0.
\end{align*}

To establish this convergence, we use \( v_\omega \) as a test function for the functional \( J_\# \). By the definition of \( I_\# \), we have
\[
I_\# \leq J_\#(v_\omega) = \frac{\int_{\mathbb{R}^d} |\nabla v_\omega|^2 \, \xd x}{\|v_\omega\|_{L^{\frac{2d}{d-2}}}^2}
= \frac{I_\omega}{\left( \int_{\mathbb{R}^d} v_\omega^{\frac{2d}{d-2}} \, \xd x \right)^{\frac{d-2}{d}}}
\leq \frac{I_\omega}{\left( \int_{\mathbb{R}^d} \Lambda(v_\omega)^{\frac{2d}{d-2}} \, \xd x \right)^{\frac{d-2}{d}}},
\]
where we used \eqref{Pb_optim_I_omega} and the fact that \( v_\omega \geq \Lambda(v_\omega) \).
Moreover, from the constraint satisfied by \( v_\omega \), we know
\[
1 = \int_{\mathbb{R}^d} \Lambda(v_\omega)^{\frac{2d}{d-2}} \, \xd x - \frac{d}{d-2} \omega \int_{\mathbb{R}^d} \Lambda(v_\omega)^2 \, \xd x.
\]
Rewriting, we obtain
\[
\int_{\mathbb{R}^d} \Lambda(v_\omega)^{\frac{2d}{d-2}} \, \xd x = 1 + \frac{d}{d-2} \omega \int_{\mathbb{R}^d} \Lambda(v_\omega)^2 \, \xd x.
\]
Plugging this into our earlier bound for \( I_\# \), we get
\[
I_\# \leq I_\omega \left( 1 + \frac{d}{d-2}\omega \int_{\mathbb{R}^d} \Lambda(v_\omega)^2 \, \xd x \right)^{-\frac{d-2}{d}}.
\]
which implies
\[
 0\le \frac{d}{d-2}\omega \int_{\mathbb{R}^d} \Lambda(v_\omega)^2 \, \xd x \leq \left( \frac{ I_\omega}{I_\# } \right)^{\frac{d}{d-2}} - 1.
\]
As a consequence, thanks to \eqref{limitIomegacritical},
\begin{equation}
\label{convergence_norm2_changeofvariable}
\omega\int_{\mathbb{R}^d} \Lambda(v_\omega)^2 \, \xd x \to 0 \text{ as } \omega \to 0,
\end{equation}
and 
\[
\omega I_\omega \int_{\mathbb{R}^d} \Lambda(v_\omega)^2 \, \xd x \to 0,
\]
which proves the desired convergence.
Finally,the inequality \eqref{ineq:fatou} leads to 
\[
\int_{\mathbb{R}^d} |\nabla W|^2 \left( \frac{\alpha \Lambda^{2\alpha} \left( \ell^{-\frac{d-2}{2}} W \right)}{ 1- \Lambda^{2\alpha} \left( \ell^{-\frac{d-2}{2}} W\right)} \right) \, \xd x = 0.
\]
As a consequence, we obtain
\[
|\nabla W(x)|^2 \, \alpha \, \Lambda^{2\alpha} \left( \ell^{-\frac{d-2}{2}} W(x) \right) = 0 \quad \text{for all } x \in \mathbb{R}^d.
\]
However, this leads to a contradiction, since \( W \not\equiv 0 \) and \( W \) is not constant, which implies that \( \nabla W \not\equiv 0 \) on a set of positive measure. 

Therefore, we conclude that
\[
\lambda_\omega \to +\infty.
\]

We now aim to prove that the rescaled function $
\tilde{\varphi}_\omega(x) = \lambda_\omega^{\frac{d-2}{2}} \varphi(\lambda_\omega x)$
converges to \( S \), the positive solution of \eqref{eq_critical_limit}, in \( \dot{H}^1(\mathbb{R}^d)\cap L^{\frac{2d}{d-2}}(\xR^d) \). We write
\[
\tilde{\varphi}_\omega(x) = \lambda_\omega^{\frac{d-2}{2}} \Lambda(u_\omega(\lambda_\omega x))=\lambda_\omega^{\frac{d-2}{2}} \Lambda(\lambda_\omega^{-\frac{d-2}{2}}\tilde u_\omega(x)), 
\]
and, we estimate $\| \tilde{\varphi}_\omega - S \|_{\dot{H}^1(\mathbb{R}^d)}$ in the following way
    \begin{align*}
    \| \tilde{\varphi}_\omega - S \|_{\dot{H}^1(\mathbb{R}^d)}=&\,\| \nabla\tilde{\varphi}_\omega - \nabla S \|_{L^2(\mathbb{R}^d)}= \| \Lambda'(\lambda_\omega^{-\frac{d-2}{2}}\tilde u_\omega)\nabla\tilde{u}_\omega - \nabla S \|_{L^2(\mathbb{R}^d)}\\
    \le &\, \| \Lambda'(\lambda_\omega^{-\frac{d-2}{2}}\tilde u_\omega)(\nabla\tilde{u}_\omega - \nabla S) \|_{L^2(\mathbb{R}^d)}+\| (1-\Lambda'(\lambda_\omega^{-\frac{d-2}{2}}\tilde u_\omega))\nabla S \|_{L^2(\mathbb{R}^d)}\\
    \le &\, \| \tilde{u}_\omega - S \|_{\dot{H}^1(\mathbb{R}^d)}+\| (1-\Lambda'(\lambda_\omega^{-\frac{d-2}{2}}\tilde u_\omega))\nabla S \|_{L^2(\mathbb{R}^d)}
\end{align*}
since $\Lambda'(t)<1$ for all $t\in \R^+$ (see Proposition \ref{prop_change_var}\ref{prop_change_var4}). Since $\tilde u_\omega$ converges to $S$ in $\dot{H}^1(\mathbb{R}^d)$, it remains to show that 
\begin{equation*}
    \| (1-\Lambda'(\lambda_\omega^{-\frac{d-2}{2}}\tilde u_\omega))\nabla S \|_{L^2(\mathbb{R}^d)}\to 0 \text{ as } \omega\to 0.
\end{equation*}
We recall that $0< \lambda_\omega^{-\frac{d-2}{2}}\tilde u_\omega(x)= u_\omega(\lambda_\omega x)<t^*$ for all $x\in \xR^d$. Hence, we have
\begin{equation*}
    0< 1-\Lambda'(\lambda_\omega^{-\frac{d-2}{2}}\tilde u_\omega)= 1-\sqrt{1-\Lambda^{2\alpha}(\lambda_\omega^{-\frac{d-2}{2}}\tilde u_\omega)}\le \Lambda^{2\alpha}(\lambda_\omega^{-\frac{d-2}{2}}\tilde u_\omega)
\end{equation*} 
so that 
\begin{equation}\label{estimate_derivative_S}
    \int_{\R^d} (1-\Lambda'(\lambda_\omega^{-\frac{d-2}{2}}\tilde u_\omega))^2|\nabla S|^2\,\xd x\le \int_{\R^d} \Lambda^{4\alpha}(\lambda_\omega^{-\frac{d-2}{2}}\tilde u_\omega)|\nabla S|^2\,\xd x.
\end{equation}
Moreover, from the explicit expression of $S$, $S(x)= \left( 1+\frac{|x|^2}{d(d-2)}\right)^{1-\frac{d}{2}}$ for all $x\in \R^d$, we can easily see that $|\nabla S|\in L^2(\R^d)\cap L^{\infty}(\R^d)$. Finally, since $\tilde u_\omega$ converges to $S$ in $L^{\frac{2d}{d-2}}(\mathbb{R}^d)$, the sequence $\{\|\tilde u_\omega\|_{L^{\frac{2d}{d-2}}(\mathbb{R}^d)}\}$ is uniformly bounded for $n$ large enough. As a consequence, the goal is to use Hölder inequality to deduce that the right hand side of \eqref{estimate_derivative_S} goes to $0$ as $\omega\to 0$. In particular, we want to apply Hölder inequality with $1\le p,q\le \infty$ such that 
\begin{equation*}
    4\alpha p=\frac{2d}{d-2} \Leftrightarrow 2\alpha p=\frac{d}{d-2} \Leftrightarrow p=\frac{d}{4}.
\end{equation*}
Let assume, for the moment, $d\ge 4$ and apply Hölder inequality with $p=\frac{d}{4}$ and $q=\frac{d}{d-4}$ to \eqref{estimate_derivative_S}. We obtain 
\begin{align*}
    \int_{\R^d} \Lambda^{4\alpha}\,(\lambda_\omega^{-\frac{d-2}{2}}\tilde u_\omega)|\nabla S|^2\,\xd x &\le \left(\int_{\R^d} \Lambda^{\frac{2d}{d-2}}(\lambda_\omega^{-\frac{d-2}{2}}\tilde u_\omega)\,\xd x\right)^{\frac{4}{d}}\|\nabla S\|_{L^{\infty}(\R^d)}^{\frac{8}{d}}\|\nabla S\|_{L^{2}(\R^d)}^{2-\frac{8}{d}}\\
    \le&\, \left(\int_{\R^d} (\lambda_\omega^{-\frac{d-2}{2}}\tilde u_\omega)^{\frac{2d}{d-2}}\,\xd x\right)^{\frac{4}{d}}\|\nabla S\|_{L^{\infty}(\R^d)}^{\frac{8}{d}}\|\nabla S\|_{L^{2}(\R^d)}^{2-\frac{8}{d}}\\
    =& \lambda_\omega^{-4}\left(\int_{\R^d} \tilde u_\omega^{\frac{2d}{d-2}}\,\xd x\right)^{\frac{4}{d}}\|\nabla S\|_{L^{\infty}(\R^d)}^{\frac{8}{d}}\|\nabla S\|_{L^{2}(\R^d)}^{2-\frac{8}{d}}
\end{align*}
where we use that $\Lambda(t)\le t$ for all $t\in \R^+$. Since $\lambda_\omega\to +\infty$ as $\omega\to 0$, we deduce
\begin{equation*}
    \int_{\R^d} \Lambda^{4\alpha}(\lambda_\omega^{-\frac{d-2}{2}}\tilde u_\omega)|\nabla S|^2\,\xd x \to 0 \text{ as } \omega\to 0
\end{equation*} 
for any $d\ge 4$. For $d=3$, we have $\alpha=2$ and, using $\Lambda(t)\le \min
\{t,1\}$, we obtain
\begin{align*}
    \int_{\R^d} \Lambda^{4\alpha}(\lambda_\omega^{-\frac{d-2}{2}}\tilde u_\omega)|\nabla S|^2\,\xd x=&\,\int_{\R^d} \Lambda^{2}(\lambda_\omega^{-\frac{d-2}{2}}\tilde u_\omega)\Lambda^{\frac{2d}{d-2}}(\lambda_\omega^{-\frac{d-2}{2}}\tilde u_\omega)|\nabla S|^2\,\xd x\\
    \le& \|\nabla S\|^2_{L^\infty(\xR^d)}\int_{\R^d} (\lambda_\omega^{-\frac{d-2}{2}}\tilde u_\omega)^{\frac{2d}{d-2}}\,\xd x \\
    =&\, \lambda_\omega^{-d}\|\nabla S\|^2_{L^\infty(\xR^d)}\int_{\R^d} \tilde u_\omega^{\frac{2d}{d-2}}\,\xd x\to 0 \text{ as } \omega\to 0.
\end{align*} 

This gives the desired convergence of $\tilde \varphi_\omega$ in $\dot{H}^1(\R^d)$ for any dimension $d\ge 3$. Note that this implies the convergence of $\tilde \varphi_\omega$ to $S$ in $L^{\frac{2d}{d-2}}(\R^d)$ and in $L^{2}_{\mathrm{loc}}(\R^d)$.

Finally, we show that the \( L^2 \)-norm of \( \varphi_\omega \) tends to infinity as \( \omega \to 0 \). Indeed:
\begin{align*}
    \int_{\mathbb{R}^d} \varphi_\omega^2(x) \, \xd x 
    = \int_{\mathbb{R}^d} \lambda_\omega^{2-d} \tilde{\varphi}_\omega^2(\lambda_\omega^{-1} x) \, \xd x 
    &\geq \lambda_\omega^2 \int_{B_R(0)} \tilde{\varphi}_\omega^2(x) \, \xd x \\
    &= \lambda_\omega^2 \left( \int_{B_R(0)} \left( \tilde{\varphi}_\omega^2(x) - S^2(x) \right) \, \xd x + \int_{B_R(0)} S^2(x) \, \xd x \right) \\
    &\geq \lambda_\omega^2 \left( o(1) + \int_{B_R(0)} S^2(x) \, \xd x \right).
\end{align*}
Since $\lambda_\omega \to +\infty$, then  $$\lim\limits_{\omega \to 0} \int_{\mathbb{R}^d} \varphi_\omega^2(x) \, \xd x  = +\infty.  $$
\end{proof}

\addtocontents{toc}{\protect\setcounter{tocdepth}{-1}}
\section*{Tool and computational resource disclosure}
\addtocontents{toc}{\protect\setcounter{tocdepth}{2}}

Large language models where used during the preparation of this manuscript to improve the clarity of the writing and assist in checking the final draft for errors and inconsistencies. The authors reviewed and verified all suggested changes and remain fully responsible for the content of the manuscript.

\end{document}